\documentclass[11pt,reqno]{amsart}

\usepackage[margin=1in]{geometry}
\usepackage{microtype}                     
\usepackage{subcaption}
\usepackage{amsmath,amssymb,amsfonts,amsthm} 
\usepackage{mathrsfs,euscript}             
\usepackage{mathtools}
\usepackage{xcolor}
\usepackage{graphicx}                      

\usepackage{hyperref}
\usepackage[nameinlink]{cleveref}          

\theoremstyle{plain}
\newtheorem{theorem}{Theorem}[section]
\newtheorem{lemma}[theorem]{Lemma}
\newtheorem{proposition}[theorem]{Proposition}
\newtheorem{corollary}[theorem]{Corollary}

\theoremstyle{definition}

\newtheorem{remark}[theorem]{Remark}

\numberwithin{equation}{section}

\DeclareMathOperator{\Id}{Id}

\newcommand{\paren}[1]{\left(#1\right)}
\newcommand{\jump}[1]{\left[\!\left[ #1\right]\!\right]}

\newcommand{\p}{\partial}

\newcommand{\at}[2]{\left. #1 \right|_{#2}}
\newcommand{\grad}[1]{\nabla #1}

\newcommand{\mc}[1]{\mathcal{#1}}
\newcommand{\wh}[1]{\widehat{#1}}
\newcommand{\wt}[1]{\widetilde{#1}}
\newcommand{\bm}[1]{\boldsymbol{#1}}
\newcommand{\abs}[1]{\left\lvert #1 \right\rvert}
\newcommand{\norm}[1]{\left\lVert #1 \right\rVert}

\newcommand{\dual}[2]{\left\langle #1,#2 \right\rangle}

\newcommand{\vph}{\varphi}

\newcommand{\R}{\mathbb{R}}

\newcommand{\thk}{\theta_\kappa}
\newcommand{\tk}{\bm\tau_\kappa}
\newcommand{\nk}{\bm n_\kappa}
\newcommand{\xk}{\bm X_\kappa}

\newcommand{\tka}[1]{\bm\tau_{\kappa_{#1}}}
\newcommand{\thka}[1]{\theta_{\kappa_{#1}}}
\newcommand{\nka}[1]{\bm n_{\kappa_{#1}}}
\newcommand{\xka}[1]{\bm X_{\kappa_{#1}}}

\newcommand{\ci}[1]{C^{#1}(I)}
\newcommand{\hi}[1]{H^{#1}(I)}
\newcommand{\thi}[1]{\wt H^{#1}(I)}
\newcommand{\hr}[1]{H^{#1}(\R)}

\title[Open Inextensible Filaments in Stokes Flow]
{Open Inextensible Filaments in Planar Stokes Flow: Well-Posedness, Endpoint Asymptotics, and Straightening}

\author{Han Zhou}
\address{Department of Mathematics, University of Pennsylvania, Philadelphia, PA 19104, USA}
\email{hzhou24@sas.upenn.edu}


\date{\today}

\begin{document}

\begin{abstract}
We study an inextensible open filament with free ends in a planar Stokes fluid. The system reduces to a third-order nonlocal curvature equation coupled to an elliptic equation for the tension. We prove local well-posedness for nearly critical initial data in supported Sobolev spaces $\widetilde H^s$, $-1/2<s\le0$, satisfying the arc-chord condition. For positive times, we prove improved Sobolev regularity and derive a $d^{3/2}$-type expansion at each free end using Wiener--Hopf factorization, where $d$ denotes the distance to that endpoint. We further prove global existence and exponential convergence to a straight filament for sufficiently small initial data and for finite-energy initial data satisfying $E(0)<\pi^2/4$. The finite-energy result follows from an energy identity and a geometric estimate relating the bending energy to the arc-chord constant. More generally, any finite-time breakdown must be accompanied by loss of the arc-chord condition, while every global solution either converges exponentially to a straight filament or has arc-chord constants tending to zero along a sequence of times tending to infinity.
\end{abstract}

\maketitle

\tableofcontents

\section{Introduction}

We consider an inextensible elastic filament with free ends immersed in a two-dimensional Stokes fluid.
Modeled as a zero-width open curve, the filament exerts forces derived from its bending energy and inextensibility constraint on the surrounding fluid, whose flow in turn transports and deforms the filament.
Such fluid--structure models describe the motion of flexible fibers and threads in viscous fluids, including cilia and flagella~\cite{TornbergShelley2004,EtienneLoheacSaramito2010Inextensible,BouzarthLaytonYoung2011Cellular,LiLayton2012OpenInterface}.

\subsection{Problem formulation}
\label{subsec:prob-formulation}
Let $\Gamma(t)$ denote the open filament at time $t$, and suppose that the exterior domain $\mathbb{R}^2\setminus\overline{\Gamma(t)}$ is occupied by a planar Stokes fluid.
The fluid velocity $\bm u$ and pressure $p$ satisfy the Stokes equations
\begin{equation}\label{eqn:intro-stokes-bulk}
    \Delta\bm u-\grad p=0,
    \quad \grad\cdot\bm u=0
    \quad\text{in }\R^2\setminus\overline{\Gamma(t)}.
\end{equation}
We impose the far-field normalization $\bm u(\bm x,t)\to0$ and
$p(\bm x,t)\to0$ as $\abs{\bm x}\to\infty$. Let
$I=(-1,1)$ and let $\bm X:\overline I\times[0,T]\to\R^2$ be a material
parametrization of the filament
$\Gamma(t)=\bm X(I,t)$. The unit tangent $\bm\tau$ and clockwise unit normal $\bm n$ are given by
\begin{equation}
    \bm\tau = \p_s\bm X, \quad \bm n  = Q_{\frac{\pi}{2}}^{-1}\bm \tau, \quad Q_{\alpha} =
    \begin{bmatrix}
        \cos\alpha & -\sin\alpha\\
        \sin\alpha & \cos\alpha
    \end{bmatrix},
\end{equation}
where $Q_\alpha$ denotes counterclockwise rotation by angle $\alpha$.
We label the side of $\Gamma(t)$ toward which $\bm n$ points by $-$ and the opposite side
by $+$, and write $\jump{\phi}=\phi^+-\phi^-$ for the jump of a quantity $\phi$ across $\Gamma(t)$ in terms of its one-sided traces $\phi^\pm$.
Across the filament, the fluid velocity is continuous and the jump in
hydrodynamic traction balances the elastic force density $\bm F(\bm X)$:
\begin{equation}\label{eqn:intro-stokes-interface}
    \jump{\bm u}=0,
    \quad
    \jump{\Sigma\bm n}
    =\bm F(\bm X)
    \quad\text{on }\Gamma(t),
\end{equation}
where the Newtonian stress tensor $\Sigma$ and symmetric velocity gradient
$\grad_S\bm u$ are defined by
\begin{equation*}
    \Sigma=2\grad_S\bm u-p\Id,
    \quad
    \grad_S\bm u=\frac12\paren{\grad\bm u+\grad\bm u^T}.
\end{equation*}
The filament moves with the local fluid velocity subject to the inextensibility
constraint:
\begin{equation}\label{eqn:intro-kinematic}
    \p_t\bm X(s,t)=\bm u(\bm X(s,t),t),
    \quad \abs{\p_s\bm X(s,t)}=1
    \quad\text{for }s\in\overline I,
\end{equation}
where the velocity at the endpoints is understood via its continuous trace.
The initial configuration is parametrized by arclength:
\begin{equation}\label{eqn:intro-initial}
    \bm X(s,0)=\bm X_0(s),
    \quad \abs{\p_s\bm X_0(s)}=1.
\end{equation}
To derive the elastic force and boundary conditions, we introduce the bending energy
\begin{equation}
    \mc E[\bm X]
    =\frac{1}{2}\int_I |\p_s^2\bm X|^2\,ds.
\end{equation}
The tension $\lambda$ acts as a Lagrange multiplier enforcing the inextensibility constraint, analogous to the pressure $p$ in the Stokes equations.
Considering the augmented functional
\begin{equation}
    \mc A[\bm X,\lambda]
    =\frac{1}{2}\int_I
    \left(|\p_s^2\bm X|^2
    +\lambda\paren{|\p_s\bm X|^2-1}\right)\,ds,
\end{equation}
we compute its first variation in the direction $\bm Y$:
\begin{equation}
\begin{aligned}
    &
    \at{\dfrac{d}{d\varepsilon}
    \mc A[\bm X+\varepsilon\bm Y,\lambda]}{\varepsilon=0}
    = \int_I \paren{\p_s^4\bm X
    -\p_s\paren{\lambda\p_s\bm X}}\cdot\bm Y\,ds \\
    &\quad
    + \at{\paren{\p_s^2\bm X\cdot\p_s\bm Y}}{\p I}
    + \at{\paren{\paren{\lambda\p_s\bm X-\p_s^3\bm X}
    \cdot\bm Y}}{\p I}.
\end{aligned}
\end{equation}
The elastic force density is therefore the negative variational derivative
with respect to $\bm X$,
\begin{equation}\label{eqn:bending-force}
\begin{aligned}
    \bm F(\bm X(s))=\p_s\paren{\lambda\p_s\bm X  -\p_s^3\bm X },
\end{aligned}
\end{equation}
while the boundary terms yield the free-end conditions
\begin{equation}\label{eqn:intro-free-end}
    \p_s^2\bm X=0,
    \quad \lambda\p_s\bm X-\p_s^3\bm X=0
    \quad\text{at }s=\pm1.
\end{equation}
Since $\p_t\abs{\p_s\bm X}^2=2\p_s\bm X\cdot\p_s\bm u$, the arclength
condition $\abs{\p_s\bm X_0}=1$ is preserved provided that
$\grad_\Gamma\cdot\bm u=\bm\tau\cdot\p_s\bm u=0$ along $\Gamma(t)$.
For smooth solutions, testing~\eqref{eqn:intro-stokes-bulk} against $\bm u$
and integrating by parts with the free-end conditions~\eqref{eqn:intro-free-end} yields the formal energy dissipation law
\begin{equation}\label{eqn:energy-law}
    \frac{d}{dt} \mc E[\bm X(t)]
    = - 2\int_{\R^2} \abs{\grad_S \bm u}^2 \,d\bm x.
\end{equation}

Because~\eqref{eqn:intro-kinematic} preserves the arclength parametrization, the curvature $\kappa$ is a natural variable for describing the evolving geometry of the filament.
As shown in Subsection~\ref{subsec:curvature-formulation}, writing the tangent as
$\bm\tau=(\cos\theta,\sin\theta)$ with $\kappa=\p_s\theta$, introducing the
modified tension $\sigma=\lambda+\kappa^2$, and eliminating $(\bm u,p)$ through Stokes single-layer representations reduces the coupled problem to a system for $\kappa$ and $\sigma$ on $I$:
\begin{equation}\label{eqn:intro-schematic-system}
    \p_t\kappa=\Lambda\kappa+\p_s\wt N(\kappa,\sigma),
    \quad
    T_\kappa\sigma=F(\kappa),
\end{equation}
with principal and tension operators
\begin{equation}\label{eqn:intro-principal-operators}
    \Lambda=\frac14\p_s^2\mc H\p_s,
    \quad
    T_\kappa=\frac14\p_s\mc H+R_\kappa.
\end{equation}
Here $\mc H$ denotes the finite Hilbert transform on $I$, $\wt N$ and $F$ collect the lower-order geometric and tension nonlinearities, and $R_\kappa$ is the
remainder operator in the tension equation.
In terms of $(\kappa,\sigma)$, the free-end conditions~\eqref{eqn:intro-free-end} take the homogeneous form
\begin{equation}\label{eqn:cur-bc}
    \kappa=0,\quad \p_s\kappa=0,\quad \sigma=0
    \quad\text{on }\p I.
\end{equation}
Solving the second equation in~\eqref{eqn:intro-schematic-system} for $\sigma=\sigma(\kappa)$ then closes the evolution equation for $\kappa$ with the reduced nonlinearity
$N(\kappa)\coloneqq\wt N(\kappa,\sigma(\kappa))$.

\subsection{Related work}

Much of the rigorous PDE analysis for immersed elastic structures in Stokes flow has focused on
closed interfaces. Closest to the present setting is the case of a closed inextensible filament, for which the tension equation and the resulting evolution problem were analyzed in
\cite{KuoLaiMoriRodenberg2023Tension,
GarciaJuarezKuoMori2025Inextensible}.
On a closed loop, the absence of endpoints allows the principal operators and their inverses to be characterized directly via Fourier series, and a fixed-point argument in the tangent-angle formulation yields local
well-posedness in subcritical H\"{o}lder spaces.

For extensible closed interfaces, the classical Peskin problem models a membrane driven by elastic stretching without bending resistance or an inextensibility constraint. Local well-posedness and nonlinear stability
near circular equilibria were established in
\cite{MoriRodenbergSpirn2019Peskin,LinTong2019StokesIB},
and subsequent works have treated general nonlinear tension laws, rough or critical initial data, viscosity contrast, and three-dimensional membranes
\cite{CameronStrain2024CriticalPeskin,ChenNguyen2023Bdot,
GarciaJuarezHaziot2025GeneralTension,ChenHuNguyen2026Schauder,
GarciaJuarezMoriStrain2023ViscosityContrast,
GarciaJuarezKuoMoriStrain2025Peskin3D,
GarciaJuarezHaziotKuoMoriZhou2026Peskin3DStability}.
In addition, Tong and Wei~\cite{TongWei2024GeometricPeskin} identified a geometric condition on the initial configuration that ensures global existence without requiring small stretching.

Open immersed interfaces also arise in problems where one or both endpoints interact with a boundary. Li and Layton~\cite{LiLayton2012OpenInterface} developed a numerical method for Stokes flow driven by an open interface and applied it to a beating cilium anchored to a rigid wall. On the analytical side, Kandalam and Spirn~\cite{KandalamSpirn2026AnchoredPeskin} proved local well-posedness and instantaneous smoothing for an anchored Peskin problem in a half-plane with both endpoints attached to a no-slip wall. In contrast to wall-mounted configurations, where the boundary introduces additional hydrodynamic interactions, the filament studied here has two free endpoints and evolves in the full plane.

Free-ended filaments and rods have also been extensively investigated in geometric flows and reduced hydrodynamic models, including curve-straightening flows, classical elastohydrodynamics, discrete-link chains, and PDE models of immersed rods
\cite{Oelz2011CurveStraightening,MoriOhm2023Elastohydrodynamics,
AlougesLefebvreLepotLevillainMoreau2025NLink,
AlbrittonOhm2025RodsInFlows},
as well as nonlocal and slender-body theories in three dimensions
\cite{TornbergShelley2004,MoriOhmSpirn2020SBT,
MoriOhmSpirn2020SBTFreeEnds,Ohm2026NonlocalCurve,
Ohm2026FreeBoundary3D,Ohm2025SlenderBodyFreeBoundary}.
Whereas three-dimensional slender-body models regularize the hydrodynamic self-interaction through a small but finite cross-sectional radius, the planar model considered here treats the filament directly as a zero-width open curve in a two-dimensional Stokes fluid.

The central analytical difficulty of the open-filament problem lies in the behavior near the free endpoints, where singular behavior can arise even when the curve and data are smooth. Analogous endpoint and edge singularities are classical in screen scattering problems
\cite{Hayashi1973OpenBoundary,Stephan1987ScreenProblems}
and crack problems in continuum mechanics
\cite{DuduchavaWendland1995CrackProblems,CostabelDaugeDuduchava2003CrackAsymptotics}.
In our setting, the Stokes flow is coupled to the moving filament, so the endpoint behavior of the fluid velocity feeds directly back into the regularity of the curve. Unlike the closed-loop problem~\cite{GarciaJuarezKuoMori2025Inextensible}, parabolic smoothing does not yield $C^\infty$ regularity up to the free ends at positive times. A related phenomenon was analyzed for a linear elliptic bulk--surface system on a fixed open arc in
\cite{EpsteinMoriZhou2026BulkSurface},
where a boundary integral reduction and Wiener--Hopf techniques were used to extract the endpoint asymptotics.

In the curvature formulation~\eqref{eqn:intro-schematic-system}, this endpoint structure is governed by the principal operators~\eqref{eqn:intro-principal-operators}. On supported Sobolev spaces, $\frac{1}{4}\p_s\mc H$ and $\Lambda$ coincide with the restricted fractional Laplacians
$\frac{1}{4}\paren{-\Delta}^{1/2}$ and
$-\frac{1}{4}\paren{-\Delta}^{3/2}$, which act on functions supported in $\overline{I}$ followed by restriction to $I$. Even for smooth right-hand sides, the exterior support condition produces fractional-power profiles at $\p I$ that limit standard Sobolev regularity. Boundary regularity for the fractional Laplacian of order between zero and one was established by Ros-Oton and Serra~\cite{RosOtonSerra2014BoundaryRegularity}, higher-order fractional powers were treated within the pseudodifferential boundary calculus of Grubb~\cite{Grubb2015FractionalLaplacians}, and explicit half-space formulas for higher powers of the Laplacian were obtained in~\cite{AbatangeloEtAl2019HalfSpace}. For the third-order operator $\Lambda$, the leading free-end profile is proportional to $d^{3/2}$, where $d$ denotes the distance to the endpoint; we isolate this singular term in~\eqref{eqn:main-endpoint-split} via Wiener--Hopf factorization following the framework of Eskin~\cite{eskin1981boundary}.

\subsection{Main results}
We work in the supported Sobolev spaces
$\widetilde H^s(I)=\overline{C_c^\infty(I)}^{\,H^s(\R)}$, equipped with the ambient $H^s(\R)$ norm. Given $\kappa$, let $\theta_\kappa$ denote the tangent angle satisfying
$\p_s\theta_\kappa=\kappa$ and $\theta_\kappa(-1)=0$, and define the normalized curve
$\bm X_\kappa(s)=\int_{-1}^s(\cos\theta_\kappa,\sin\theta_\kappa)(\eta)\,d\eta$.
Its arc-chord constant is defined by
\begin{equation}\label{eqn:intro-arc-chord}
    \abs{\kappa}_*
    =\inf_{s,\eta\in I,\ s\ne\eta}
    \frac{\abs{\bm X_\kappa(s)-\bm X_\kappa(\eta)}}{\abs{s-\eta}}.
\end{equation}
We impose the initial arc-chord condition $\abs{\kappa_0}_*>0$, which rules out self-intersection and degeneration of the filament, and denote the spectral upper bound of $\Lambda$ by $\omega_\Lambda
    =\sup\left\{\operatorname{Re}z:z\in\sigma(\Lambda)\right\}<0.$ 
Detailed definitions of the function spaces and operators are provided in
Section~\ref{sec:formulation}.

\begin{theorem}[Local well-posedness]\label{thm:main-local}
Fix $0<\delta\le\frac12$ and let the initial datum
$\kappa_0\in\thi{-\frac12+\delta}$ satisfy $\abs{\kappa_0}_*>0$.
There exist $T=T(\kappa_0)>0$ and a unique mild solution satisfying
\begin{equation}\label{eqn:main-local-class}
    \kappa \in C\paren{[0,T];\thi{-\frac12+\delta}} \cap C\paren{(0,T];\thi{\frac32}}, \quad \sup_{0<t\le
    T}t^{\frac{2-\delta}{3}} \norm{\kappa(t)}_{\hr{\frac32}} <\infty.
\end{equation}
For this $T$, there are $r_0,C>0$, depending only on
$\delta,\kappa_0$, such that each $\vph_0\in\thi{-\frac12+\delta}$ with
$\norm{\vph_0-\kappa_0}_{\hr{-\frac12+\delta}}\le r_0$ has a unique solution
in the class~\eqref{eqn:main-local-class}, with
\begin{equation}\label{eqn:main-local-lipschitz}
\begin{aligned}
    &\sup_{0\le t\le T}
    \norm{\kappa(t;\vph_0)-\kappa(t;\kappa_0)}_{
        \hr{-\frac12+\delta}}\\
    &\quad+
    \sup_{0<t\le T}t^{\frac{2-\delta}{3}}
    \norm{\kappa(t;\vph_0)-\kappa(t;\kappa_0)}_{\hr{\frac32}}
    \le C\norm{\vph_0-\kappa_0}_{\hr{-\frac12+\delta}}.
\end{aligned}
\end{equation}
\end{theorem}

We establish Theorem~\ref{thm:main-local} by inverting the nonlocal tension operator $T_\kappa$ via Fredholm theory and solving the reduced curvature equation through a contraction argument in the time-weighted space~\eqref{eqn:XT-norm-intro}.

At positive times, solutions gain instantaneous Sobolev regularity and admit an explicit endpoint expansion governed by the third-order dissipation.

\begin{theorem}[Positive-time regularity and endpoint asymptotics]
\label{thm:main-regularity}
Let $\kappa$ be the solution of Theorem~\ref{thm:main-local}.
For every $0<\tau<T$ and $0\le\varepsilon<\frac12$,
\begin{equation}\label{eqn:main-positive-time-regularity}
    \kappa\in C\paren{[\tau,T];\thi{\frac32+\varepsilon}}
    \cap C^1\paren{[\tau,T];\hi{-\frac32+\varepsilon}}.
\end{equation}
Moreover, let $\chi_\pm$ be smooth cutoffs with disjoint supports, equal
to one near $s=\pm1$, respectively. For each $t\in(0,T]$, there exist
$c_\pm(t)\in\R$ and
$\kappa_r(t)\in\thi{\frac52+\varepsilon}$ such that $\kappa$ has the form
\begin{equation}\label{eqn:main-endpoint-split}
    \kappa(s,t)
    =c_-(t)\chi_-(s)d_-(s)^{\frac32}
     +c_+(t)\chi_+(s)d_+(s)^{\frac32}
     +\kappa_r(s,t),
\end{equation}
where $d_-(s)=(s+1)_+$ and $d_+(s)=(1-s)_+$, with $r_+=\max\{r,0\}$.
\end{theorem}

Since $d_\pm^{3/2}\in H^{2-\eta}$ for every $\eta>0$ but $d_\pm^{3/2}\notin H^2$, a nonzero coefficient $c_\pm(t)$ is the exact obstruction to $H^2$ regularity at the corresponding endpoint.
This positive-time regularity enables the reconstruction of the physical filament curve and fluid fields, which satisfy the original Stokes--filament system classically for $t>0$.

\begin{corollary}[Reconstruction of a classical solution]\label{cor:main-reconstruction}
Let $\kappa$ be as in Theorem~\ref{thm:main-regularity}.
Given an initial tangent angle $\theta_{*,0}\in\R$ and
centroid $\overline{\bm X}_0\in\R^2$, $\kappa$ uniquely
determines a reconstruction $(\bm X,\bm u,p,\lambda)$ solving
the Stokes--filament system
\eqref{eqn:intro-stokes-bulk}--\eqref{eqn:intro-kinematic} and \eqref{eqn:intro-free-end}
classically for $t\in(0,T]$. The initial filament $\bm X_0$ is determined
by $\kappa_0$, $\theta_{*,0}$, and $\overline{\bm X}_0$.
The fluid fields satisfy the prescribed far-field normalization.
The bulk equations hold pointwise off $\overline{\Gamma(t)}$, the
interface conditions hold pointwise along the filament
interior, and the free-end conditions hold at $s=\pm1$.
The kinematic equation and inextensibility constraint hold
throughout $\overline I$, with the velocity at the endpoints
understood through its continuous trace.
\end{corollary}

For any $0<a<b$ in the interval of existence, the reconstructed solution satisfies the energy identity
\begin{equation}\label{eqn:intro-energy-identity}
    \mc E(b)
    +2\int_a^b\int_{\R^2}
        \abs{\grad_S\bm u(\bm x,t)}^2\,d\bm x\,dt
    =\mc E(a),
    \quad \mc E(t)=\frac12\norm{\kappa(t)}_{L^2(I)}^2.
\end{equation}
When $\kappa_0\in L^2(I)$, \eqref{eqn:intro-energy-identity} extends to
$a=0$ with $\mc E(0)=\frac12\norm{\kappa_0}_{L^2(I)}^2$.
Since $L^2(I)$ is strictly subcritical for~\eqref{eqn:intro-schematic-system}, the dissipation identity~\eqref{eqn:intro-energy-identity} provides the a priori control required to continue solutions, yielding global existence for small data, an energy gap dichotomy, and an explicit threshold for exponential straightening.

\begin{theorem}[Global well-posedness and long-time behavior]
\label{thm:main-global}
Fix $0<\delta\le\frac12$ and let
$\kappa_0\in\thi{-\frac12+\delta}$ satisfy $\abs{\kappa_0}_*>0$.
Let $\kappa$ be the unique maximal continuation of the solution in
Theorem~\ref{thm:main-local}, with lifetime $\tau_{\mathrm{max}}$.
\begin{enumerate}
    \item \emph{Nearly critical small data.}
    For each $\omega\in(\omega_\Lambda,0)$, there are $M_0,C>0$, depending
    only on $\delta,\omega$, such that, if
    $\norm{\kappa_0}_{\hr{-\frac12+\delta}}\le M_0$, then the solution
    is global and
    \begin{equation}\label{eqn:main-global-decay}
        \norm{\kappa(t)}_{\hr{-\frac12+\delta}} +\min\left\{t^{\frac{2-\delta}{3}},1\right\}
        \norm{\kappa(t)}_{\hr{\frac32}} \le Ce^{t\omega} \norm{\kappa_0}_{\hr{-\frac12+\delta}}, \quad t>0.
    \end{equation}
    The same smallness condition implies $\abs{\kappa_0}_*>0$.

    \item \emph{Continuation and long-time alternatives.}
    Let $\mc E_{\mathrm{ac}}=\frac{\pi^2}{4}$. The limiting energy exists and satisfies
    \begin{equation}\label{eqn:main-energy-gap}
        \mc E_*:=\lim_{t\uparrow\tau_{\mathrm{max}}}\mc E(t)
        \in\{0\}\cup[\mc E_{\mathrm{ac}},\infty),
    \end{equation}
    where the limit is taken as $t\to\infty$ when
    $\tau_{\mathrm{max}}=\infty$.
    If $\tau_{\mathrm{max}}<\infty$, then
    $\mc E_*\ge\mc E_{\mathrm{ac}}$ and
    \begin{equation*}
        \lim_{t\uparrow\tau_{\mathrm{max}}}\abs{\kappa(t)}_*=0.
    \end{equation*}
    In particular, $\mc E_*=0$ implies global existence.

    Every global solution satisfies exactly one of the following alternatives:
    \begin{enumerate}
        \item If $\mc E_*=0$, then $\kappa(t)\to0$ in $L^2(I)$ and
        $\abs{\kappa(t)}_*\to1$. For each
        $\omega\in(\omega_\Lambda,0)$, there is $T_\omega>0$ such that
        \begin{equation*}
            \norm{\kappa(T_\omega+t)}_{L^2(I)}
            +\min\{t^{1/2},1\}
            \norm{\kappa(T_\omega+t)}_{\hr{\frac32}}
            \le Ce^{t\omega}\norm{\kappa(T_\omega)}_{L^2(I)},
            \quad t>0,
        \end{equation*}
        where $C$ depends only on $\omega$. In particular, the filament
        straightens exponentially modulo rigid motions.

        \item If $\mc E_*>0$, then $\mc E_*\ge\mc E_{\mathrm{ac}}$ and
        \begin{equation*}
            \liminf_{t\to\infty}\abs{\kappa(t)}_*=0.
        \end{equation*}
    \end{enumerate}

    \item \emph{Finite-energy threshold.}
    If $\kappa(t_0)\in L^2(I)$ and
    $\norm{\kappa(t_0)}_{L^2(I)}<\pi/\sqrt2$ for some
    $t_0\in[0,\tau_{\mathrm{max}})$, then the solution is global,
    $\mc E_*=0$, and
    \begin{equation}\label{eqn:main-post-threshold-arc-chord}
        \abs{\kappa(t)}_*
        \ge\cos\sqrt{\mc E(t_0)}>0,
        \quad t\ge t_0.
    \end{equation}
    Thus the exponential straightening in part~\emph{(2)} applies.
    In particular, every $\kappa_0\in L^2(I)$ with
    $\norm{\kappa_0}_{L^2(I)}<\pi/\sqrt2$ generates such a global solution.
\end{enumerate}
\end{theorem}

For small data, the exponential semigroup bound~\eqref{eqn:lambda-semigroup-est} combined with quadratic nonlinear estimates yields a uniform contraction in the exponentially weighted space~\eqref{eqn:Y-intro}.
For general data, the turning bound~\eqref{eqn:turning-chord-bound} guarantees preservation of the arc-chord condition below $\mc E_{\mathrm{ac}}=\pi^2/4$. Compactness of time translates and the rigidity of zero-dissipation states then force any non-degenerate limiting solution to be straight, ruling out intermediate asymptotic energies and triggering exponential relaxation once the solution enters the small-data regime.

In physical space, exponential decay of the curvature translates directly into convergence to a straight segment:
\begin{corollary}[Convergence to a straight filament]
\label{cor:main-straightening}
Let $\bm X$ and $\theta$ be reconstructed from a solution in
Theorem~\ref{thm:main-global} with $\mc E_*=0$.
There exist $\theta_\infty\in\R$ and
$\overline{\bm X}_\infty\in\R^2$ such that for every
$\omega\in(\omega_\Lambda,0)$ there is $C_\omega>0$ satisfying
\begin{equation*}
    |\bar\theta(t)-\theta_\infty|
    +|\overline{\bm X}(t)-\overline{\bm X}_\infty|
    +\norm{\bm X(\cdot,t)-\bm X_\infty}_{C^1(\overline I)}
    \le C_\omega e^{\omega t},\quad t\ge0.
\end{equation*}
Here $\bm X_\infty(s)=\overline{\bm X}_\infty
+s(\cos\theta_\infty,\sin\theta_\infty)$.
\end{corollary}

\begin{remark}
The rescaling $\kappa_\rho(s,t)=\rho\kappa(\rho s,\rho^3t)$ leaves the homogeneous $\dot H^{-1/2}$ norm invariant when the interval is rescaled accordingly. The local well-posedness and small-data results are nearly critical, since they hold at subcritical regularities arbitrarily close to the scaling-critical index.
\end{remark}

\subsection{Organization of the paper}
The remainder of the paper is organized as follows.
Section~\ref{sec:formulation} develops the curvature--tension formulation, the scale of supported Sobolev spaces, semigroup estimates, and auxiliary commutator and boundary integral bounds.
Section~\ref{sec:tension} analyzes the nonlocal tension equation and establishes the uniform invertibility of $T_\kappa$.
Section~\ref{sec:local-theory} proves local well-posedness (Theorem~\ref{thm:main-local}).
Section~\ref{sec:regularity-theory} establishes positive-time smoothing, the Wiener--Hopf endpoint expansion, and the classical reconstruction (Theorem~\ref{thm:main-regularity} and Corollary~\ref{cor:main-reconstruction}).
Section~\ref{sec:global-theory} proves global existence, the energy gap, and exponential straightening (Theorem~\ref{thm:main-global} and Corollary~\ref{cor:main-straightening}).
Appendix~\ref{app:operator-estimates} collects the proofs of the auxiliary geometric, commutator, and boundary integral estimates.

\section{Preliminaries}
\label{sec:formulation}
\label{sec:prelim}

This section establishes the analytical framework for the curvature--tension system, developing the intrinsic formulation of the interface dynamics, the functional setting of supported Sobolev spaces, semigroup estimates for the dissipative operator $\Lambda$, and the requisite commutator and Wiener--Hopf bounds.

\subsection{Curvature formulation}
\label{subsec:curvature-formulation}
We reduce the fluid--structure system~\eqref{eqn:intro-stokes-bulk}--\eqref{eqn:intro-initial} with free-end conditions~\eqref{eqn:intro-free-end} to an intrinsic curvature--tension system, following geometric reductions for viscous curve dynamics~\cite{Oelz2011CurveStraightening,MoriOhm2023Elastohydrodynamics,Ohm2026NonlocalCurve}.
Recall that $G = (G_{ij})_{i,j=1}^2$ denotes the free-space Stokeslet on $\R^2$:
\begin{equation}\label{eqn:stokeslet-definition}
    G_{ij}(\bm x) = \frac{1}{4\pi}\paren{-\log|\bm x|\delta_{ij} + \frac{x_i x_j}{|\bm x|^2}}, \quad i,j = 1,2.
\end{equation}
With the shorthand difference notation $\Delta\bm X=\bm X(s,t)-\bm X(\eta,t)$, the boundary integral representation of the Stokes velocity along the interface yields the closed evolution equation for the filament:
\begin{equation}\label{eqn:filament-boundary-evolution}
    \p_t\bm X=\int_I G(\Delta\bm X)\p_\eta\paren{\lambda\p_\eta\bm X  -\p_\eta^3\bm X}\,d\eta,\quad \abs{\p_s\bm X} \equiv 1,
\end{equation}
supplemented by the initial condition $\bm X(\cdot,0)=\bm X_0(\cdot)$ and the free-end boundary conditions~\eqref{eqn:intro-free-end}.
To reformulate the internal elastic and tension forces in intrinsic geometric variables, we express the unit tangent vector in polar coordinates as
\begin{equation}
    \bm\tau(s)=\bm e_r(\theta(s))
    =\paren{\cos\theta(s),\sin\theta(s)}^T,
\end{equation}
where $\theta$ is the tangent angle. With the scalar curvature defined by $\kappa=\p_s\theta$ and the modified tension by $\sigma=\kappa^2+\lambda$, the Frenet--Serret formulas give
\begin{equation}
    \lambda\p_s\bm X-\p_s^3\bm X
    =
    \paren{\lambda+\kappa^2}\bm\tau+\p_s\kappa\,\bm n
    =
    \sigma\bm\tau+\p_s\kappa\,\bm n .
\end{equation}
Because the parametrization satisfies $\abs{\p_s\bm X}\equiv 1$ and $\p_s^2\bm X=-\kappa\bm n$, the total elastic bending energy reduces directly to
\begin{equation}\label{eqn:curvature-energy}
    \mc E[\bm X]
    =\frac12\norm{\kappa}_{L^2(I)}^2
    =\frac12\norm{\kappa}_{L^2(\R)}^2,
\end{equation}
where the $L^2(\R)$ norm is computed after extending $\kappa$ by zero outside $I$.
For a time-dependent reconstructed curve, we write $\mc E(t)\coloneqq\mc E[\bm X(t)]$.

The scalar curvature $\kappa$ determines the planar curve uniquely up to a rigid motion (translation and global rotation).
To factor out the overall orientation, we introduce the normalized angle and curve anchored at the left endpoint $s=-1$:
\begin{equation}\label{eqn:normalized-geometry}
\begin{aligned}
    \thk(s)&=J\kappa(s)\coloneqq\int_{-1}^s\kappa(\eta)\,d\eta,
    &\tk(s)&=\paren{\cos\thk(s),\sin\thk(s)}^T,\\
    \nk(s)&=Q_{\frac{\pi}{2}}^{-1}\tk(s),
    &\xk(s)&=\int_{-1}^s\tk(\eta)\,d\eta.
\end{aligned}
\end{equation}
Since $\p_s\theta=\kappa$, the physical tangent angle decomposes as $\theta(s)=\theta_*+\thk(s)$, where $\theta_*=\theta(-1)$ represents the instantaneous orientation angle at the left end.
Rotation by the orthogonal matrix $Q_{\theta_*}$ relates the physical and normalized geometries through
\begin{equation}
    \Delta\bm X=Q_{\theta_*}\Delta\xk,
    \quad
    \bm\tau=Q_{\theta_*}\tk,
    \quad
    \bm n=Q_{\theta_*}\nk,
\end{equation}
where here and throughout we write $\Delta\xk=\xk(s)-\xk(\eta)$.

Projecting the velocity equation~\eqref{eqn:filament-boundary-evolution} onto the normal and tangential directions yields the coupled system for the tangent angle $\theta$ and tension $\sigma$:
\begin{subequations}\label{eqn:theta-eqn}
\begin{align}
    \p_t\theta &= - \bm n(s)\cdot\p_s\int_I G(\Delta\bm X)\p_\eta\paren{ \p_\eta^2\theta\bm n+\sigma\bm \tau } \,d\eta ,\\
     0&= \bm\tau(s) \cdot\p_s\int_I G(\Delta\bm X)\p_\eta\paren{ \p_\eta^2\theta \bm n + \sigma\bm \tau } \,d\eta.
\end{align}
\end{subequations}
Differentiating the evolution equation for $\theta$ with respect to arclength $s$ and applying the rotational covariance of the Stokeslet~\eqref{eqn:stokeslet-definition}, $G(Q\bm x)=QG(\bm x)Q^T$ for any rotation matrix $Q$, we observe that the overall rotation $Q_{\theta_*}$ drops out completely.
This produces the autonomous curvature--tension system
\begin{subequations}\label{eqn:kappa-eqn}
\begin{align}
    \p_t\kappa
    &= -\p_s\paren{\nk(s)\cdot\p_s\int_I
    G(\Delta\xk)\p_\eta\paren{
    \p_\eta\kappa\nk+\sigma\tk}\,d\eta},\\
    0
    &= \tk(s)\cdot\p_s\int_I
    G(\Delta\xk)\p_\eta\paren{
    \p_\eta\kappa\nk+\sigma\tk}\,d\eta,
\end{align}
\end{subequations}
subject to the endpoint conditions~\eqref{eqn:cur-bc}.
In particular, the system~\eqref{eqn:kappa-eqn} is entirely independent of the rigid orientation $\theta_*$.

Once the curvature--tension system~\eqref{eqn:kappa-eqn} is solved for $(\kappa,\sigma)$, the physical filament position $\bm X(s,t)$ is fully reconstructed by integrating ordinary differential equations for the rigid degrees of freedom.
Specifically, the spatial mean tangent angle $\bar\theta=\frac12\int_I\theta(\eta)\,d\eta$ and the filament centroid $\overline{\bm X}=\frac12\int_I\bm X(\eta)\,d\eta$ evolve according to
\begin{equation}\label{eqn:theta-bar}
    \p_t\bar\theta
    = -\frac{1}{2}\int_I \bm n(s)\cdot\p_s\int_I
    G(\Delta\bm X)\p_\eta\paren{ \p_\eta\kappa\bm n+\sigma\bm \tau }
    \,d\eta \,ds,
\end{equation}
and
\begin{equation}\label{eqn:X-bar}
    \p_t \overline{\bm X}
    = \frac{1}{2}\int_I \int_I
    G(\Delta\bm X)\p_\eta\paren{ \p_\eta\kappa\bm n+\sigma\bm \tau }
    \,d\eta ds.
\end{equation}
Prescribing the initial orientation angle $\theta_{*,0}$ at the left endpoint and the initial centroid $\overline{\bm X}_0 = \frac12\int_I\bm X_0(\eta)\,d\eta$ determines the initial data
\begin{equation}\label{eqn:reconstruction-initial-means}
    \bar\theta(0)
    =\theta_{*,0}
    +\frac12\int_I(1-\eta)\kappa_0(\eta)\,d\eta,
    \quad
    \overline{\bm X}(0)=\overline{\bm X}_0.
\end{equation}
At any time $t>0$, the left-endpoint tangent angle is recovered from the mean angle by
\begin{equation}\label{eqn:reconstruction-endpoint-angle}
    \theta_*=\bar\theta-\frac12\int_I(1-\eta)\kappa(\eta)\,d\eta.
\end{equation}
With $\theta(\eta) = \theta_* + \thk(\eta)$ in hand, the physical curve is explicitly reconstructed via
\begin{equation}\label{eqn:curve-reconstruction}
    \bm X(s)
    =\overline{\bm X}
    -\frac12\int_I(1-\eta)\bm e_r(\theta(\eta))\,d\eta
    +\int_{-1}^s\bm e_r(\theta(\eta))\,d\eta.
\end{equation}

Returning to the curvature--tension system~\eqref{eqn:kappa-eqn}, we isolate the leading singular interaction corresponding to a straight reference filament from the smoother geometric remainder terms that are analyzed in Subsection~\ref{sec:divided-diff}.
This leading part generates the order-one operator $\frac14\p_s\mc H$ and the dissipative order-three operator $\Lambda=\frac14\p_s^2\mc H\p_s$.
To extract this structure, we define the regularized matrix kernel
\begin{equation}\label{eqn:R-kernel}
    r(s,\eta;\kappa) = \log\abs{\frac{\Delta\xk}{s-\eta} } \Id - \frac{ \Delta\xk \otimes \Delta\xk }{|\Delta\xk |^2}.
\end{equation}
Componentwise, the hydrodynamic Stokeslet along the curve admits the decomposition
\begin{equation}\label{eqn:stokeslet-splitting}
    G_{ij}(\Delta\xk)
    =-\frac{1}{4\pi}\log\abs{s-\eta}\,\delta_{ij}
    -\frac{1}{4\pi}r_{ij}(s,\eta;\kappa).
\end{equation}
The logarithmic singularity generates the finite Hilbert transform on $I$.
Letting $\mc H_{\R}$ denote the standard Hilbert transform on $\R$, we define its restriction $\mc H=\mc H_I$ to functions supported in $I$ by $\mc H_I f=(\mc H_{\R}f)|_I$, which for smooth test functions $f\in C_c^\infty(I)$ takes the principal value form
\begin{equation}\label{eqn:finite-hilbert-definition}
    \mc Hf(s) = \frac{1}{\pi}\operatorname{p.v.} \int_I \frac{f(\eta)}{s-\eta} \,d\eta.
\end{equation}
Throughout the paper, the commutator of two operators $A$ and $B$ is denoted by
\begin{equation}
    [A,B] f \coloneqq A(Bf)-B(Af).
\end{equation}
For vector-valued profiles $\bm b=(b_1,b_2)$ and $\bm q=(q_1,q_2)$, we adopt the standard componentwise commutator and scalar contraction notations:
\begin{equation*}
    [A,\bm b]g\coloneqq([A,b_j]g)_{j=1}^2,
    \quad
    [A,\bm b]\cdot\bm q\coloneqq\sum_{j=1}^2[A,b_j]q_j.
\end{equation*}

Under the decomposition~\eqref{eqn:stokeslet-splitting}, the inextensibility constraint in~\eqref{eqn:kappa-eqn} is equivalent to the operator equation
\begin{equation}\label{eqn:ten-det}
    T_\kappa\sigma=F(\kappa),
\end{equation}
where $T_\kappa = \frac14\p_s\mc H + R_\kappa$ is the tension operator introduced in~\eqref{eqn:intro-principal-operators}, and $F(\kappa)$ represents the hydrodynamic forcing generated by bending.
The regular operator $R_\kappa$ and the forcing $F(\kappa)$ are explicitly given by
\begin{equation}
\begin{aligned}
    R_\kappa\sigma
    &= \frac{1}{4}\tk \cdot [\p_s \mc H, \tk]\sigma
    - \frac{1}{4\pi} \tk\cdot\p_s\int_I
    \p_\eta r(s,\eta;\kappa)\sigma(\eta)\tk(\eta)\,d\eta,
\end{aligned}
\label{eqn:Rop-def}
\end{equation}
and
\begin{equation}\label{eqn:F-kappa}
\begin{aligned}
    F(\kappa)
    &= \frac{1}{4}[\p_s\mc H, \tk]\cdot(\p_s\kappa\nk)
    + \frac{1}{4\pi}\tk\cdot\p_s\int_I
    \p_\eta r(s,\eta;\kappa)\p_\eta\kappa\nk(\eta)\,d\eta.
\end{aligned}
\end{equation}

Similarly, prior to eliminating the tension variable, the curvature nonlinearity appearing in~\eqref{eqn:intro-schematic-system} takes the form
\begin{equation}
\begin{aligned}
    \wt N(\kappa,\sigma)
    &= -\frac14[\p_s\mc H,\nk]\cdot
    \paren{\p_s\kappa\nk+\sigma\tk}\\
    &\quad -\frac{1}{4\pi}\nk(s)\cdot\p_s\int_I
    \p_\eta r(s,\eta;\kappa)
    \paren{\p_\eta\kappa\nk+\sigma\tk}\,d\eta.
\end{aligned}
\label{eqn:N-def}
\end{equation}
In Section~\ref{sec:tension}, we prove that the tension constraint~\eqref{eqn:ten-det} can be uniquely solved for $\sigma=\sigma(\kappa)$ on the appropriate scale of supported Sobolev spaces.
Substituting this into $\wt N$ defines the curvature nonlinearity $N(\kappa)\coloneqq\wt N(\kappa,\sigma(\kappa))$ and reduces the coupled system~\eqref{eqn:kappa-eqn} to a single closed, nonlocal parabolic evolution equation:
\begin{equation}\label{eqn:kap_evo}
    \p_t\kappa=\Lambda\kappa+\p_s\paren{N(\kappa)}.
\end{equation}

For any initial datum $\kappa_0$, we define the corresponding Duhamel mapping by
\begin{equation}\label{eqn:mild-def}
    (\Psi_{\kappa_0}\kappa)(t)
    \coloneqq
    e^{t\Lambda}\kappa_0
    +\int_0^t e^{(t-\tau)\Lambda}
    \p_s\paren{N(\kappa(\tau))}\,d\tau.
\end{equation}
The precise solution class and the rigorous definition of a mild solution in the weighted trajectory space $X_T^{(\delta)}$ are formulated in Subsection~\ref{subsec:local-proof}.

\subsection{Supported Sobolev spaces and geometric estimates}
\label{subsec:notation}

\emph{Supported spaces and endpoint conditions.}
The Fourier transform on $\R$ is normalized by
\begin{equation}
    \wh f(\xi) = \frac{1}{\sqrt{2\pi}}\int_{\R} e^{-is\xi}f(s)\,ds.
\end{equation}
Writing $\langle\xi\rangle=(1+\abs{\xi}^2)^{1/2}$ for the Japanese bracket, for any $s\in\R$ the $L^2$-based Sobolev space $H^s(\R)$ consists of all tempered distributions $f\in\mathcal S'(\R)$ such that $\langle\xi\rangle^s\wh f(\xi)\in L^2(\R)$, equipped with the inner product norm
\begin{equation}
    \norm{f}_{H^s(\R)}^2
    =
    \int_{\R}\langle\xi\rangle^{2s}\abs{\wh f(\xi)}^2\,d\xi .
\end{equation}

For an open interval $I\subset\R$ (whether bounded or unbounded), let $p_I$ denote the standard restriction operator on distributions:
\begin{equation}
    p_I:\mathcal D'(\R)\longrightarrow\mathcal D'(I),
    \quad p_IU\coloneqq U|_I.
\end{equation}
The restriction Sobolev space $H^s(I)$ is defined by
\begin{equation}
    H^s(I)
    \coloneqq
    \left\{u\in\mathcal D'(I):
    u=p_IU\text{ for some }U\in H^s(\R)\right\},
\end{equation}
endowed with the quotient norm
\begin{equation}
    \norm{u}_{H^s(I)}
    \coloneqq
    \inf\left\{\norm{U}_{H^s(\R)}:
    U\in H^s(\R),\ p_IU=u\right\}.
\end{equation}
The restriction map $p_I$ is continuous, and every distribution $u\in H^s(I)$ admits an extension $U\in H^s(\R)$ satisfying $\norm{U}_{H^s(\R)}\le C\norm{u}_{H^s(I)}$.

In contrast, the supported Sobolev space $\wt H^s(I)$ is defined as the closure of smooth compactly supported test functions in the whole-line topology:
\begin{equation}\label{eqn:supported-space-definition}
    \wt H^s(I)
    \coloneqq
    \overline{C_c^\infty(I)}^{\,H^s(\R)},
\end{equation}
equipped with the inherited norm from $H^s(\R)$.
Under this canonical identification, extension by zero outside $I$ defines an isometric embedding of $\wt H^s(I)$ into $H^s(\R)$.
We also consider the subspace of Sobolev distributions with support in the closure $\overline I$:
\begin{equation}
    H^s_{\overline I}(\R)
    \coloneqq
    \left\{u\in H^s(\R):\operatorname{supp}u\subset\overline I\right\}.
\end{equation}
These definitions generalize to arbitrary open domains $\Omega \subset \R^n$, for which $\wt H^s(\Omega)\subset H^s_{\overline \Omega}(\R^n)$, with equality holding whenever $\Omega$ has a $C^0$ boundary~\cite[Chapter~3]{McLean2000}.
In particular, for any open interval $I\subset\R$, the two notions coincide across all real indices:
\begin{equation}
    \wt H^s(I)=H^s_{\overline I}(\R),\quad s\in\R.
\end{equation}
With $L^2(I)$ serving as the pivot space, we have the canonical duality identifications
\begin{equation}
    \paren{\wt H^s(I)}'=H^{-s}(I),
    \quad
    \paren{H^s(I)}'=\wt H^{-s}(I),
\end{equation}
where the pairing $\dual{\cdot}{\cdot}$ extends the standard $L^2(I)$ inner product.

For non-integer indices $0<s<1$, the intrinsic Slobodeckij seminorm on $I$ is given by
\begin{equation}
    [u]_{H^s(I)}^2
    \coloneqq
    \int_I\int_I
    \frac{\abs{u(x)-u(y)}^2}{\abs{x-y}^{1+2s}}\,dx\,dy,
\end{equation}
and is defined analogously on $\R$. For $s=1$, we set $[u]_{H^1(I)}=\norm{\p_su}_{L^2(I)}$.
The combined quantity $\paren{\norm{u}_{L^2(I)}^2+[u]_{H^s(I)}^2}^{1/2}$ defines an equivalent norm on $H^s(I)$~\cite[Chapter~3]{McLean2000}.

Crucially, the supported space $\wt H^s(I)$ must be distinguished from the classical boundary-value space $H_0^s(I)$ at half-integer indices.
For $s\ge0$, set $H_0^s(I)=\overline{C_c^\infty(I)}^{\,H^s(I)} \supset \wt H^s(I)$.
Whenever $s\notin\{\frac12,\frac32,\frac52,\ldots\}$, one has $\wt H^s(I) = H_0^s(I)$.
For low regularity $0\le s<\frac12$, these spaces moreover coincide with the full restriction space $H^s(I)$; in particular, $\wt L^2(I) = L^2(I)$.
For $s>\frac12$, the space $\wt H^s(I)$ is strictly smaller than $H^s(I)$, because Sobolev embedding ensures that elements of $\wt H^s(I)$ possess vanishing boundary traces, a condition not imposed on general elements of $H^s(I)$.
At the critical half-integer indices, $\wt H^s(I)$ enforces an additional boundary integrability condition that prevents logarithmic singularity blowup near the endpoints.
Specifically, $\wt H^{\frac12}(I)$ coincides with the Lions--Magenes space $H_{00}^{\frac12}(I)$, whose vanishing boundary behavior cannot be captured in the ordinary trace sense.
Likewise, a distribution $u\in\wt H^{\frac32}(I)$ satisfies $\p_su\in\wt H^{\frac12}(I)$, even though an ordinary pointwise endpoint trace of $\p_su$ need not exist; see~\cite[Chapter~3]{McLean2000} and~\cite[Chapter~33]{Tartar2007}.
These supported conditions provide the precise functional framework to encode the free-end boundary conditions before positive-time parabolic smoothing yields classical endpoint traces.

For an integer $m\ge0$ and exponent $0<\alpha<1$, the H\"older space $\ci{m,\alpha}$ comprises all functions $f\in\ci{m}$ with finite norm
\begin{align*}
    \norm{f}_{\ci{m,\alpha}}
    \coloneqq\sum_{j=0}^{m}\norm{\p_s^j f}_{L^\infty(I)}
        +[\p_s^m f]_{\ci{0,\alpha}}, \quad 
    [f]_{\ci{0,\alpha}}
    \coloneqq\sup_{x\ne y\in I}
        \frac{\abs{f(x)-f(y)}}{\abs{x-y}^{\alpha}}.
\end{align*}
By continuous extension to the boundary $\partial I$, we identify $\ci{m,\alpha}$ with $C^{m,\alpha}(\overline I)$.
For vector-valued profiles, absolute values denote the standard Euclidean norm.
Throughout the remainder of the paper, we fix the reference interval $I=(-1,1)$ and collect fundamental operational estimates on these spaces.


\emph{Interpolation and multiplication.}
We first record the action of distributional differentiation on the scale of supported Sobolev spaces.

\begin{lemma}\label{lem:wt-der}
For every $s\in\R$, differentiation defines a bounded map $\p_s:\wt H^s(I)\to\wt H^{s-1}(I).$
\end{lemma}
\begin{proof}
If $f\in\wt H^s(I)$, then its distributional derivative $\p_s f$ remains supported in the compact interval $\overline I$.
The whole-line Fourier multiplier bound
\begin{equation}
    \norm{\p_s f}_{H^{s-1}(\R)} \le \norm{f}_{H^s(\R)}
\end{equation}
therefore implies that $\p_s f\in\wt H^{s-1}(I)$, proving boundedness.
\end{proof}

We denote real interpolation by the standard functor $(Y_0,Y_1)_{\theta,2}$.
The following interpolation identity characterizes the intermediate spaces that arise in our semigroup generation and trace estimates.

\begin{lemma}\label{lem:hs-intp}
Let $s>0$. With equivalent norms,
\begin{equation}\label{eqn:hs-intp}
    \paren{\hi{-s},\thi{s}}_{\theta,2}
    =
    \begin{cases}
        \thi{s(2\theta-1)}, & \frac12<\theta<1,\\
        L^2(I), & \theta=\frac12,\\
        \hi{-s(1-2\theta)}, & 0<\theta<\frac12.
    \end{cases}
\end{equation}
The supported spaces carry the inherited norms from~\eqref{eqn:supported-space-definition}, including at half-integer indices.
\end{lemma}
\begin{proof}
The canonical duality and continuous dense embeddings $\thi{s}\hookrightarrow L^2(I)\hookrightarrow\hi{-s}$ imply, via \cite[Lemma~B.10]{McLean2000}, that
\begin{equation*}
    \paren{\hi{-s},\thi{s}}_{\frac12,2}=L^2(I).
\end{equation*}
The classical reiteration theorem \cite[Theorem~B.6]{McLean2000} then reduces the remaining parameter ranges to interpolation between $H^{-s}(I)$ and $L^2(I)$ with parameter $2\theta$, or between $L^2(I)$ and $\wt H^s(I)$ with parameter $2\theta-1$.
Applying the restriction- and supported-space interpolation theorems \cite[Theorems~B.8 and~B.9]{McLean2000} yields~\eqref{eqn:hs-intp}.
\end{proof}

The following bilinear multiplication estimates will be used extensively to bound composition operators for the reconstructed tangent and to control nonlinear commutator remainders.

\begin{lemma}\label{point_mult_sobo}
Let $s,s_1,s_2\in\R$ satisfy one of the following conditions:
\begin{enumerate}
    \item $s\ge0$, $s_i\ge s$ for $i=1,2$, and
    $s_1+s_2-s>\frac12$;
    \item $s<0$, $s_i\ge s$ for $i=1,2$, $\min\{s_1,s_2\}<0$,
    $s_1+s_2\ge0$, and $s_1+s_2-s>\frac12$;
    \item $s<0$, $s_i\ge s$ for $i=1,2$, $\min\{s_1,s_2\}\ge0$,
    $s_1+s_2>0$, and $s_1+s_2-s>\frac12$.
\end{enumerate}
Then multiplication extends to bounded bilinear maps
\begin{align*}
    \hi{s_1}\times\thi{s_2}&\longrightarrow\thi{s},\\
    \hi{s_1}\times\hi{s_2}&\longrightarrow\hi{s},
\end{align*}
with the estimates
\begin{align}
    \norm{fg}_{\hr{s}}
    &\le C\norm{f}_{\hi{s_1}}\norm{g}_{\hr{s_2}},
    &&f\in\hi{s_1},\quad g\in\thi{s_2},\label{eqn:product-supported}\\
    \norm{fg}_{\hi{s}}
    &\le C\norm{f}_{\hi{s_1}}\norm{g}_{\hi{s_2}},
    &&f\in\hi{s_1},\quad g\in\hi{s_2},\label{eqn:product-restriction}
\end{align}
where $C=C(s,s_1,s_2)$.
\end{lemma}
\begin{proof}
Under each of the stated conditions, the whole-line multiplication theorem \cite[Theorems~7.3, 8.1, and~8.2]{Behzadan2021} guarantees that
\begin{equation*}
    \norm{FG}_{\hr{s}}
    \le C\norm{F}_{\hr{s_1}}\norm{G}_{\hr{s_2}}.
\end{equation*}
Let $F\in\hr{s_1}$ be an arbitrary extension of $f\in\hi{s_1}$.
For $g\in\thi{s_2}$, choose an approximating sequence $g_n\in C_c^\infty(I)$ converging to $g$ in $\hr{s_2}$.
Each product $Fg_n$ belongs to $\thi{s}$ and depends only on the restriction $f = F|_I$.
The whole-line estimate shows that $Fg_n\to Fg$ in $\hr{s}$.
Because $\thi{s}$ is closed in $\hr{s}$, this uniquely defines the product $fg \coloneqq Fg \in\thi{s}$ independently of the extension $F$.
Taking the infimum over all extensions $F$ of $f$ establishes~\eqref{eqn:product-supported}.

For the restriction estimate, given $g\in\hi{s_2}$, let $G\in\hr{s_2}$ be an extension of $g$ and define $fg=p_I(FG)$.
Because pointwise multiplication on $\R$ is a local operation, the restricted distribution $p_I(FG)$ is independent of the choice of extensions $F$ and $G$.
The continuity of the restriction operator $p_I$ and the whole-line product estimate then yield~\eqref{eqn:product-restriction} upon taking the infimum over all extensions of $f$ and $g$.
\end{proof}

\emph{Reconstruction and arc-chord control.}
The primitive estimate below controls the integrated tangent angle $J\kappa=\theta_\kappa$.
Lemmas~\ref{tau-n-est} and~\ref{lem:Xk-holder} subsequently establish Sobolev bounds for the unit tangent and H\"older regularity for the reconstructed filament curve.

\begin{lemma}\label{lem:theta-kappa}
For $s\ge -1$ and $\kappa\in\thi{s}$, one has
\begin{equation}\label{eqn:theta-Hs+1}
    \norm{\theta_\kappa}_{\hi{s+1}} \le C \norm{\kappa}_{\hr{s}}.
\end{equation}
\end{lemma}
\begin{proof}
Because $C_c^\infty(I)$ is dense in $\thi{s}$, it suffices to prove the estimate for smooth, compactly supported test functions.
Recall the integration operator $J$ and its formal adjoint $J^*$:
\begin{equation}
    (Jf)(x) \coloneqq \int_{-1}^x f(y)\,dy,
    \quad
    (J^*f)(x) \coloneqq \int_x^1 f(y)\,dy.
\end{equation}
These operators satisfy $(Jf)'=f$, $(J^*f)'=-f$, and $\theta_\kappa=J\kappa$.

We first establish the bounds at the integer endpoints $s=-1$ and $s=0$.
The elementary pointwise inequality
\begin{equation}
    \abs{Jf(x)} \le \int_I \abs{f(y)}\,dy
\end{equation}
together with the derivative identity $(Jf)'=f$ yields
\begin{equation}\label{eqn:Jf-1}
    \norm{Jf}_{\hi{1}} \le C\norm{f}_{L^2(I)}.
\end{equation}
The identical estimate holds for $J^*$.
Moreover, by Fubini's theorem, $J^*$ is the $L^2(I)$-adjoint of $J$.
Thus, for any test function $g\in L^2(I)$, duality gives
\begin{equation}
    \abs{\dual{Jf}{g}}
    = \abs{\dual{f}{J^*g}}
    \le \norm{f}_{\hr{-1}}\norm{J^*g}_{\hi{1}}
    \le C\norm{f}_{\hr{-1}}\norm{g}_{L^2(I)}.
\end{equation}
Taking the supremum over all $g\in L^2(I)$ with $\norm{g}_{L^2(I)}\le 1$ yields
\begin{equation}\label{eqn:Jf-2}
    \norm{Jf}_{L^2(I)} \le C \norm{f}_{\hr{-1}}.
\end{equation}
Consequently, $J:\thi{-1}\to L^2(I)$ and $J:L^2(I)\to \hi{1}$ are bounded linear operators.
Operator interpolation therefore implies
\begin{equation}
    \norm{Jf}_{\hi{s+1}} \le C\norm{f}_{\hr{s}},
    \quad -1\le s\le 0.
\end{equation}
For non-negative regularity, let $\ell\ge0$ be any non-negative integer.
Because $\partial_x^{j+1}Jf=\partial_x^j f$, estimate~\eqref{eqn:Jf-1} implies
\begin{equation}
    \norm{Jf}_{\hi{\ell+1}}^2 = \norm{Jf}_{L^2(I)}^2 + \sum_{j=0}^{\ell}\norm{\partial_x^{j+1}Jf}_{L^2(I)}^2 \le
    C \norm{f}_{L^2(I)}^2 + \sum_{j=0}^{\ell}\norm{\partial_x^j f}_{L^2(I)}^2 \le C \norm{f}_{\hi{\ell}}^2.
\end{equation}
Interpolating between consecutive integer indices gives
\begin{equation}
    \norm{Jf}_{\hi{s+1}} \le C\norm{f}_{\hi{s}},
    \quad s\ge 0.
\end{equation}
Applying this inequality to $f=\kappa$ and using the continuous embedding bound $\norm{\kappa}_{\hi{s}}\le C\norm{\kappa}_{\hr{s}}$ establishes~\eqref{eqn:theta-Hs+1} for all $s\ge0$.
The negative range follows from the first interpolation argument.
\end{proof}

Since $\xk'=\bm e_r(\theta_\kappa)$, the angle bound~\eqref{eqn:theta-Hs+1} yields the following Sobolev estimates for the Frenet frame.

\begin{lemma}\label{tau-n-est}
Let $-\frac{1}{2}<s\le 0$ and $\kappa\in \thi{s}$. Then
\begin{align}
    [\xk']_{\hi{s+1}} &\le C\norm{\kappa}_{\hr{s}},\\
    \norm{\xk''}_{\hi{s}} &\le C\paren{1+\norm{\kappa}_{\hr{s}}}\norm{\kappa}_{\hr{s}},\\
    \norm{\xk'}_{\hi{s+1}} &\le C\paren{1+\norm{\kappa}_{\hr{s}}}.\label{eqn:tangent-Sobolev-bound}
\end{align}
Moreover, if $\kappa_i\in \thi{s}$ satisfy
$\norm{\kappa_i}_{\hr{s}}\le M$ for $i=1,2$, then
\begin{align}
    \norm{\xka{1}'-\xka{2}'}_{\hi{s+1}}
    &\le C\paren{1+M} \norm{\kappa_1-\kappa_2}_{\hr{s}},\\
    \norm{\xka{1}''-\xka{2}''}_{\hi{s}}
    &\le C\paren{1+M} \norm{\kappa_1-\kappa_2}_{\hr{s}}.
\end{align}
\end{lemma}

The proof is given in Appendix~\ref{app:geometry-estimates}.

\begin{lemma}\label{lem:Xk-holder}
Let $m\in \mathbb N\cup\{0\}$, $\delta>0$, and
$s=m-\frac{1}{2}+\delta$. Suppose that
$0<\alpha<\min\{\delta,1\}$ and $\kappa\in \thi{s}$. Then
\begin{equation}\label{eqn:Xk-holder-higher}
    \norm{\xk}_{\ci{m+1,\alpha}} \le C\paren{1+\norm{\kappa}_{\hr{s}}}^{m+1}.
\end{equation}
When $m=0$, we have the sharper seminorm estimate
\begin{equation}\label{eqn:Xk-holder}
    [\xk']_{\ci{0,\alpha}} \le C \norm{\kappa}_{\hr{s}}.
\end{equation}
Moreover, if $\kappa_i\in \thi{s}$ satisfy
$\norm{\kappa_i}_{\hr{s}}\le M$ for $i=1,2$, then
\begin{equation}\label{eqn:Xk-holder-higher-lip}
    \norm{\xka{1}-\xka{2}}_{\ci{m+1,\alpha}}
    \le C\paren{1+M}^{m+1}\norm{\kappa_1-\kappa_2}_{\hr{s}}.
\end{equation}
\end{lemma}

The proof is given in Appendix~\ref{app:geometry-estimates}.

For a parametrized curve $\bm X$, the ratio of Euclidean distance to arclength distance is measured by the arc-chord quotient, whose infimum defines the arc-chord constant:
\begin{equation}\label{eqn:arc-chord-intro}
    \abs{\bm X}_*
    \coloneqq
    \inf_{s,\eta\in I,\ s\neq\eta}
    \frac{\abs{\bm X(s)-\bm X(\eta)}}{\abs{s-\eta}}.
\end{equation}
We restrict our attention throughout to embedded curves satisfying the uniform arc-chord condition $\abs{\bm X}_*>0$, which precludes self-intersections and corner singularities.
For the normalized reference curve $\xk$ reconstructed from $\kappa$, we write $\abs{\kappa}_*\coloneqq\abs{\xk}_*$.
Because the curve is parametrized by unit speed ($\abs{\xk'}\equiv 1$), the chord length is always bounded by the arclength:
\begin{equation}
    \abs{\xk(s)-\xk(\eta)}\le\abs{s-\eta},
    \quad s,\eta\in I.
\end{equation}
Consequently, the arc-chord constant is universally confined to the interval $0\le\abs{\kappa}_*\le1$.

Fixing a small exponent $0<\delta\le\frac12$, the space of admissible low-regularity curvatures is defined by
\begin{equation}\label{eqn:U-intro}
    \mc U_\delta
    \coloneqq
    \left\{
    \kappa\in\thi{-\frac12+\delta}:\abs{\kappa}_*>0
    \right\}.
\end{equation}

Throughout the subsequent estimates, $C_{M,m}$ denotes a positive constant that depends exclusively on an upper bound $M$ for the curvature norm and a lower bound $m$ for the arc-chord constant.
The constant $C_{M,m}$ is locally bounded for $M\ge0$ and $m>0$, nondecreasing in $M$ and $m^{-1}$, and its value may change from line to line.
Its dependence on fixed regularity exponents is suppressed, whereas factors expressing vanishing at the origin are maintained explicitly.

To formulate the local existence theory, recall the integration operator $J\kappa=\thk$ and define its operator norm
\begin{equation}
    C_{J,\delta}
    \coloneqq
    \sup_{0\neq\kappa\in\thi{-\frac12+\delta}}
    \frac{\norm{J\kappa}_{\ci{0}}}
    {\norm{\kappa}_{\hr{-\frac12+\delta}}}.
\end{equation}
Lemma~\ref{lem:theta-kappa} together with the Sobolev embedding $H^{\frac12+\delta}(I)\hookrightarrow \ci{0}(I)$ ensures that $C_{J,\delta}<\infty$.
Given a reference curvature $\kappa_0\in\mc U_\delta$, the following open neighborhood provides a common curvature norm bound and a uniform arc-chord lower bound for all reconstructed curves and their convex combinations, as required in the difference estimates:
\begin{equation}\label{eqn:Bkappa0}
\mc B(\kappa_0) \coloneqq \left\{\kappa\in \thi{-\frac{1}{2}+\delta}: \norm{\kappa-\kappa_0}_{\hr{-\frac{1}{2}+\delta}}\le \frac{1}{2C_{J,\delta}}\abs{\kappa_0}_* \right\}.
\end{equation}

\begin{lemma}\label{lem:set-ok0}
For $\kappa_0\in \mc U_\delta$, let
\begin{equation}
    M
    = \norm{\kappa_0}_{\hr{-\frac{1}{2}+\delta}}
    +\frac{1}{2C_{J,\delta}}\abs{\kappa_0}_*,
    \quad
    m = \frac{1}{2}\abs{\kappa_0}_*.
\end{equation}
If $\kappa_1,\kappa_2 \in \mc B(\kappa_0)$, then
\begin{align}\label{eqn:neighborhood-geometry}
    \norm{\kappa_i}_{\hr{-\frac{1}{2}+\delta}} \le M,
    \quad
    \abs{\rho\xka{1} + \paren{1-\rho}\xka{2}}_* \ge m,
\end{align}
for $i=1,2$ and every $\rho\in[0,1]$.
Hence $\mc B(\kappa_0)\subset\mc U_\delta$.
\end{lemma}
\begin{proof}
For any $\kappa_i\in\mc B(\kappa_0)$, the triangle inequality gives
\begin{equation}
    \norm{\kappa_i}_{\hr{-\frac{1}{2}+\delta}}
    \le \norm{\kappa_0}_{\hr{-\frac{1}{2}+\delta}}
    + \frac{1}{2C_{J,\delta}}\abs{\kappa_0}_*
    =M.
\end{equation}
For any two curves $\bm X,\bm Y\in\ci{1}$, the arc-chord functional satisfies the Lipschitz continuity bound
\begin{equation}\label{eqn:arc-chord-C1-continuity}
\begin{aligned}
    &\abs{\abs{\bm X}_* - \abs{\bm Y}_*}
    = \abs{ \inf_{s\neq t}\abs{\frac{\bm X(s)-\bm X(t)}{s-t}}
    - \inf_{s\neq t}\abs{\frac{\bm Y(s)-\bm Y(t)}{s-t}} } \\
    &\le \sup_{s\neq t}\abs{\abs{\frac{\bm X(s)-\bm X(t)}{s-t}}
    - \abs{\frac{\bm Y(s)-\bm Y(t)}{s-t}} }
    \le \sup_{s\neq t}\abs{\frac{(\bm X-\bm Y)(s)-(\bm X-\bm Y)(t)}{s-t}} \\
    &= \sup_{s\neq t}\abs{\int_0^1
    (\bm X-\bm Y)'(t+\lambda(s-t))\,d\lambda}
    \le \norm{\bm X'-\bm Y'}_{\ci{0}}.
\end{aligned}
\end{equation}
For $i=1,2$, Lemma~\ref{lem:theta-kappa} gives
\begin{equation}
\begin{aligned}
    &\norm{\xka{i}'-\xka{0}'}_{\ci{0}} = \sup_{s\in I} \abs{\bm e_r(\thka{i}(s)) - \bm e_r(\thka{0}(s))} \le \sup_{s\in I} \abs{\thka{i}(s)-\thka{0}(s)} \\
    & \le C_{J,\delta}\norm{\kappa_i-\kappa_0}_{\hr{-\frac{1}{2}+\delta}} \le \frac{1}{2}\abs{\kappa_0}_*.
\end{aligned}
\end{equation}
For any convex combination parameter $\rho\in[0,1]$, estimate~\eqref{eqn:arc-chord-C1-continuity} implies
\begin{equation}
\begin{aligned}
&\abs{\abs{\rho\xka{1}+(1-\rho)\xka{2}}_* - \abs{\kappa_0}_* }
\le \rho\norm{\xka{1}'-\xka{0}'}_{\ci{0}}
+(1-\rho)\norm{\xka{2}'-\xka{0}'}_{\ci{0}} \\
&\le C_{J,\delta} \rho\norm{\kappa_1-\kappa_0}_{\hr{-\frac{1}{2}+\delta}}
+C_{J,\delta}(1-\rho)\norm{\kappa_2-\kappa_0}_{\hr{-\frac{1}{2}+\delta}}
\le \frac{1}{2}\abs{\kappa_0}_*.
\end{aligned}
\end{equation}
Hence, by the reverse triangle inequality,
\begin{equation}
    \abs{\rho\xka{1} + \paren{1-\rho}\xka{2}}_*
    \ge \abs{\kappa_0}_* - \frac{1}{2}\abs{\kappa_0}_* = \frac{1}{2}\abs{\kappa_0}_*=m.
\end{equation}
This confirms that the convex combination remains in $\mc U_\delta$, and in particular $\mc B(\kappa_0)\subset\mc U_\delta$.
\end{proof}

\subsection{Linear operators and semigroup estimates}
\label{subsec:linear-theory}
\label{subsec:interval-regularity}

The order-one operator $\p_s\mc H$ is the elliptic principal part of the tension equation, while $-\Lambda = -\frac14\p_s^2\mc H\p_s$ governs curvature dissipation.
We establish coercivity and invertibility for both operators, and then construct the analytic semigroup $e^{t\Lambda}$ with sharp decay and smoothing estimates.

\begin{lemma}\label{Hf-bounded}
For every $s\in\R$, the finite Hilbert transform defines a bounded map
$\mc H:\wt H^s(I)\to H^s(I)$. More precisely,
\begin{equation}\label{eqn:hilbert-bound}
    \norm{\mc Hf}_{H^s(I)} \le C\norm{f}_{H^s(\R)}
\end{equation}
for every $f\in\wt H^s(I)$.
\end{lemma}
\begin{proof}
The whole-line Hilbert transform $\mc H_{\R}$ corresponds to the Fourier multiplier with symbol $-i\operatorname{sgn}(\xi)$, which is bounded on $H^s(\R)$ for every $s\in\R$.
Because the restriction map $p_I: H^s(\R)\to H^s(I)$ is continuous and zero extension is bounded from $\wt H^s(I)$ into $H^s(\R)$, restricting $\mc H_{\R}f$ to the interval $I$ gives
\begin{equation}
    \norm{\mc Hf}_{H^s(I)}
    \le \norm{\mc H_{\R}f}_{H^s(\R)}
    \le C\norm{f}_{H^s(\R)},
\end{equation}
completing the proof.
\end{proof}

The following fractional Poincar\'e-type inequality allows us to control the $L^2$ norm directly by the homogeneous $H^{1/2}$ seminorm on supported spaces.

\begin{lemma}\label{lem:frac-poincare}
There exists $C>0$ such that every $f\in\thi{\frac{1}{2}}$ satisfies
\begin{equation}
    \norm{f}_{L^2(\R)}\le C[f]_{\hr{\frac{1}{2}}}.
\end{equation}
\end{lemma}
\begin{proof}
For any smooth test function $f\in C_c^\infty(I)$, extending $f$ by zero outside $I$ and computing the interaction between $I$ and its complement $\R\setminus I$ yields
\begin{equation}
    [f]_{\hr{\frac{1}{2}}}^2 \ge 2\int_I\abs{f(y)}^2 \left(\int_{-\infty}^{-1}\frac{dx}{\abs{x-y}^2}
    +\int_1^\infty\frac{dx}{\abs{x-y}^2}\right)dy =4\int_I\frac{\abs{f(y)}^2}{1-y^2}\,dy \ge
    2\norm{f}_{L^2(\R)}^2,
\end{equation}
where we used that $1-y^2\le 1$ on $I=(-1,1)$ so that $(1-y^2)^{-1}\ge 1$.
The full estimate for general $f\in\wt H^{1/2}(I)$ follows by density.
\end{proof}

Combining the Fourier multiplier representation of $\p_s\mc H$ with the fractional Poincar\'e inequality establishes coercivity and invertibility for the principal part of the tension operator.

\begin{lemma}\label{ds-H-biject}
The operator
$\p_s\mc H:\thi{\frac{1}{2}}\to H^{-\frac{1}{2}}(I)$ is bounded and
bijective.
\end{lemma}
\begin{proof}
Lemmas~\ref{lem:wt-der} and~\ref{Hf-bounded}, combined with Plancherel's theorem, establish continuity:
\begin{equation}
\begin{aligned}
    \norm{\p_s\mc Hf}_{H^{-\frac{1}{2}}(I)}^2
    \le \norm{\mc H_{\R}f}_{\hr{\frac{1}{2}}}^2  =\int_\R\paren{1+\abs{\xi}^2}^{\frac12}
    \abs{\wh{\mc H_{\R}f}(\xi)}^2\,d\xi
    =\norm{f}_{\hr{\frac{1}{2}}}^2.
\end{aligned}
\end{equation}
To establish coercivity, we write the pairing on test functions $f\in C_c^\infty(I)$ and symmetrize the principal-value kernel:
\begin{equation}\label{eqn:dsh-coer}
\begin{aligned}
    \dual{\p_s\mc Hf}{f}
    =\frac{1}{2\pi}\int_\R\int_\R
    \frac{\paren{f(s)-f(\eta)}^2}{\paren{s-\eta}^2}\,d\eta ds =\frac{1}{2\pi}[f]_{\hr{\frac{1}{2}}}^2.
\end{aligned}
\end{equation}
By the fractional Poincar\'e inequality of Lemma~\ref{lem:frac-poincare}, the Slobodeckij seminorm on the right-hand side controls the full $H^{1/2}(\R)$ norm.
The induced bilinear form $(u,v)\mapsto \dual{\p_s\mc Hu}{v}$ is therefore continuous and strictly coercive on the Hilbert space $\wt H^{1/2}(I)$.
The Lax--Milgram theorem thus guarantees that $\p_s\mc H:\wt H^{1/2}(I)\to H^{-1/2}(I)$ is an isomorphism with a bounded inverse $(\p_s\mc H)^{-1}:H^{-1/2}(I)\to\wt H^{1/2}(I)$.
\end{proof}

Applying the same coercivity identity to $\p_s f$ yields the corresponding isomorphism theorem for the third-order curvature operator $-\Lambda$.

\begin{lemma}\label{ds-Lam-biject}
The operator
$-\Lambda=-\frac14\p_s^2\mc H\p_s:
\thi{\frac{3}{2}}\to H^{-\frac{3}{2}}(I)$ is bounded and bijective.
\end{lemma}
\begin{proof}
Lemmas~\ref{lem:wt-der} and~\ref{ds-H-biject} immediately provide the upper bound:
\begin{equation}
\begin{aligned}
    \norm{\Lambda f}_{H^{-\frac{3}{2}}(I)}^2
    &\le \frac{1}{16}\norm{\p_s \mc H\p_s f}_{\hi{-\frac{1}{2}}}^2
    \le \frac{1}{16}\norm{\p_s f}_{\hr{\frac{1}{2}}}^2
    \le \frac{1}{16}\norm{f}_{\hr{\frac{3}{2}}}^2.
\end{aligned}
\end{equation}
To verify coercivity, integrating by parts and applying the identity~\eqref{eqn:dsh-coer} to $\p_s f$ gives
\begin{equation}\label{eqn:lam-coer}
    \dual{-\Lambda f}{f} = \frac{1}{4}\int_\R f(s)(-\p_s^2\mc H_{\R}\p_s f)(s)\,ds = \frac{1}{4}\int_\R \p_s
    f(s)(\p_s\mc H_{\R}\p_s f)(s)\,ds = \frac{1}{8\pi}[\p_s f]_{\hr{\frac{1}{2}}}^2.
\end{equation}
Lemma~\ref{lem:frac-poincare} bounds $\norm{\p_s f}_{L^2(\R)}$ by $[\p_s f]_{H^{1/2}(\R)}$.
Furthermore, since $f\in\wt H^{3/2}(I)\hookrightarrow H_0^1(I)$, the classical Poincar\'e inequality bounds $\norm{f}_{L^2(I)}$ by $\norm{\p_s f}_{L^2(I)}$.
Thus the full $H^{3/2}(\R)$ norm of $f$ is controlled by the right-hand side of~\eqref{eqn:lam-coer}, yielding the coercivity estimate $\dual{-\Lambda f}{f}\ge c\norm{f}_{\wt H^{3/2}(I)}^2$.
The bilinear form $(u,v)\mapsto \dual{-\Lambda u}{v}$ is therefore continuous and coercive on $\wt H^{3/2}(I)$, and the Lax--Milgram theorem proves bijectivity.
\end{proof}

The quadratic forms associated with~\eqref{eqn:dsh-coer} and~\eqref{eqn:lam-coer} induce natural equivalent Hilbert norms on the supported spaces and their duals.
Specifically, for either $B=\p_s\mc H$ on $V=\thi{\frac12}$ or $B=-\Lambda$ on $V=\thi{\frac32}$, continuity and coercivity yield the two-sided bounds
\begin{equation}\label{eqn:coercive-norm-equivalences}
\begin{aligned}
    c\norm{u}_V^2
    &\le \dual{Bu}{u}\le C\norm{u}_V^2,
    &&u\in V,\\
    c\norm{h}_{V'}^2
    &\le \dual{h}{B^{-1}h}\le C\norm{h}_{V'}^2,
    &&h\in V',
\end{aligned}
\end{equation}
where the dual inequality follows by setting $h=Bu$ and employing the continuity of $B$ and $B^{-1}$.

To establish the semigroup theory for the curvature evolution, we follow the abstract framework of Lunardi~\cite[Chapter~2]{Lunardi1995} and consider the realization of $\Lambda$ on the base space
\begin{equation}
    X=\hi{-\frac32},
    \quad D(\Lambda)=\thi{\frac32}.
\end{equation}
The domain $D(\Lambda)$ is dense in $X$, and the boundary conditions~\eqref{eqn:cur-bc} are intrinsically encoded through the supported Sobolev space $\wt H^{3/2}(I)$.
By Lemma~\ref{ds-Lam-biject}, the graph norm $\norm{u}_X + \norm{\Lambda u}_X$ is equivalent to the standard $\wt H^{3/2}(I)$ norm.
At the interpolation endpoints, we adopt the standard conventions $(X,D(\Lambda))_{0,2}=X$ and $(X,D(\Lambda))_{1,2}=D(\Lambda)$.
By exploiting the symmetry and coercivity of $-\Lambda$, we equip $X$ with an equivalent inner product under which $-\Lambda$ is strictly positive and self-adjoint.
This structure leads to the following fundamental generation and decay estimates.

\begin{proposition}\label{prop:lambda-semigroup}
The realization $\Lambda:D(\Lambda)\subset X\to X$ is sectorial and
generates an analytic semigroup $e^{t\Lambda}$ on $X$. Its spectral bound
satisfies
\begin{equation}
    \omega_\Lambda
    \coloneqq\sup\{\operatorname{Re}z:z\in\sigma(\Lambda)\}<0.
\end{equation}
For $\omega\in(\omega_\Lambda,0)$, $n\in\{0,1\}$, and
$\alpha,\beta\in[0,1]$, with $\alpha\le\beta$ when $n=0$, one has
\begin{equation}\label{eqn:lambda-semigroup-est}
    \norm{\Lambda^n e^{t\Lambda}f}_{(X,D(\Lambda))_{\beta,2}}
    \le C_{n,\alpha,\beta,\omega}e^{\omega t}
    \max\{t^{-n+\alpha-\beta},1\}
    \norm{f}_{(X,D(\Lambda))_{\alpha,2}}
\end{equation}
for every $f\in(X,D(\Lambda))_{\alpha,2}$ and $t>0$.
Moreover, for $0<\alpha<1$ and
$f\in(X,D(\Lambda))_{\alpha,2}$,
\begin{equation}\label{eqn:lambda-semigroup-cont}
    \lim_{t\downarrow0}
    \norm{e^{t\Lambda}f-f}_{(X,D(\Lambda))_{\alpha,2}}=0.
\end{equation}
\end{proposition}
\begin{proof}
Set $B=-\Lambda$. By the symmetry and coercivity properties established in~\eqref{eqn:lam-coer} and~\eqref{eqn:coercive-norm-equivalences}, the formula
\begin{equation}\label{eqn:x-lambda-norm}
    (f,g)_{X_{-\Lambda}}
    \coloneqq\dual{f}{B^{-1}g}_{X,D(\Lambda)}
\end{equation}
defines an inner product on $X$ whose induced norm is equivalent to the original dual Sobolev norm on $H^{-3/2}(I)$.
In this Hilbert space $(X, (\cdot,\cdot)_{X_{-\Lambda}})$, the inverse operator $B^{-1}$ is bounded and satisfies
\begin{equation*}
    (B^{-1}f,g)_{X_{-\Lambda}}
    =(B^{-1}f,B^{-1}g)_{L^2(I)},
    \quad f,g\in X.
\end{equation*}
Consequently, $B^{-1}$ is a strictly positive, self-adjoint, and injective operator with dense range $D(\Lambda)$.
Its unbounded inverse $B = -\Lambda$ is therefore self-adjoint and positive with domain $D(B) = D(\Lambda)$.
The spectral theorem for unbounded self-adjoint operators implies that the spectrum of $\Lambda = -B$ is purely real and satisfies
\begin{equation*}
    \sigma(\Lambda)\subset(-\infty,-\mu],
    \quad
    \mu=\norm{B^{-1}}_{\mc L(X_{-\Lambda})}^{-1}>0.
\end{equation*}
This ensures the negative spectral bound $\omega_\Lambda<0$, sectoriality, and the generation of an analytic semigroup $e^{t\Lambda}$ on $X$.

The standard fractional power and semigroup bounds of Lunardi~\cite[Propositions~2.2.2, 2.2.9, and~2.3.1]{Lunardi1995} with $p=2$ yield the smoothing estimate~\eqref{eqn:lambda-semigroup-est} for intermediate indices $0<\alpha,\beta<1$.
The endpoint cases $\alpha,\beta\in\{0,1\}$ follow from the base-space semigroup bounds and graph-norm estimates via real interpolation, using the commutation relation $\Lambda e^{t\Lambda}f=e^{t\Lambda}\Lambda f$ for $f\in D(\Lambda)$.
The graph-norm equivalence and the interpolation identity~\eqref{eqn:hs-intp} identify the intermediate domains $(X,D(\Lambda))_{\theta,2}$ with the concrete supported and restriction Sobolev spaces.
Finally, strong continuity at the initial time~\eqref{eqn:lambda-semigroup-cont} follows from \cite[Proposition~2.2.8]{Lunardi1995}.
\end{proof}

\subsection{Wiener--Hopf factorization and endpoint regularity}
\label{subsec:endpoint-factorization}

To determine the endpoint singularity profiles of the tension and curvature at the free ends, we analyze the order-$k$ multiplier $A_k$ with symbol $(1+\xi^2)^{k/2}$ via Wiener--Hopf factorization on the half-line and subsequent localization to $I=(-1,1)$.
The case $k=1$ governs the tension, while $k=3$ generates the $(1\pm s)^{3/2}$ curvature boundary profile in Theorem~\ref{thm:main-regularity}.

Write $a_k$ for the Fourier symbol of $A_k$ and $A_{k,\pm}$ for the multipliers associated with the factors $a_{k,\pm}$, where
\begin{equation}\label{eqn:orderk-symbol-factorization}
    a_k(\xi)
    =(1+\xi^2)^{\frac{k}{2}}
    =a_{k,+}(\xi)a_{k,-}(\xi),
    \quad
    a_{k,+}(\xi)=(\xi-i)^{\frac{k}{2}},
    \quad
    a_{k,-}(\xi)=(\xi+i)^{\frac{k}{2}}.
\end{equation}
The symbol $a_{k,+}$ is the boundary trace of the holomorphic function $(z-i)^{k/2}=\exp(\frac{k}{2}\log(z-i))$ in the lower half-plane $\{z\in\mathbb C:\operatorname{Im} z<0\}$, with branch cut along the upper imaginary axis such that $-\pi<\arg(z-i)<0$.
Similarly, $a_{k,-}$ is the boundary trace of $(z+i)^{k/2}$ holomorphic in the upper half-plane $\{z\in\mathbb C:\operatorname{Im} z>0\}$, with $0<\arg(z+i)<\pi$.
For any real frequency $\xi\in\R$, the arguments of $\xi-i$ and $\xi+i$ are opposite and their moduli both equal $(1+\xi^2)^{1/2}$.
Consequently, their $(k/2)$-th powers multiply to exactly $(1+\xi^2)^{k/2}$, regardless of whether $k$ is even or odd.
The reciprocal symbols $a_{k,\pm}^{-1}$ are holomorphic in the same respective half-planes.

Under our Fourier transform convention, the support-preserving properties of these operators follow from the Laplace-transform identity
\begin{equation}\label{eq:halfline-factor-kernel}
    \mathcal F^{-1}\!\bigl[(\xi\mp i)^{-a}\bigr](x)
    =\frac{\sqrt{2\pi}\,e^{\pm i\pi a/2}}{\Gamma(a)}
    \mathbf1_{\R_\pm}(x)e^{-|x|}|x|^{a-1},
    \quad a>0.
\end{equation}
Here $\Gamma$ denotes Euler's gamma function.
Formula~\eqref{eq:halfline-factor-kernel} demonstrates that the inverse operators $A_{k,\pm}^{-1}$ strictly preserve support in $\overline{\R_\pm}$.
Writing $(\xi\mp i)^{k/2}=(\xi\mp i)^N(\xi\mp i)^{-(N-k/2)}$ for an integer $N>k/2$ confirms the identical support preservation for $A_{k,\pm}$.
Standard multiplier bounds extend these identities to all Sobolev spaces: for each sign, $A_{k,\pm}:H^r(\R)\to H^{r-k/2}(\R)$ is an isomorphism whose inverse shifts the regularity by $+k/2$.
Moreover, both $A_{k,\pm}$ and $A_{k,\pm}^{-1}$ map distributions supported in $\overline{\R_\pm}$ into distributions supported in $\overline{\R_\pm}$.

For $g\in L^2(\R)$, the spatial half-line projections in frequency space are defined by
\begin{equation}\label{eq:Pi-support-projections}
    \Pi_\pm g
    \coloneqq
    \mathcal F\!\left(
    \mathbf1_{\R_\pm}\mathcal F^{-1}g
    \right).
\end{equation}
Multiplication by the characteristic function $\mathbf1_{\R_\pm}$ defines a bounded operator on $H^s(\R)$ for all indices $|s|<\frac12$.
Consequently, the projections $\Pi_\pm$ are bounded on $\mathcal F(H^s(\R))$ in this range and project onto $\mathcal F(H^s_{\overline{\R_\pm}}(\R))$.
For negative regularity $s$, the projections are defined by continuous extension.
Throughout this range, they satisfy $\Pi_++\Pi_-=\Id$ and $\Pi_+\Pi_-=\Pi_-\Pi_+=0$, yielding the topological direct sum decomposition
\begin{equation*}
    H^s(\R)=H^s_{\overline{\R_+}}(\R)
    \oplus H^s_{\overline{\R_-}}(\R).
\end{equation*}
This direct sum has trivial intersection: any distribution supported exclusively at the origin is a finite linear combination of Dirac derivatives $\delta_0^{(j)}$, none of which belong to $H^s(\R)$ when $s>-\frac12$.
Thus $\Pi_+$ acts as the identity on the positive-support subspace and annihilates the negative-support subspace.
We now consider the half-line model problem
\begin{equation}\label{eq:A-orderk-halfline}
    p_+ A_k\psi_+ = f \quad \text{on } \R_+,
\end{equation}
where $p_+$ denotes restriction to $\R_+$ and $\psi_+$ is supported in $\overline{\R_+}$, meaning $\psi_+(x)=\psi(x)$ for $x>0$ and $\psi_+(x)=0$ for $x<0$.

\begin{lemma}\label{lem:A-orderk-halfline}
Let $k\ge 1$ be an integer and let $|\varepsilon|<\frac12$.
The constants in the estimates below depend only on $k,\varepsilon$.
\begin{enumerate}
    \item If $f\in H^{-\frac{k}{2}+\varepsilon}(\R_+)$, then the equation
    \eqref{eq:A-orderk-halfline} admits a unique solution
    $\psi_+\in H^{\frac{k}{2}+\varepsilon}_{\overline{\R_+}}(\R)$, and
    \begin{equation}
        \norm{\psi_+}_{H^{\frac{k}{2}+\varepsilon}(\R)}
        \le C \norm{f}_{H^{-\frac{k}{2}+\varepsilon}(\R_+)}.
    \end{equation}
    \item If $f\in H^{-\frac{k}{2}+1+\varepsilon}(\R_+)$, then the solution admits the decomposition
    \begin{equation}
        \psi_+(x)
        = c\,\mathbf1_{[0,\infty)}(x)e^{-x}x^{\frac{k}{2}} + \psi_r(x),
    \end{equation}
    where $c$ depends on $f$ and
    $\psi_r\in H^{\frac{k}{2}+1+\varepsilon}_{\overline{\R_+}}(\R)$. Moreover,
    \begin{equation}
        \norm{\psi_r}_{H^{\frac{k}{2}+1+\varepsilon}(\R)}
        + |c|
        \le C \norm{f}_{H^{-\frac{k}{2}+1+\varepsilon}(\R_+)}.
    \end{equation}
\end{enumerate}
\end{lemma}
\begin{proof}
\emph{Solution formula.}
Let $\wt f\in H^{-\frac{k}{2}+\varepsilon}(\R)$ be an extension of $f$ such that
$\norm{\wt f}_{H^{-\frac{k}{2}+\varepsilon}(\R)}
\le 2 \norm{f}_{H^{-\frac{k}{2}+\varepsilon}(\R_+)}$.
For a solution of~\eqref{eq:A-orderk-halfline}, set
\begin{equation}
    \psi_-=\wt f-A_k\psi_+
    \in H^{-\frac{k}{2}+\varepsilon}_{\overline{\R_-}}(\R).
\end{equation}
Taking Fourier transforms gives
\begin{equation}
    a_{k,+}(\xi)a_{k,-}(\xi)\wh{\psi_+}(\xi) + \wh{\psi_-}(\xi)
    = \wh{\wt f}(\xi).
\end{equation}
Multiplication by $a_{k,-}^{-1}(\xi)$ gives
\begin{equation}\label{eq:A-orderk-wh}
    a_{k,+}(\xi)\wh{\psi_+}(\xi)
    + a_{k,-}^{-1}(\xi)\wh{\psi_-}(\xi)
    = a_{k,-}^{-1}(\xi)\wh{\wt f}(\xi).
\end{equation}
Since $a_{k,-}^{-1}$ has order $-\frac{k}{2}$, the right-hand side belongs to
$\mathcal{F}(H^\varepsilon(\R))$. The support properties following~\eqref{eq:halfline-factor-kernel} give
$a_{k,+}\wh{\psi_+}\in \mathcal{F}(H^\varepsilon_{\overline{\R_+}}(\R))$
and
$a_{k,-}^{-1}\wh{\psi_-}\in \mathcal{F}(H^\varepsilon_{\overline{\R_-}}(\R))$.
Because $\Pi_+:\mathcal{F}(H^\varepsilon(\R))\to\mathcal{F}(H^\varepsilon(\R))$
is bounded for $|\varepsilon|<\frac12$, applying $\Pi_+$ to
\eqref{eq:A-orderk-wh} yields
$a_{k,+}(\xi)\wh{\psi_+}(\xi)
= \Pi_+\!\bigl[a_{k,-}^{-1}\wh{\wt f}\bigr](\xi)$.
Hence
\begin{equation}\label{eq:A-orderk-repr}
    \wh{\psi_+}(\xi)
    = a_{k,+}^{-1}(\xi)\Pi_+\!\bigl[a_{k,-}^{-1}\wh{\wt f}\bigr](\xi).
\end{equation}

\emph{Existence and uniqueness.}
Conversely, \eqref{eq:A-orderk-repr} defines a solution: if
$g=a_{k,-}^{-1}\wh{\wt f}$, then
$\wh{\wt f}-a_k\wh{\psi_+}=a_{k,-}\Pi_-g$, and the inverse Fourier
transform of the right-hand side is supported in $\overline{\R_-}$.
Thus
$\wh{\psi_+}(\xi)\in
\mathcal{F}(H^{\frac{k}{2}+\varepsilon}_{\overline{\R_+}}(\R))$ and
\begin{equation}
    \norm{\psi_+}_{H^{\frac{k}{2}+\varepsilon}(\R)}
    \le C \norm{\wt f}_{H^{-\frac{k}{2}+\varepsilon}(\R)}
    \le C \norm{f}_{H^{-\frac{k}{2}+\varepsilon}(\R_+)}.
\end{equation}
For homogeneous data, take $\wt f=0$ in
\eqref{eq:A-orderk-wh}. Applying $\Pi_+$ gives
$a_{k,+}\wh{\psi_+}=0$, hence $\psi_+=0$. This proves uniqueness in
the stated supported Sobolev class and also shows that
\eqref{eq:A-orderk-repr} is independent of the chosen extension.

\emph{Extraction of the endpoint term.}
For part (2), assume
$f\in H^{-\frac{k}{2}+1+\varepsilon}(\R_+)$ and choose an extension
$\wt f\in H^{-\frac{k}{2}+1+\varepsilon}(\R)$ such that
\begin{equation}
    \norm{\wt f}_{H^{-\frac{k}{2}+1+\varepsilon}(\R)}
    \le 2\norm{f}_{H^{-\frac{k}{2}+1+\varepsilon}(\R_+)}.
\end{equation}
Then the same representation~\eqref{eq:A-orderk-repr} holds and
\begin{equation}
    g\coloneqq a_{k,-}^{-1}\wh{\wt f}
    \in\mathcal F(H^{1+\varepsilon}(\R)).
\end{equation}
Let $h=\mathcal F^{-1}g$. The extra regularity of the forcing gives
$h\in H^{1+\varepsilon}(\R)$, so $h(0)$ is well defined since
$1+\varepsilon>\frac12$; this trace generates the explicit $x^{k/2}$
boundary term. Define
\begin{equation}\label{eq:Pi-prime-definition}
    \ell_0(g)
    \coloneqq
    -\frac{i}{\sqrt{2\pi}}h(0)
    =
    -\frac{i}{\sqrt{2\pi}}
    \left(\mathcal F^{-1}g\right)(0).
\end{equation}
The distributional identity
\begin{equation}
    \p_x\paren{\mathbf1_{\R_+}h}
    =\mathbf1_{\R_+}\p_xh+h(0)\delta_0
\end{equation}
and the conventions $\mathcal F(\p_xh)=i\xi\wh h$ and
$\mathcal F\delta_0=(2\pi)^{-1/2}$ give
\begin{equation*}
    i(\xi-i)\Pi_+g
    =i\Pi_+\!\left[(\xi-i)g\right]+\frac{h(0)}{\sqrt{2\pi}}.
\end{equation*}
Dividing by $i(\xi-i)$ yields
\begin{equation}\label{eq:Pi-first-order-expansion}
    \Pi_+g
    =(\xi-i)^{-1}\ell_0(g)
    +(\xi-i)^{-1}\Pi_+\!\left[(\xi-i)g\right].
\end{equation}
The term $(\xi-i)^{-1}\ell_0(g)$ comes from $h(0)\delta_0$ and will
determine the coefficient $c$.
Here $\Pi_+$ is applied to $(\xi-i)g\in\mathcal F(H^\varepsilon(\R))$,
not as a bounded projection on $\mathcal F(H^{1+\varepsilon}(\R))$.
Thus, with $b_+(\xi)=\xi-i$, the representation
\eqref{eq:A-orderk-repr} becomes
\begin{equation}\label{eq:A-orderk-expansion}
    \wh{\psi_+}(\xi)
    = a_{k,+}^{-1}(\xi)b_+^{-1}(\xi)
      \ell_0\!\bigl[a_{k,-}^{-1}\wh{\wt f}\bigr]
    + a_{k,+}^{-1}(\xi)b_+^{-1}(\xi)
      \Pi_+\!\bigl[b_+a_{k,-}^{-1}\wh{\wt f}\bigr](\xi).
\end{equation}
The second term in~\eqref{eq:A-orderk-expansion} belongs to $\mathcal{F}(H^{\frac{k}{2}+1+\varepsilon}_{\overline{\R_+}}(\R))$ because
$b_+a_{k,-}^{-1}$ has order $1-\frac{k}{2}$ and
$\Pi_+$ is bounded on $\mathcal F(H^\varepsilon(\R))$. Hence
\begin{equation}
    \psi_r
    \coloneqq \mathcal{F}^{-1}\!\left(
    a_{k,+}^{-1}b_+^{-1}
    \Pi_+\!\bigl[b_+a_{k,-}^{-1}\wh{\wt f}\bigr]\right)
    \in H^{\frac{k}{2}+1+\varepsilon}_{\overline{\R_+}}(\R).
\end{equation}
For the first term, note that
$a_{k,+}^{-1}(\xi)b_+^{-1}(\xi)
=(\xi-i)^{-\frac{k}{2}-1}$.
By \eqref{eq:halfline-factor-kernel} with $a=k/2+1$, its inverse
Fourier transform is
$\sqrt{2\pi}e^{i\pi(k/2+1)/2}/\Gamma(k/2+1)$ times
$\mathbf1_{[0,\infty)}(x)e^{-x}x^{\frac{k}{2}}$.
Hence the first term in
\eqref{eq:A-orderk-expansion} has the form
$c\,\mathbf1_{[0,\infty)}(x)e^{-x}x^{\frac{k}{2}}$,
and the trace and multiplier bounds give
\begin{equation}
    |c|
    +\norm{\psi_r}_{H^{\frac{k}{2}+1+\varepsilon}(\R)}
    \le C\norm{\mathcal F^{-1}\!\left(
    a_{k,-}^{-1}\wh{\wt f}\right)}_{H^{1+\varepsilon}(\R)}
    \le C \norm{f}_{H^{-\frac{k}{2}+1+\varepsilon}(\R_+)}.
\end{equation}
The coefficient $c$ is intrinsic. By~\eqref{eq:halfline-factor-kernel},
the Fourier transform of $\mathbf1_{[0,\infty)}e^{-x}x^{k/2}$ has modulus
a nonzero constant times $\langle\xi\rangle^{-k/2-1}$, so this profile
does not belong to $H^{k/2+1+\varepsilon}(\R)$ for
$\varepsilon>-\frac12$. Two such decompositions of the same solution
therefore have the same coefficient and remainder.
\end{proof}

Lemma~\ref{lem:A-orderk-halfline} establishes the half-line model theory.
For equations posed on the finite interval $I=(-1,1)$, localization via a smooth partition of unity near the two endpoints yields the corresponding elliptic regularity and endpoint asymptotic expansions.

\begin{lemma}\label{lem:A-orderk-interval}
Let $k\ge 1$ be an integer, let $0\le \varepsilon<\frac12$, and let
$I=(-1,1)$. Suppose $u\in H^{\frac{k}{2}}_{\overline I}(\R)$ satisfies
\begin{equation}\label{eq:A-orderk-interval}
    p_I A_k u = f \quad \text{on } I,
\end{equation}
with $f\in H^{-\frac{k}{2}+\varepsilon}(I)$.
The constants in these estimates depend only on $k$, $\varepsilon$, and
the fixed cutoffs.
\begin{enumerate}
    \item Then $u\in H^{\frac{k}{2}+\varepsilon}_{\overline I}(\R)$ and
\begin{equation}\label{eq:A-orderk-interval-est}
    \norm{u}_{H^{\frac{k}{2}+\varepsilon}(\R)}
    \le C\paren{
    \norm{f}_{H^{-\frac{k}{2}+\varepsilon}(I)}
    + \norm{u}_{H^{\frac{k}{2}}(\R)}} .
\end{equation}
    \item If in addition $f\in H^{-\frac{k}{2}+1+\varepsilon}(I)$, then
    \begin{equation}\label{eq:A-orderk-interval-split}
        u(s)=c_-\chi_-(s)e^{-(s+1)}(s+1)_+^{\frac{k}{2}}
        +c_+\chi_+(s)e^{-(1-s)}(1-s)_+^{\frac{k}{2}}+u_r(s),
    \end{equation}
    where $\chi_\pm$ are real-valued smooth cutoffs supported near the
    respective endpoints and equal to one there, with supports disjoint
    from the opposite endpoint, $c_\pm\in\mathbb C$, and
    $u_r\in H^{\frac{k}{2}+1+\varepsilon}_{\overline I}(\R)$.
    The coefficients are real when $u$ is real-valued. Moreover,
    \begin{equation}\label{eq:A-orderk-interval-split-est}
        |c_-|+|c_+|+\norm{u_r}_{H^{\frac{k}{2}+1+\varepsilon}(\R)}
        \le C\paren{
        \norm{f}_{H^{-\frac{k}{2}+1+\varepsilon}(I)}
        +\norm{u}_{H^{\frac{k}{2}}(\R)}} .
    \end{equation}
\end{enumerate}
\end{lemma}

\begin{proof}
We localize at the left endpoint and translate to the half-line, where
Lemma~\ref{lem:A-orderk-halfline} applies. Reflection treats the right
endpoint; the whole-line $A_k$-multiplier estimate treats the interior. We then
combine these three pieces.
The cutoffs introduce commutators, for which we use

\begin{equation}\label{eq:A-orderk-commutator}
    \norm{[A_k,\phi]v}_{H^{r-k+1}(\R)}
    \le C_{k,r,\phi}\norm{v}_{H^r(\R)},
    \quad r\in\R,\quad \phi\in C_c^\infty(\R).
\end{equation}
Indeed, its Fourier kernel is
$(2\pi)^{-1/2}(\langle\xi\rangle^k-\langle\eta\rangle^k)
\wh\phi(\xi-\eta)$. The mean-value estimate for the symbol, Peetre's
inequality, and rapid decay of $\wh\phi$ give
\eqref{eq:A-orderk-commutator} by Young's inequality. Thus the
commutator has order $k-1$, although it need not have compactly supported
output.

Let $\wt f\in H^{-\frac{k}{2}+\varepsilon}(\R)$ be an extension of $f$ with
$\norm{\wt f}_{H^{-\frac{k}{2}+\varepsilon}(\R)}
\le C\norm{f}_{H^{-\frac{k}{2}+\varepsilon}(I)}$. Choose real-valued smooth cutoffs
$\vph_-$, $\vph_+$, and $\vph_0$ such that $\vph_-$ is supported in a small
neighborhood of $s=-1$, $\vph_+$ is supported in a small neighborhood of
$s=1$, $\vph_0$ is compactly supported in $I$, and
$\vph_-+\vph_++\vph_0=1$ on $I$. Take $\vph_\pm=1$ near the respective
endpoint, with support disjoint from the other endpoint. Set
$\chi_\pm=\vph_\pm$ in the statement.

First consider the left endpoint. Shift $s=-1$ to the origin by setting
$x=s+1$, and define $u_-(x)=\vph_-(s)u(s)$ for $x>0$ and $u_-(x)=0$ for
$x<0$. Under this change of variables,
$\operatorname{supp}\vph_-\cap I$ lies in $\R_+$. Since $\vph_-$
vanishes near and to the right of $s=1$, the identity
$\vph_-A_ku=\vph_-\wt f$ holds on $s>-1$. Thus the localized equation,
with all terms on the right translated to the $x$ variable, is
\begin{equation}\label{eq:A-orderk-left-local}
    p_+ A_k u_-
    = p_+\paren{\vph_-\wt f + [A_k,\vph_-]u}.
\end{equation}
The first term satisfies
\begin{equation}
    \norm{p_+\vph_-\wt f}_{H^{-\frac{k}{2}+\varepsilon}(\R_+)}
    \le C\norm{f}_{H^{-\frac{k}{2}+\varepsilon}(I)}.
\end{equation}
For the commutator term, \eqref{eq:A-orderk-commutator} with $r=k/2$
gives
\begin{equation}
    \norm{p_+[A_k,\vph_-]u}_{H^{-\frac{k}{2}+\varepsilon}(\R_+)}
    \le \norm{[A_k,\vph_-]u}_{H^{-\frac{k}{2}+\varepsilon}(\R)}
    \le C\norm{u}_{H^{\frac{k}{2}}(\R)}.
\end{equation}
Here we used $\varepsilon<1$ to embed $H^{1-k/2}(\R)$ into
$H^{-k/2+\varepsilon}(\R)$. Thus \eqref{eq:A-orderk-left-local} has
right-hand side in $H^{-k/2+\varepsilon}(\R_+)$.
Lemma~\ref{lem:A-orderk-halfline}(1), with the stated $\varepsilon$, gives
a solution $w_-\in H^{\frac{k}{2}+\varepsilon}_{\overline{\R_+}}(\R)$
of \eqref{eq:A-orderk-left-local}. Since both $u_-$ and $w_-$ belong to
$H^{\frac{k}{2}}_{\overline{\R_+}}(\R)$ and solve the same equation,
uniqueness in Lemma~\ref{lem:A-orderk-halfline}(1) with $\varepsilon=0$
gives $u_-=w_-$. Hence
\begin{equation}\label{eq:A-orderk-left-est}
    \norm{\vph_-u}_{H^{\frac{k}{2}+\varepsilon}(\R)}
    \le C\paren{
    \norm{f}_{H^{-\frac{k}{2}+\varepsilon}(I)}
    + \norm{u}_{H^{\frac{k}{2}}(\R)}} .
\end{equation}

At the right endpoint, reflect by $x=1-s$. Since the symbol
$(1+\xi^2)^{k/2}$ is even, the same argument gives
\begin{equation}\label{eq:A-orderk-right-est}
    \norm{\vph_+u}_{H^{\frac{k}{2}+\varepsilon}(\R)}
    \le C\paren{
    \norm{f}_{H^{-\frac{k}{2}+\varepsilon}(I)}
    + \norm{u}_{H^{\frac{k}{2}}(\R)}} .
\end{equation}

For the interior piece, $\vph_0$ is compactly supported in $I$, so we work
on the whole line:
\begin{equation}\label{eq:A-orderk-int-local}
    A_k(\vph_0u)=\vph_0\wt f + [A_k,\vph_0]u .
\end{equation}
As in the proof of \eqref{eq:A-orderk-left-est},
\begin{equation}
    \norm{\vph_0\wt f + [A_k,\vph_0]u}_{H^{-\frac{k}{2}+\varepsilon}(\R)}
    \le C\paren{
    \norm{f}_{H^{-\frac{k}{2}+\varepsilon}(I)}
    + \norm{u}_{H^{\frac{k}{2}}(\R)}} .
\end{equation}
The whole-line $A_k$-multiplier estimate then gives
\begin{equation}\label{eq:A-orderk-int-est}
    \norm{\vph_0u}_{H^{\frac{k}{2}+\varepsilon}(\R)}
    \le C\paren{
    \norm{f}_{H^{-\frac{k}{2}+\varepsilon}(I)}
    + \norm{u}_{H^{\frac{k}{2}}(\R)}} .
\end{equation}
Combining \eqref{eq:A-orderk-left-est}, \eqref{eq:A-orderk-right-est}, and
\eqref{eq:A-orderk-int-est}, and using
$u=\vph_-u+\vph_+u+\vph_0u$, proves part (1).

For part (2), the Sobolev gain from part (1) improves the commutator
terms and allows extraction of the endpoint profiles.
Assume $f\in H^{-\frac{k}{2}+1+\varepsilon}(I)$ and choose the
extension $\wt f\in H^{-\frac{k}{2}+1+\varepsilon}(\R)$. By part (1),
$u\in H^{\frac{k}{2}+\varepsilon}_{\overline I}(\R)$. Hence the commutator
terms in \eqref{eq:A-orderk-left-local} and \eqref{eq:A-orderk-int-local}
satisfy, by \eqref{eq:A-orderk-commutator} with $r=k/2+\varepsilon$,
\begin{equation}
\begin{aligned}
    \norm{p_+[A_k,\vph_-]u}_{H^{-\frac{k}{2}+1+\varepsilon}(\R_+)}
    +\norm{[A_k,\vph_0]u}_{H^{-\frac{k}{2}+1+\varepsilon}(\R)}
    \le C\norm{u}_{H^{\frac{k}{2}+\varepsilon}(\R)} .
\end{aligned}
\end{equation}
Together with \eqref{eq:A-orderk-interval-est}, this is bounded by
$C\paren{\norm{f}_{H^{-\frac{k}{2}+1+\varepsilon}(I)}
+\norm{u}_{H^{\frac{k}{2}}(\R)}}$.
Thus the right-hand side of the left endpoint equation
\eqref{eq:A-orderk-left-local} belongs to
$H^{-\frac{k}{2}+1+\varepsilon}(\R_+)$. Applying
Lemma~\ref{lem:A-orderk-halfline}(2) gives
\begin{equation}
    \vph_-u
    =c_-\mathbf1_{[0,\infty)}(x)e^{-x}x^{\frac{k}{2}}+u_{-,r},
    \quad u_{-,r}\in H^{\frac{k}{2}+1+\varepsilon}_{\overline{\R_+}}(\R),
\end{equation}
with
\begin{equation}
    |c_-|+\norm{u_{-,r}}_{H^{\frac{k}{2}+1+\varepsilon}(\R)}
    \le C\paren{
    \norm{f}_{H^{-\frac{k}{2}+1+\varepsilon}(I)}
    +\norm{u}_{H^{\frac{k}{2}}(\R)} } .
\end{equation}
Returning to the variable $s$ gives the singular term
$c_-e^{-(s+1)}(s+1)_+^{\frac{k}{2}}$. Redefine the left remainder by
\begin{equation*}
    u_{-,r}(s)=\vph_-(s)u(s)
    -c_-\vph_-(s)e^{-(s+1)}(s+1)_+^{\frac{k}{2}}.
\end{equation*}
This adds
\begin{equation*}
    c_-(1-\vph_-(s))e^{-(s+1)}(s+1)_+^{k/2}
\end{equation*}
to the translated half-line remainder. Since $1-\vph_-$ vanishes near
$s=-1$, this correction is smooth across that endpoint; it vanishes for
$s<-1$ and decays exponentially with all derivatives as $s\to+\infty$.
It belongs to every $H^r(\R)$, with norm bounded by a constant times
$|c_-|$, so the same Sobolev estimate holds for the redefined remainder.
Moreover, both terms in this definition are supported in
$\overline I$; hence
$u_{-,r}\in H^{k/2+1+\varepsilon}_{\overline I}(\R)$.
The right endpoint, after the
reflection $x=1-s$, gives
$c_+\vph_+(s)e^{-(1-s)}(1-s)_+^{\frac{k}{2}}$ and a remainder
$u_{+,r}\in H^{\frac{k}{2}+1+\varepsilon}_{\overline I}(\R)$ with the same
estimate, defined in the same way.
The whole-line $A_k$-multiplier estimate applied to
\eqref{eq:A-orderk-int-local} yields
$\vph_0u\in H^{\frac{k}{2}+1+\varepsilon}(\R)$ with the same bound.
Setting
$u_r=u_{-,r}+u_{+,r}+\vph_0u$ proves
\eqref{eq:A-orderk-interval-split} and
\eqref{eq:A-orderk-interval-split-est}.
Localizing the difference of two decompositions near each endpoint and
using the nonmembership of the half-line profile in the remainder space
shows that $c_\pm$ are unique. For real-valued $u$, complex conjugation
gives another decomposition with the same real cutoffs and profiles, so
uniqueness implies $c_\pm\in\R$. This applies in particular to the
curvature expansion in Theorem~\ref{thm:main-regularity}.
\end{proof}

\subsection{Commutator estimates}
\label{subsec:commutator-estimates}

We establish fractional commutator bounds for $\p_s\mc H$ with the Frenet frame vectors $\tk$ and $\nk$.
Lemmas~\ref{dsH_f_comm} and~\ref{lem:tn-comm} supply the low-regularity bounds needed for the tension equation and local well-posedness, while Lemma~\ref{lem:comm-psH} provides the higher-order estimate used for positive-time regularity.

\begin{lemma}\label{dsH_f_comm}
Let $0\le\varepsilon<\alpha<\delta\le\frac12$,
$f\in H^{\frac12+\delta}(I)$, and $g\in\wt H^{\frac12}(I)$. Then, with
$C=C(\varepsilon,\alpha,\delta)$,
\begin{equation}\label{eqn:commutator-low-bound}
    \norm{[\p_s\mc H,f]g}_{\hi{-\frac12+\varepsilon}}
    \le C\paren{[f]_{\ci{0,\alpha}}
    +\norm{\p_s f}_{\hi{-\frac12+\delta}}}
    \norm{g}_{H^{\frac12}(\R)}.
\end{equation}
\end{lemma}
\begin{proof}
We decompose the commutator into a product term and a commutator with the finite Hilbert transform:
\begin{equation}
    [\p_s\mc H,f]g
    =\p_s f\,\mc Hg+\p_s\paren{[\mc H,f]g}.
\end{equation}
For any test function $\vph\in\thi{\frac12-\varepsilon}$, duality together with the product estimate of Lemma~\ref{point_mult_sobo} and the boundedness of $\mc H$ from Lemma~\ref{Hf-bounded} yields
\begin{equation}
\begin{aligned}
    \abs{\dual{\p_s f\mc H g}{\vph}}
    &\le C\norm{\vph}_{H^{\frac12-\varepsilon}(\R)}
    \norm{\p_s f}_{\hi{-\frac12+\delta}}
    \norm{\mc Hg}_{H^{\frac12}(I)} \\
    &\le C\norm{\vph}_{H^{\frac12-\varepsilon}(\R)}
    \norm{\p_s f}_{\hi{-\frac12+\delta}}
    \norm{g}_{H^{\frac12}(\R)}.
\end{aligned}
\end{equation}
Taking the supremum over all $\vph\in\wt H^{\frac12-\varepsilon}(I)$ with unit norm gives
\begin{equation}\label{eqn:dsHf1}
    \norm{\p_s f\,\mc Hg}_{\hi{-\frac12+\varepsilon}}
    \le C\norm{\p_s f}_{\hi{-\frac12+\delta}}
    \norm{g}_{H^{\frac12}(\R)}.
\end{equation}

It remains to estimate the commutator $[\mc H,f]g$.
Set
\begin{equation}
    q=\frac12+\varepsilon,\quad \beta=\delta-\varepsilon>0.
\end{equation}
Because the commutator $[\mc H,f]$ is invariant under adding an arbitrary constant to $f$, there exists a constant $c\in\R$ and a bounded extension $F\in\hr{\frac12+\delta}$ of $f-c$ satisfying
\begin{equation}
    \norm{F}_{\hr{\frac12+\delta}}
    \le C\paren{[f]_{\ci{0,\alpha}}
    +\norm{\p_s f}_{\hi{-\frac12+\delta}}}.
\end{equation}
Identifying $g$ with its zero extension on $\R$, we have
\begin{equation}
    [\mc H,f]g=\left.[\mc H_{\R},F]g\right|_I.
\end{equation}

The Fourier representation of the whole-line commutator is given by
\begin{equation}
    \wh{[\mc H_{\R},F]g}(\xi)
    =c_0\int_\R
    \paren{\operatorname{sgn}(\xi)-\operatorname{sgn}(\eta)}
    \wh F(\xi-\eta)\wh g(\eta)\,d\eta.
\end{equation}
The difference of signs $\operatorname{sgn}(\xi)-\operatorname{sgn}(\eta)$ vanishes unless $\xi$ and $\eta$ have opposite signs, which implies the frequency dominance $\abs{\xi}\le\abs{\xi-\eta}$ on the support of the integrand.
Consequently,
\begin{equation}
    \langle\xi\rangle^q
    \abs{\wh{[\mc H_{\R},F]g}(\xi)}
    \le C\int_\R
    \langle\xi-\eta\rangle^q
    \abs{\wh F(\xi-\eta)}\abs{\wh g(\eta)}\,d\eta.
\end{equation}
Define the auxiliary profiles $U$ and $V$ by
\begin{equation}
    \wh U(\xi)=\langle\xi\rangle^q\abs{\wh F(\xi)},
    \quad
    \wh V(\xi)=\abs{\wh g(\xi)}.
\end{equation}
Plancherel's theorem and the whole-line product estimate \cite[Theorem~7.3]{Behzadan2021}, applied with $s=0$, $s_1=\beta$, and $s_2=\frac12$, imply
\begin{equation}
\begin{aligned}
    \norm{[\mc H_{\R},F]g}_{\hr{q}}
    &\le C\norm{UV}_{L^2(\R)}
    \le C\norm{U}_{\hr{\beta}}\norm{V}_{\hr{\frac12}} \\
    &=C\norm{F}_{\hr{q+\beta}}\norm{g}_{\hr{\frac12}}
    =C\norm{F}_{\hr{\frac12+\delta}}\norm{g}_{\hr{\frac12}}.
\end{aligned}
\end{equation}
The multiplication estimate applies directly because $\beta=\delta-\varepsilon>0$.
Restriction to $I$ and the bound for $F$ now yield
\begin{equation}\label{eqn:dsHf2}
    \norm{\p_s[\mc H,f]g}_{\hi{-\frac12+\varepsilon}}
    \le C\norm{[\mc H,f]g}_{\hi{\frac12+\varepsilon}}
    \le C\paren{[f]_{\ci{0,\alpha}}
    +\norm{\p_s f}_{\hi{-\frac12+\delta}}}
    \norm{g}_{\hr{\frac12}}.
\end{equation}
Combining \eqref{eqn:dsHf1} and~\eqref{eqn:dsHf2} completes the proof.
\end{proof}

Combining the general commutator bound~\eqref{eqn:commutator-low-bound} with the geometric reconstruction bounds specializes the estimates to the unit tangent $\tk$ and normal $\nk$.

\begin{lemma}\label{lem:tn-comm}
Let $0<\varepsilon<\delta\le\frac12$,
$\kappa\in\thi{-\frac12+\delta}$, and $f\in\thi{\frac12}$. Then
\begin{equation}
    \norm{[\p_s\mc H,a]f}_{\hi{-\frac12+\varepsilon}}
    \le C\paren{1+\norm{\kappa}_{\hr{-\frac12+\delta}}}
    \norm{\kappa}_{\hr{-\frac12+\delta}}
    \norm{f}_{\hr{\frac12}},\quad a\in\{\tk,\nk\}.
\end{equation}
If $\delta=\frac12$ and $\kappa\in\wt L^2(I)$, then the sharper estimate
\begin{equation}
    \norm{[\p_s\mc H,a]f}_{\hi{-\frac12+\varepsilon}}
    \le C\norm{\kappa}_{L^2(\R)}\norm{f}_{\hr{\frac12}},
    \quad a\in\{\tk,\nk\},
\end{equation}
holds for every $f\in\thi{\frac12}$.
\end{lemma}
\begin{proof}
Choose $\alpha\in(\varepsilon,\delta)$. We apply Lemma~\ref{dsH_f_comm} with the multipliers $\tk$ and $\nk$.
Because $\tk=\xk'$, Lemma~\ref{lem:Xk-holder} with $m=0$ gives
\begin{equation}
    [\tk]_{\ci{0,\alpha}}\le C\norm{\kappa}_{\hr{-\frac12+\delta}},
\end{equation}
and Lemma~\ref{tau-n-est} gives
\begin{equation}
    \norm{\p_s\tk}_{\hi{-\frac12+\delta}}
    =\norm{\xk''}_{\hi{-\frac12+\delta}}
    \le C\paren{1+\norm{\kappa}_{\hr{-\frac12+\delta}}}
        \norm{\kappa}_{\hr{-\frac12+\delta}}.
\end{equation}
Because $\nk=Q_{\frac{\pi}{2}}^{-1}\tk$ is a rigid rotation of $\tk$, the identical two estimates hold for $\nk$.
Substituting these into Lemma~\ref{dsH_f_comm} establishes the first estimate.

For the critical index $\delta=\frac12$, the Frenet--Serret relations $\p_s\tk=-\kappa\nk$ and $\p_s\nk=\kappa\tk$, together with the unit-speed normalization $\abs{\tk}=\abs{\nk}\equiv 1$, yield
\begin{equation}
    [\tk]_{\ci{0,\alpha}}+[\nk]_{\ci{0,\alpha}}
    +\norm{\p_s\tk}_{L^2(I)}+\norm{\p_s\nk}_{L^2(I)}
    \le C\norm{\kappa}_{L^2(\R)}.
\end{equation}
Applying Lemma~\ref{dsH_f_comm} once more gives the claimed $L^2$-based bound.
\end{proof}

\begin{lemma}\label{lem:comm-psH}
Let $0\le\varepsilon<\delta\le\frac12$,
$f\in H^{\frac52+\delta}(I)$, and
$g\in\thi{\frac12+\varepsilon}$. Then
\begin{equation}\label{eq:comm-psH-bound}
    \norm{[\p_s\mc H,f]g}_{\hi{\frac12+\varepsilon}}
    \le C\norm{f}_{H^{\frac52+\delta}(I)}
    \norm{g}_{\hr{\frac12+\varepsilon}},
\end{equation}
where $C$ depends only on $\delta$ and $\varepsilon$.
\end{lemma}

The proof is given in Appendix~\ref{app:commutator-estimates}.

\subsection{Boundary integral estimates}
\label{sec:divided-diff}

We control the regular Stokeslet remainder $r(s,\eta;\kappa)$ through divided-difference kernel bounds and an abstract weakly singular operator estimate, specializing these to $S_\kappa$ and establishing the Stokes single-layer energy identity on the open filament.

For $s,\eta,s+h\in I$ and $h\in\R$, we define the difference and shift notations:
\begin{align}
    \Delta f(s,\eta)&\coloneqq f(s)-f(\eta), \\
    T_{s,h}f(s,\eta)&\coloneqq f(s+h,\eta), \\
    \Delta_{s,h}f(s,\eta)&\coloneqq f(s+h,\eta)-f(s,\eta).
\end{align}
For $s,\eta\in I$ with $s\ne\eta$, let
\begin{equation}\label{eqn:LM-kernel}
    \ell(s,\eta;\bm Z)
    \coloneqq\log\abs{\frac{\Delta\bm Z}{s-\eta}},
    \quad
    m_{ij}(s,\eta;\bm Z)
    \coloneqq\frac{\Delta Z_i\Delta Z_j}{\abs{\Delta\bm Z}^2}.
\end{equation}
The associated boundary integral operators are defined by
\begin{align}
    (L_{\bm Z}f)(s)
    &\coloneqq\int_I\p_\eta \ell(s,\eta;\bm Z)f(\eta)\,d\eta,
    \label{eqn:F-def}\\
    (M_{\bm Z,ij}f)(s)
    &\coloneqq\int_I\p_\eta m_{ij}(s,\eta;\bm Z)f(\eta)\,d\eta,
    \label{eqn:M-def}\\
    (S_{\kappa,ij}f)(s)
    &\coloneqq\bigl(\delta_{ij}L_{\bm X_\kappa}f
    -M_{\bm X_\kappa,ij}f\bigr)(s) = \int_I\p_\eta r_{ij}(s,\eta;\kappa)f(\eta)\,d\eta,\label{eqn:Sop-definition}
\end{align}
for $i,j=1,2$, where the kernel $r_{ij}$ was defined in~\eqref{eqn:R-kernel}.
For vector-valued profiles, we write $S_\kappa=(S_{\kappa,ij})_{i,j=1}^2$, so that
\begin{equation*}
    (S_\kappa\bm f)(s)
    =\int_I\p_\eta r(s,\eta;\kappa)\bm f(\eta)\,d\eta.
\end{equation*}

The divided-difference estimates below are open interval counterparts of the periodic estimates developed in~\cite[Section~2]{MoriRodenbergSpirn2019Peskin}.

\begin{lemma}\label{LM-kernel-est}
Let $\alpha\in(0,1)$ and
$\bm Z=(Z_1,Z_2)^T\in\ci{1,\alpha}$ satisfy $\abs{\bm Z}_*>0$.
Then
\begin{align}
    \abs{\p_s\ell(s,\eta;\bm Z)},\quad
    \abs{\p_\eta \ell(s,\eta;\bm Z)}
    &\le C\abs{\bm Z}_*^{-1}[\bm Z']_{\ci{0,\alpha}}
    \abs{s-\eta}^{\alpha-1},\label{eqn:LM-1}\\
    \abs{\p_sm_{ij}(s,\eta;\bm Z)},\quad
    \abs{\p_\eta m_{ij}(s,\eta;\bm Z)}
    &\le C\abs{\bm Z}_*^{-1}[\bm Z']_{\ci{0,\alpha}}
    \abs{s-\eta}^{\alpha-1}.
    \label{eqn:LM-2}
\end{align}
If $s,\eta,s+h\in I$ and $\abs{s-\eta}\ge2\abs h$, then
\begin{align}
    \abs{\Delta_{s,h}\paren{\p_\eta \ell(s,\eta;\bm Z)}},\quad
    \abs{\Delta_{s,h}\paren{\p_s\ell(s,\eta;\bm Z)}}
    &\le C\abs{\bm Z}_*^{-2}\norm{\bm Z'}_{\ci{0}}[\bm Z']_{\ci{0,\alpha}}\abs h^\alpha\abs{s-\eta}^{-1},
    \label{eqn:LM-3}\\
    \abs{\Delta_{s,h}\paren{\p_\eta m_{ij}(s,\eta;\bm Z)}},\quad
    \abs{\Delta_{s,h}\paren{\p_sm_{ij}(s,\eta;\bm Z)}}
    &\le C\abs{\bm Z}_*^{-2}\norm{\bm Z'}_{\ci{0,\alpha}}[\bm Z']_{\ci{0,\alpha}}\abs h^\alpha\abs{s-\eta}^{-1}.
    \label{eqn:LM-4}
\end{align}
\end{lemma}

The proof is given in Appendix~\ref{app:boundary-integral-estimates}.

We rely on the following general mapping estimate for integral operators with weakly singular kernels.

\begin{lemma}\label{K-op}
Let $0\le\varepsilon<\alpha<1$.
For $f\in\wt L^2(I)$, define
\begin{equation}
    K f  (s) = \int_I k(s,\eta) f(\eta)\,d\eta.
\end{equation}
Suppose that the pointwise bound
\begin{equation}\label{eqn:kernel-pointwise-hypothesis}
    \abs{k(s,\eta)} \le a_1 \abs{s-\eta}^{\alpha-1}
\end{equation}
holds for all $s,\eta\in I$ with $s\ne\eta$, and that
\begin{equation}\label{eqn:kernel-increment-hypothesis}
    \abs{k(s+h,\eta)-k(s,\eta)}
    \le a_2\abs{h}^\alpha\abs{s-\eta}^{-1}
\end{equation}
whenever $s,\eta,s+h\in I$, $h\in\R$, and
$\abs{s-\eta}\ge2\abs h$.
Then there exists a constant $C$, independent of $f,a_1,a_2$, such that
\begin{equation}\label{eqn:kernel-mapping-bound}
    \norm{K f}_{\hi{\varepsilon}} \le C \paren{a_1+a_2}\norm{f}_{L^2(\R)}.
\end{equation}
\end{lemma}
The proof combines the Schur test with the difference characterization of
$H^\varepsilon(I)$ and is given in Appendix~\ref{app:boundary-integral-estimates}.

\begin{lemma}\label{LM-op-est}
Let $0\le\varepsilon<\alpha\le\frac12$ and
$\bm Z=(Z_1,Z_2)^T\in\ci{1,\alpha}$ satisfy $\abs{\bm Z}_*>0$.
Then, for every $r\in[0,1]$ and $f\in\thi{r}$,
\begin{align}
    \norm{L_{\bm Z}f}_{H^{r+\varepsilon}(I)}
    &\le C\abs{\bm Z}_*^{-2}\norm{\bm Z'}_{\ci{0}}[\bm Z']_{\ci{0,\alpha}}\norm{f}_{H^r(\R)}, \\
    \norm{M_{\bm Z,ij}f}_{H^{r+\varepsilon}(I)}
    &\le C\abs{\bm Z}_*^{-2}\norm{\bm Z'}_{\ci{0,\alpha}}[\bm Z']_{\ci{0,\alpha}}\norm{f}_{H^r(\R)}.
\end{align}
\end{lemma}
\begin{proof}
We prove the endpoint estimates $r=0$ and $r=1$ and then interpolate.
For $f\in\wt L^2(I)$, bounds~\eqref{eqn:LM-1}--\eqref{eqn:LM-4} show that
the kernels $\p_\eta\ell$ and $\p_\eta m_{ij}$ satisfy hypotheses~\eqref{eqn:kernel-pointwise-hypothesis}--\eqref{eqn:kernel-increment-hypothesis}.
Lemma~\ref{K-op} therefore yields
\begin{align}
    \norm{L_{\bm Z}f}_{\hi{\varepsilon}}
    &\le C\abs{\bm Z}_*^{-2}\norm{\bm Z'}_{\ci{0}}[\bm Z']_{\ci{0,\alpha}}\norm{f}_{L^2(\R)},\\
    \norm{M_{\bm Z,ij}f}_{\hi{\varepsilon}}
    &\le C\abs{\bm Z}_*^{-2}\norm{\bm Z'}_{\ci{0,\alpha}}[\bm Z']_{\ci{0,\alpha}}\norm{f}_{L^2(\R)}.
\end{align}

For the $r=1$ endpoint, let $f\in\wt H^1(I)$. Integration by parts gives
\begin{align}
    \p_s(L_{\bm Z}f)(s)
    &=-\int_I\p_s\ell(s,\eta;\bm Z)\p_\eta f(\eta)\,d\eta,\\
    \p_s(M_{\bm Z,ij}f)(s)
    &=-\int_I\p_sm_{ij}(s,\eta;\bm Z)\p_\eta f(\eta)\,d\eta.
\end{align}
The boundary terms vanish identically because $f\in\wt H^1(I)=H_0^1(I)$.
Applying Lemmas~\ref{K-op} and~\ref{LM-kernel-est} to the kernels on the right-hand side, and combining the resulting derivative estimates with the $r=0$ bounds, yields
\begin{align}
    \norm{L_{\bm Z}f}_{\hi{1+\varepsilon}}
    &\le C\abs{\bm Z}_*^{-2}\norm{\bm Z'}_{\ci{0}}[\bm Z']_{\ci{0,\alpha}}\norm{f}_{H^1(\R)},\\
    \norm{M_{\bm Z,ij}f}_{\hi{1+\varepsilon}}
    &\le C\abs{\bm Z}_*^{-2}\norm{\bm Z'}_{\ci{0,\alpha}}[\bm Z']_{\ci{0,\alpha}}\norm{f}_{H^1(\R)}.
\end{align}
Finally, the standard interpolation identities
$\paren{\wt L^2(I),\wt H^1(I)}_{r,2}=\wt H^r(I)$ and
$\paren{H^\varepsilon(I),H^{1+\varepsilon}(I)}_{r,2}
=H^{r+\varepsilon}(I)$ establish the estimate for all intermediate $r\in[0,1]$.
\end{proof}

\begin{lemma}\label{lem:LM-Lip}
Let $0\le\varepsilon<\alpha\le\frac12$ and
$\bm X,\bm Y\in\ci{1,\alpha}$ satisfy
$\norm{\bm X'}_{\ci{0,\alpha}},\norm{\bm Y'}_{\ci{0,\alpha}}\le M$ and
$\abs{\rho\bm X+(1-\rho)\bm Y}_*\ge m$ for every $\rho\in[0,1]$, where
$M,m>0$.
Then, for every $f\in\thi{s}$ with $s\in[0,1]$,
\begin{align}
    \norm{L_{\bm X}f-L_{\bm Y}f}_{\hi{s+\varepsilon}}
    &\le C_{M,m}
    \norm{\bm X'-\bm Y'}_{\ci{0,\alpha}}\norm{f}_{\hr{s}},
    \label{eqn:Lop-Lip}\\
    \norm{M_{\bm X,ij}f-M_{\bm Y,ij}f}_{\hi{s+\varepsilon}}
    &\le C_{M,m}
    \norm{\bm X'-\bm Y'}_{\ci{0,\alpha}}\norm{f}_{\hr{s}}.
    \label{eqn:Mop-Lip}
\end{align}
Here $L_{\bm Z}$ and $M_{\bm Z,ij}$ are defined in
\eqref{eqn:F-def} and~\eqref{eqn:M-def}.
\end{lemma}

The proof is given in Appendix~\ref{app:boundary-integral-estimates}.

We now specialize these geometric mapping and difference bounds to the reconstructed filament curve $\xk$, obtaining the corresponding operator estimates for $S_\kappa$.

\begin{lemma}\label{lem:Sop-est}
Let $0\le\varepsilon<\delta\le\frac12$ and
$\kappa\in\thi{-\frac12+\delta}$ satisfy $\abs{\kappa}_*>0$.
Then, for every $s\in[0,1]$ and $f\in\thi{s}$,
\begin{equation}\label{eqn:Sop-bound}
    \norm{S_{\kappa,ij}f}_{\hi{s+\varepsilon}}
    \le C\abs{\kappa}_*^{-2}
    \norm{\kappa}_{\hr{-\frac12+\delta}}
    \paren{1+\norm{\kappa}_{\hr{-\frac12+\delta}}}
    \norm{f}_{\hr{s}}.
\end{equation}
\end{lemma}
\begin{proof}
Set $\alpha=\frac12(\varepsilon+\delta)$, so that
$\varepsilon<\alpha<\delta$. The reconstruction estimate
\eqref{eqn:Xk-holder} ensures that $\xk\in\ci{1,\alpha}$ with
\begin{equation}
    [\xk']_{\ci{0,\alpha}} \le C\norm{\kappa}_{\hr{-\frac12+\delta}}, \quad \norm{\xk'}_{\ci{0}} \le
    C\paren{1+\norm{\kappa}_{\hr{-\frac12+\delta}}}.
\end{equation}
Since $S_{\kappa,ij}=\delta_{ij}L_{\xk}-M_{\xk,ij}$ by definition~\eqref{eqn:Sop-definition}, the claim follows directly from Lemma~\ref{LM-op-est}.
\end{proof}

\begin{lemma}\label{lem:Sop-Lip}
Let $0\le\varepsilon<\delta\le\frac12$ and
$\kappa_1,\kappa_2\in\thi{-\frac12+\delta}$ satisfy
$\norm{\kappa_i}_{\hr{-\frac12+\delta}}\le M$ for $i=1,2$. Assume also that
$\abs{\rho\xka{1}+(1-\rho)\xka{2}}_*\ge m$ for every $\rho\in[0,1]$.
Then, for every $s\in[0,1]$ and $f\in\thi{s}$,
\begin{equation}\label{eqn:Sop-difference}
    \norm{S_{\kappa_1,ij}f-S_{\kappa_2,ij}f}
    _{\hi{s+\varepsilon}}
    \le C_{M,m}
    \norm{\kappa_1-\kappa_2}_{\hr{-\frac12+\delta}}
    \norm{f}_{\hr{s}}.
\end{equation}
The dependence of $C_{M,m}$ on $\delta$ and $\varepsilon$ is suppressed.
\end{lemma}
\begin{proof}
Set $\alpha=\frac12(\varepsilon+\delta)$.
Lemma~\ref{lem:Xk-holder} with $m=0$ gives
\begin{equation}
    \norm{\bm X_{\kappa_i}'}_{\ci{0,\alpha}}
    \le C\paren{1+\norm{\kappa_i}_{\hr{-\frac12+\delta}}}
    \le C_{M,m}.
\end{equation}
Lemma~\ref{lem:LM-Lip} bounds the operator difference by the difference of the reconstructed tangents, which Lemma~\ref{lem:Xk-holder} controls in terms of $\kappa_1-\kappa_2$.
Therefore,
\begin{equation}
\begin{aligned}
    &\norm{S_{\kappa_1,ij}f-S_{\kappa_2,ij}f} _{\hi{s+\varepsilon}} \le \norm{L_{\bm X_{\kappa_1}}f-L_{\bm X_{\kappa_2}}f} _{\hi{s+\varepsilon}} \\
    &\quad+ \norm{M_{\bm X_{\kappa_1},ij}f -M_{\bm X_{\kappa_2},ij}f}_{\hi{s+\varepsilon}} \\
    &\quad\le C_{M,m} \norm{\bm X_{\kappa_1}'-\bm X_{\kappa_2}'}_{\ci{0,\alpha}} \norm{f}_{\hr{s}} \le C_{M,m}
    \norm{\kappa_1-\kappa_2}_{\hr{-\frac12+\delta}} \norm{f}_{\hr{s}}.
\end{aligned}
\end{equation}
\end{proof}

We now construct the finite-energy Stokes solution generated by a single-layer potential on the open arc and establish its energy dissipation identity.

\begin{lemma}[Stokes single-layer potential and energy identity]\label{lem:stokes-layer-open-arc}\label{lem:weak-stokes-open-arc}
Let $0<\delta\le\frac12$, $\kappa\in\mc U_\delta$, and
$\Gamma=\xk(I)$, with closure $\overline{\Gamma}=\xk(\overline I)$. For
$\bm g\in\wt H^{\frac12}(I)$, set
$\bm f=\p_s\bm g\in\wt H^{-\frac12}(I)$.
The Stokes single-layer potential with density $\bm f$ defines a solution
$(\bm u,p)$ on $\R^2\setminus\overline{\Gamma}$ with
$\bm u\in H^1_{\rm loc}(\R^2)$, $\grad\bm u,p\in L^2(\R^2)$, and
$\bm u(\bm x),p(\bm x)\to0$ as $|\bm x|\to\infty$.
For smooth supported data,
\begin{equation}\label{eqn:stokes-single-layer}
    \bm u(\bm x)=\int_I G(\bm x-\xk(\eta))\bm f(\eta)\,d\eta.
\end{equation}
The velocity has a common parametrized trace
$\bm v=(\bm u|_\Gamma)\circ\xk\in H^{\frac12}(I)$ on the two sides, and
$\jump{\Sigma\bm n}\circ\xk=\bm f$ on $I$ with the convention
in~\eqref{eqn:intro-stokes-interface}.
If $\norm{\kappa}_{\hr{-\frac12+\delta}}\le M$ and
$\abs{\kappa}_*\ge m>0$, then
\begin{equation}\label{eqn:stokes-layer-bound}
    \norm{\bm v}_{\hi{\frac12}}
    \le C_{M,m}\norm{\bm g}_{\hr{\frac12}}.
\end{equation}
Moreover,
\begin{equation}\label{eqn:weak-stokes-energy}
    2\int_{\R^2\setminus\overline{\Gamma}}\abs{\grad_S\bm u}^2\,d\bm x
    =\dual{\bm f}{\bm v}
    =-\dual{\p_s\bm v}{\bm g},
\end{equation}
where the pairings are the
$\wt H^{-\frac12}(I)$--$H^{\frac12}(I)$ and
$H^{-\frac12}(I)$--$\wt H^{\frac12}(I)$ dualities, respectively.
\end{lemma}
\begin{proof}
Fix exponents $0<\alpha<\beta<\delta$. Lemmas~\ref{tau-n-est} and~\ref{lem:Xk-holder}, together with the arc-chord bound $\abs{\kappa}_*>0$, ensure that $\overline{\Gamma}$ is a regular embedded $C^{1,\beta}$ arc.
Extend $\xk$ slightly past each endpoint along its tangent direction, choosing the extensions sufficiently short so that they remain mutually disjoint and intersect the original arc only at the respective attachment points.
The two resulting endpoints can then be joined by a smooth simple arc, disjoint from the extended arc except at its endpoints, with matching oriented endpoint tangents.
This connection can be routed through the exterior to a sufficiently large enclosing circle and closed outside that circle, yielding a regular $C^{1,\alpha}$ Jordan curve $\wt\Gamma$.
Choose a regular periodic parametrization $\wt{\bm X}$ on the torus $\mathbb T$ that coincides with $\xk$ on $\overline I$.
Periodizing the zero extension of $\bm g$ yields $\wt{\bm g}\in H^{\frac12}(\mathbb T)$.
We then define the traction density on $\wt\Gamma$ by
\begin{equation*}
    \dual{\wt{\bm f}}{\bm\phi}_{\wt\Gamma}
    =\dual{\p_s\wt{\bm g}}{\bm\phi\circ\wt{\bm X}}_{\mathbb T}.
\end{equation*}
Then $\wt{\bm f}\in H^{-\frac12}(\wt\Gamma)$ is supported on $\overline{\Gamma}$ and has vanishing total force resultant, since a periodic derivative annihilates constant functions.
For this fixed parametrization, its norm is bounded by $C\norm{\bm g}_{\hr{\frac12}}$.

The classical closed-curve potential theory of Sayas and Selgas~\cite[Propositions~7.1--7.3 and~9.2--9.3]{SayasSelgas2014} guarantees that $(\bm u,p)=(S_u\wt{\bm f},S_p\wt{\bm f})$ satisfies the stated Sobolev regularity, continuous velocity traces, jump relations, and single-layer representation.
Because the density $\wt{\bm f}$ vanishes identically on the auxiliary portion $\wt\Gamma\setminus\overline\Gamma$, the homogeneous Stokes equations hold weakly across the auxiliary curve; interior elliptic regularity therefore ensures that $(\bm u,p)$ is a smooth classical Stokes solution on $\R^2\setminus\overline{\Gamma}$.
The vanishing of the total resultant cancels the leading logarithmic far-field terms in the Stokeslet kernel \cite[(9.1) and~(9.4)--(9.5)]{SayasSelgas2014}, ensuring the far-field decay $\bm u(\bm x)=O(|\bm x|^{-1})$ and $|\grad\bm u(\bm x)|+|p(\bm x)|=O(|\bm x|^{-2})$.

Testing the weak formulation on the closed curve \cite[(7.6)]{SayasSelgas2014} with the velocity $\bm u$ itself yields $2\norm{\grad_S\bm u}_{L^2}^2 =\dual{\wt{\bm f}}{\bm u|_{\wt\Gamma}}_{\wt\Gamma}$.
Because $\wt{\bm f}$ is supported strictly in $\overline\Gamma$, this duality pairing reduces to the canonical interval pairing $\dual{\bm f}{\bm v}$.
Integrating by parts on the periodic domain without boundary terms yields $-\dual{\p_s\bm v}{\bm g}$, which proves the energy identity~\eqref{eqn:weak-stokes-energy}.

Finally, the trace estimate follows from the closed-curve mapping continuity, and Appendix~\ref{app:boundary-integral-estimates} establishes the uniformity of $C_{M,m}$ via a compactness argument on the family of admissible curves.
Continuity of differentiation and identity~\eqref{eqn:weak-stokes-energy} furthermore provide uniform bounds for $\p_s\bm v$ in $H^{-1/2}(I)$ and $\grad_S\bm u$ in $L^2$.
\end{proof}

\section{The tension equation}
\label{sec:tension}

This section addresses the tension determination problem~\eqref{eqn:ten-det}.
We establish the bounded invertibility of the nonlocal operator $T_\kappa:\thi{\frac12}\to\hi{-\frac12}$, prove uniform bounds for its inverse over admissible curvature classes, and deduce quantitative Sobolev and Lipschitz estimates for the tension $\sigma$.

\subsection{Invertibility of the tension operator}
\label{subsec:tension-existence}

The commutator and boundary integral estimates from Section~\ref{sec:prelim} ensure that the remainder $R_\kappa$ is a compact perturbation of the principal operator $\frac14\p_s\mc H$.
Combining this compactness with the Stokes single-layer energy identity~\eqref{eqn:weak-stokes-energy} to rule out nontrivial kernel elements, the Fredholm alternative delivers the bounded invertibility of $T_\kappa$ for each admissible curvature $\kappa\in\mc U_\delta$.

\begin{proposition}[Solvability of the tension equation]\label{prop:main-tension}
Fix $0<\delta\le\frac12$ and let $\kappa\in\mc U_\delta$. Then
\begin{equation}
    T_\kappa:
    \thi{\frac12}\longrightarrow\hi{-\frac12}
\end{equation}
has a bounded inverse. Consequently, for every
$F\in\hi{-\frac12}$, the equation
\begin{equation}
    T_\kappa\sigma=F
\end{equation}
has a unique solution $\sigma\in\thi{\frac12}$ satisfying
\begin{equation}\label{eqn:c-kappa}
    \norm{\sigma}_{\hr{\frac12}}
    \le C_\kappa\norm{F}_{\hi{-\frac12}},
    \quad
    C_\kappa
    \coloneqq\norm{T_\kappa^{-1}}_{
    \hi{-\frac12}\mapsto\thi{\frac12}}.
\end{equation}
\end{proposition}

The compactness of $R_\kappa$ rests on the following gain of regularity.

\begin{lemma}\label{R-est}
Let $0\le \varepsilon < \delta \le \frac{1}{2}$.
Let $\kappa\in\mc U_\delta$ and $\sigma\in \thi{\frac{1}{2}}$. Then
\begin{equation}\label{eqn:Rsigma}
\norm{R_\kappa \sigma}_{\hi{-\frac{1}{2}+\varepsilon}} \le C\abs{\kappa}_*^{-2} \paren{1+\norm{\kappa}_{\hr{-\frac{1}{2}+\delta}}}^4\norm{\kappa}_{\hr{-\frac{1}{2}+\delta}}\norm{\sigma}_{\hr{\frac{1}{2}}}.
\end{equation}
\end{lemma}
\begin{proof}
Set $\alpha=\frac{1}{2}\paren{\varepsilon+\delta}$.
Applying the Sobolev multiplication property (Lemma~\ref{point_mult_sobo}), the commutator estimate of Lemma~\ref{dsH_f_comm}, and the tangent frame bound of Lemma~\ref{tau-n-est}, we control the commutator contribution by
\begin{equation}
\begin{aligned}
    &\norm{\tk\cdot[\p_s\mc H, \tk]\sigma}_{\hi{-\frac{1}{2}+\varepsilon}} \le C\norm{\tk}_{\hi{\frac{1}{2}+\delta}} \paren{[\tk]_{\ci{0,\alpha}}+\norm{\p_s\tk}_{\hi{-\frac{1}{2}+\delta}}} \norm{\sigma}_{\hr{\frac{1}{2}}}\\
    & \le C\paren{1+\norm{\kappa}_{\hr{-\frac{1}{2}+\delta}}}^2 \norm{\kappa}_{\hr{-\frac{1}{2}+\delta}} \norm{\sigma}_{\hr{\frac{1}{2}}}.
\end{aligned}
\end{equation}
For the hydrodynamic remainder, Lemmas~\ref{point_mult_sobo} and~\ref{lem:Sop-est} yield
\begin{equation}
\begin{aligned}
    &\norm{\tk\cdot\p_s S_\kappa(\sigma\tk)}_{\hi{-\frac{1}{2}+\varepsilon}}
    \le C \norm{\tk}_{\hi{\frac{1}{2}+\delta}}^2 \abs{\kappa}_*^{-2}\paren{1+\norm{\kappa}_{\hr{-\frac{1}{2}+\delta}}}\norm{\kappa}_{\hr{-\frac{1}{2}+\delta}} \norm{\sigma}_{\hr{\frac{1}{2}}} \\
    &\le C\abs{\kappa}_*^{-2} \paren{1+\norm{\kappa}_{\hr{-\frac{1}{2}+\delta}}}^3\norm{\kappa}_{\hr{-\frac{1}{2}+\delta}}\norm{\sigma}_{\hr{\frac{1}{2}}}.
\end{aligned}
\end{equation}
Finally, the chord-arc constant satisfies
\begin{equation}
    \abs{\kappa}_*
    \le \norm{\tk}_{\ci{0}}
    \le C\norm{\tk}_{\hi{\frac{1}{2}+\delta}}
    \le C\paren{1+\norm{\kappa}_{\hr{-\frac{1}{2}+\delta}}},
\end{equation}
and combining these two estimates yields~\eqref{eqn:Rsigma}.
\end{proof}

\begin{lemma}\label{ten-unique}
Let $0<\delta\le\frac12$, let
$\kappa\in\mc U_\delta$, and let
$\sigma\in\thi{\frac{1}{2}}$ solve the homogeneous tension equation
\begin{equation}\label{eqn:tension-homogeneous}
    T_\kappa\sigma=0.
\end{equation}
Then $\sigma=0$. Thus the tension equation~\eqref{eqn:ten-det} has at
most one solution in $\thi{\frac{1}{2}}$ for each right-hand side
$F\in H^{-\frac{1}{2}}(I)$.
\end{lemma}
\begin{proof}
The homogeneous equation $T_\kappa\sigma=0$ represents vanishing tangential strain rate along the filament, which forces the viscous Stokes dissipation to vanish; single-layer coercivity on a closed extension then implies that the surface density must be zero.

Let $(\bm u,p)$ be the finite-energy Stokes solution generated by the single-layer density $\bm f=\p_s(\sigma\tk)$ as in Lemma~\ref{lem:stokes-layer-open-arc}, and set $\bm v=(\bm u|_\Gamma)\circ\xk$. For smooth $\sigma$, the Stokeslet splitting~\eqref{eqn:R-kernel}--\eqref{eqn:stokeslet-splitting} and the definition~\eqref{eqn:Rop-def} yield the identity
\begin{equation}\label{eqn:tension-tangential-id}
    \tk\cdot\p_s\bm v
    =
    -T_\kappa\sigma.
\end{equation}
Here the straight-filament contribution produces $-\frac{1}{4}\p_s\mc H\sigma$, while the geometric remainder and commutator terms reproduce $-R_\kappa\sigma$.
Both sides of~\eqref{eqn:tension-tangential-id} depend continuously on $\sigma\in\thi{\frac{1}{2}}$ with values in $\hi{-\frac{1}{2}}$, the left-hand side by Lemma~\ref{lem:stokes-layer-open-arc} and the right-hand side by Lemma~\ref{R-est}. A density argument therefore extends~\eqref{eqn:tension-tangential-id} to all $\sigma\in\thi{\frac{1}{2}}$.
The homogeneous equation~\eqref{eqn:tension-homogeneous} then gives
\begin{equation}
    \tk\cdot\p_s\bm v=0.
\end{equation}
Since $\bm g=\sigma\tk\in\wt H^{\frac12}(I)$ by Lemma~\ref{point_mult_sobo}, the energy identity~\eqref{eqn:weak-stokes-energy} implies
\begin{equation*}
    2\norm{\grad_S\bm u}_{L^2}^2
    =-\dual{\tk\cdot\p_s\bm v}{\sigma}=0.
\end{equation*}
Let $V$ denote the single-layer boundary operator on the closed curve extension $\wt\Gamma$ constructed in Lemma~\ref{lem:stokes-layer-open-arc}. In terms of the extended density $\wt{\bm f}$, the energy identity reads $\dual{\wt{\bm f}}{V\wt{\bm f}}_{\wt\Gamma}=0$.
Coercivity of $V$ modulo rigid motions and the normal field on closed curves~\cite[Proposition~7.4]{SayasSelgas2014} implies that $\wt{\bm f}=c\wt{\bm n}$ for some constant $c\in\R$. Because $\wt{\bm f}$ vanishes identically on the appended portion $\wt\Gamma\setminus\Gamma$, we must have $c=0$, and therefore $\bm f=\p_s(\sigma\tk)=0$ on $I$.

Consequently, the vector field $\sigma\tk$ is constant on $I$. The only constant vector whose extension by zero belongs to $H^{\frac12}(\R)$ is the zero vector, since any non-zero constant develops endpoint jumps that fail the $H^{1/2}$ threshold. Thus $\sigma\tk=0$, and because $\abs{\tk}=1$, we conclude that $\sigma=0$.

Finally, if $\sigma_1,\sigma_2\in\thi{\frac12}$ both satisfy~\eqref{eqn:ten-det} with the same forcing $F$, their difference solves~\eqref{eqn:tension-homogeneous}, establishing uniqueness.
\end{proof}

\begin{proof}[Proof of Proposition~\ref{prop:main-tension}]
Fix $\varepsilon\in(0,\delta)$. By Lemma~\ref{R-est}, $R_\kappa:\thi{\frac{1}{2}}\to\hi{-\frac{1}{2}+\varepsilon}$ is bounded. Since the Sobolev embedding $\hi{-\frac{1}{2}+\varepsilon}\hookrightarrow\hi{-\frac{1}{2}}$ is compact, the remainder
\begin{equation}
    R_\kappa:\thi{\frac{1}{2}}\to\hi{-\frac{1}{2}}
\end{equation}
is compact. In view of Lemma~\ref{ds-H-biject}, the straight-filament operator $\frac{1}{4}\p_s\mc H:\thi{\frac{1}{2}}\to\hi{-\frac{1}{2}}$ is an isomorphism. Consequently, the perturbation
\begin{equation}
    T_\kappa = \frac14\p_s\mc H + R_\kappa:
    \thi{\frac{1}{2}}\to\hi{-\frac{1}{2}}
\end{equation}
is Fredholm of index zero.
Because $\ker(T_\kappa)=\{0\}$ by Lemma~\ref{ten-unique}, the Fredholm alternative implies that $T_\kappa$ is bijective with a bounded inverse on $\hi{-\frac12}$. This establishes the unique solvability of~\eqref{eqn:ten-det} and yields the estimate
\begin{equation}\label{eqn:sigma-H12}
    \norm{\sigma}_{\hr{\frac{1}{2}}}
    \le C_\kappa \norm{F}_{\hi{-\frac{1}{2}}},
\end{equation}
where $C_\kappa<\infty$ is the operator norm defined in~\eqref{eqn:c-kappa}.
\end{proof}

We next record a Lipschitz bound for the remainder operator under variations of the curvature.

\begin{lemma}\label{R-Lip}
Let $0\le \varepsilon<\delta\le \frac{1}{2}$.
Let $\kappa_1,\kappa_2 \in \thi{-\frac{1}{2}+\delta}$ satisfy
\begin{equation*}
    \norm{\kappa_i}_{\hr{-\frac{1}{2}+\delta}}\le M,
    \quad i=1,2,
    \quad
    \abs{\rho \xka{1} + \paren{1-\rho}\xka{2}}_*\ge m,
    \quad \rho\in[0,1].
\end{equation*}
Then, for every $\sigma\in\thi{\frac12}$,
\begin{equation}\label{eqn:R-curvature-difference}
    \norm{\paren{R_{\kappa_1} - R_{\kappa_2}} \sigma}_{\hi{-\frac{1}{2}+\varepsilon}} \le C_{M,m} \norm{\kappa_1-\kappa_2}_{\hr{-\frac{1}{2}+\delta}} \norm{\sigma}_{\hr{\frac{1}{2}}}.
\end{equation}
\end{lemma}

The proof is given in Appendix~\ref{app:nonlinear-estimates}.

\begin{lemma}[Continuous dependence on the curvature]
\label{lem:inv-op-cont}
The map
\begin{equation}
    \kappa\longmapsto
    T_\kappa^{-1}
\end{equation}
is continuous from $\mc U_\delta$ into
$\mc L(\hi{-\frac{1}{2}},\thi{\frac{1}{2}})$.
\end{lemma}
\begin{proof}
By Lemma~\ref{lem:set-ok0}, the admissible set $\mc U_\delta$ is open. Fix $\kappa_0\in\mc U_\delta$ and define
\begin{equation}
    M
    =\norm{\kappa_0}_{\hr{-\frac{1}{2}+\delta}}
    +\frac{1}{2C_{J,\delta}}\abs{\kappa_0}_*,
    \quad
    m = \frac{1}{2}\abs{\kappa_0}_*.
\end{equation}
For any radius $0<r\le \abs{\kappa_0}_*/(2C_{J,\delta})$, consider the closed ball
\begin{equation}
    B_r(\kappa_0) = \{\kappa\in\thi{-\frac{1}{2}+\delta}:\norm{\kappa-\kappa_0}_{\hr{-\frac{1}{2}+\delta}}\le r\}\subset\mc U_\delta.
\end{equation}
For each $\kappa\in B_r(\kappa_0)$, the operator difference satisfies
\begin{equation}
    \Delta R=T_\kappa-T_{\kappa_0}
    =R_\kappa-R_{\kappa_0}.
\end{equation}
By Proposition~\ref{prop:main-tension}, $T_{\kappa_0}:\thi{\frac{1}{2}}\to\hi{-\frac{1}{2}}$ is boundedly invertible. We express the resolvent difference as
\begin{equation}
    T_\kappa^{-1}-T_{\kappa_0}^{-1} =\paren{\paren{I+T_{\kappa_0}^{-1}\Delta R}^{-1}-I} T_{\kappa_0}^{-1}
    =-\paren{I+T_{\kappa_0}^{-1}\Delta R}^{-1} T_{\kappa_0}^{-1}\Delta R T_{\kappa_0}^{-1}.
\end{equation}
Lemmas~\ref{lem:set-ok0} and~\ref{R-Lip} bound the perturbation by
\begin{equation}
    \norm{\Delta R}_{\thi{\frac{1}{2}}\mapsto\hi{-\frac{1}{2}}}
    \le K_0\norm{\kappa-\kappa_0}_{\hr{-\frac{1}{2}+\delta}}
    \le K_0 r,
\end{equation}
where $K_0=C_{M,m}$. Choosing $r>0$ small enough that
\begin{equation}
    \norm{T_{\kappa_0}^{-1}\Delta R}_{\thi{\frac{1}{2}}\mapsto\thi{\frac{1}{2}}} \le
    \norm{T_{\kappa_0}^{-1}}_{\hi{-\frac{1}{2}}\mapsto\thi{\frac{1}{2}}} \norm{\Delta R}_{\thi{\frac{1}{2}}\mapsto\hi{-\frac{1}{2}}} \le C_{\kappa_0}K_0 r < 1,
\end{equation}
the Neumann series yields the quantitative stability estimate
\begin{equation}
    \norm{T_\kappa^{-1}-T_{\kappa_0}^{-1}}_{\hi{-\frac{1}{2}}\to\thi{\frac{1}{2}}}
    \le \frac{C_{\kappa_0}^2K_0r}{1-C_{\kappa_0}K_0r}.
\end{equation}
This proves the continuity of $\kappa\mapsto T_\kappa^{-1}$ at $\kappa_0$.
\end{proof}

\begin{lemma}[Uniform bound for the inverse]\label{lem:c-kappa}
Let $0<\delta\le\frac{1}{2}$ and let $M,m>0$. Define
\begin{equation}
    \mc K_{M,m}
    \coloneqq
    \left\{
    \kappa\in\thi{-\frac{1}{2}+\delta}:
    \norm{\kappa}_{\hr{-\frac{1}{2}+\delta}}\le M,
    \abs{\kappa}_*\ge m
    \right\}.
\end{equation}
Then the inverse constants $C_\kappa$ defined in~\eqref{eqn:c-kappa} are
uniformly bounded on $\mc K_{M,m}$:
\begin{equation}\label{eqn:inverse-uniform-bound}
    \sup_{\kappa\in\mc K_{M,m}} C_\kappa \le C_{M,m}<\infty.
\end{equation}
\end{lemma}
\begin{proof}
Fix $\delta_0\in(0,\delta)$. By Rellich compactness, the embedding $\thi{-\frac{1}{2}+\delta}\hookrightarrow\thi{-\frac{1}{2}+\delta_0}$ is compact, so $\mc K_{M,m}$ is relatively compact in $\thi{-\frac{1}{2}+\delta_0}$.
Let $\kappa_n\in\mc K_{M,m}$ be a sequence converging to some $\kappa$ in $\thi{-\frac{1}{2}+\delta_0}$. Since $J:\thi{-\frac{1}{2}+\delta_0}\to C^0(I)$ is bounded, we have $J\kappa_n\to J\kappa$ and therefore $\bm X_{\kappa_n}'\to\bm X_{\kappa}'$ uniformly on $I$. In particular,
\begin{equation}
    \abs{\abs{\kappa_n}_*-\abs{\kappa}_*}
    \le \norm{\bm X_{\kappa_n}'-\bm X_{\kappa}'}_{\ci{0}}\to0.
\end{equation}
Passing to the limit gives $\abs{\kappa}_*\ge m$, showing that the closure of $\mc K_{M,m}$ in $\thi{-\frac{1}{2}+\delta_0}$ is entirely contained within the admissible set $\{\kappa\in\thi{-\frac{1}{2}+\delta_0}:\abs{\kappa}_*>0\}$.

Applying Lemma~\ref{lem:inv-op-cont} with $\delta_0$ in place of $\delta$, the mapping $\kappa\mapsto T_\kappa^{-1}$ is continuous on this compact closure in the operator norm of $\mc L(\hi{-\frac{1}{2}},\thi{\frac{1}{2}})$. Consequently, the continuous functional $\kappa\mapsto C_\kappa = \norm{T_\kappa^{-1}}$ achieves its maximum on the closure, establishing the uniform bound~\eqref{eqn:inverse-uniform-bound}.
\end{proof}

\subsection{Tension estimates}
\label{subsec:nonlinear-estimates}

With the bounded invertibility of $T_\kappa$ in hand, we now derive quantitative Sobolev bounds and Lipschitz stability estimates for the hydrodynamic forcing $F(\kappa)$ and the resulting tension $\sigma$.

\begin{lemma}\label{F-est}
Let $0\le\varepsilon<\delta\le\frac12$.
Let $\kappa\in\thi{\frac{3}{2}}$ with $\abs{\kappa}_*>0$. Then
\begin{equation}\label{eqn:forcing-bound}
\begin{aligned}
    \norm{F(\kappa)}_{\hi{-\frac{1}{2}+\varepsilon}} \le C \abs{\kappa}_*^{-2}\paren{1+\norm{\kappa}_{\hr{-\frac{1}{2}+\delta}}}^4\norm{\kappa}_{\hr{-\frac{1}{2}+\delta}}\norm{\kappa}_{\hr{\frac{3}{2}}}.
\end{aligned}
\end{equation}

\end{lemma}
\begin{proof}
Set $\alpha=\frac12(\varepsilon+\delta)$.
The commutator and Stokes-remainder estimates of Lemmas~\ref{point_mult_sobo}, \ref{dsH_f_comm}, and~\ref{lem:Sop-est} apply directly to the two components of~\eqref{eqn:F-kappa}. Using $\norm{\p_s\kappa}_{\hr{\frac12}}\le\norm{\kappa}_{\hr{\frac32}}$ together with the chord-arc bound $\abs{\kappa}_*\le C(1+\norm{\kappa}_{\hr{-\frac{1}{2}+\delta}})$, we obtain
\begin{equation}
\begin{aligned}
    \norm{F(\kappa)}_{\hi{-\frac{1}{2}+\varepsilon}} &\le C\paren{[\tk]_{\ci{0,\alpha}}+\norm{\p_s\tk}_{\hi{-\frac{1}{2}+\delta}}}\norm{\nk}_{\hi{\frac{1}{2}+\delta}}\norm{\p_s\kappa}_{\hr{\frac{1}{2}}} \\
    & \quad + C\abs{\kappa}_*^{-2}\norm{\tk}_{\ci{0,\alpha}}[\tk]_{\ci{0,\alpha}}\norm{\tk}_{\hi{\frac{1}{2}+\delta}}\norm{\nk}_{\hi{\frac{1}{2}+\delta}}\norm{\p_s\kappa}_{\hr{\frac{1}{2}}} \\
    & \le C\abs{\kappa}_*^{-2}\paren{1+\norm{\kappa}_{\hr{-\frac{1}{2}+\delta}}}^4\norm{\kappa}_{\hr{-\frac{1}{2}+\delta}}\norm{\kappa}_{\hr{\frac{3}{2}}}.
\end{aligned}
\end{equation}
\end{proof}

\begin{lemma}[Tension estimate]\label{sigma-est-F}
Let $0<\delta\le\frac12$.
Let $\kappa\in\thi{\frac{3}{2}}$ with $\abs{\kappa}_*>0$, and let
$F(\kappa)$ be given by~\eqref{eqn:F-kappa}. There exists a unique solution
$\sigma\in\thi{\frac{1}{2}}$ to~\eqref{eqn:ten-det} satisfying
\begin{equation}\label{eqn:tension-bound}
    \norm{\sigma}_{\hr{\frac{1}{2}}} \le C C_\kappa \abs{\kappa}_*^{-2}\paren{1+\norm{\kappa}_{\hr{-\frac{1}{2}+\delta}}}^4\norm{\kappa}_{\hr{-\frac{1}{2}+\delta}}\norm{\kappa}_{\hr{\frac{3}{2}}},
\end{equation}
where $C_\kappa < \infty$ is given in~\eqref{eqn:c-kappa}.
\end{lemma}
\begin{proof}
Specializing Lemma~\ref{F-est} to $\varepsilon=0$ yields
\begin{equation}
    \norm{F(\kappa)}_{\hi{-\frac12}}
    \le
    C\abs{\kappa}_*^{-2}
    \paren{1+\norm{\kappa}_{\hr{-\frac12+\delta}}}^4
    \norm{\kappa}_{\hr{-\frac12+\delta}}
    \norm{\kappa}_{\hr{\frac32}} .
\end{equation}
By Proposition~\ref{prop:main-tension}, the tension equation~\eqref{eqn:ten-det} possesses a unique solution $\sigma = T_\kappa^{-1}F(\kappa)\in\thi{\frac{1}{2}}$. Applying the inverse bound~\eqref{eqn:sigma-H12},
\begin{equation}
    \norm{\sigma}_{\hr{\frac12}}
    \le C_\kappa\norm{F(\kappa)}_{\hi{-\frac12}},
\end{equation}
which immediately delivers~\eqref{eqn:tension-bound}.
\end{proof}

\begin{lemma}[Higher regularity of tension]\label{sigma-est-F-higher}
Let $0 < \varepsilon<\delta\le \frac{1}{2}$.
Let $\kappa\in\thi{\frac{3}{2}}$, $\abs{\kappa}_* > 0$, and let $\sigma$ be the
unique solution to~\eqref{eqn:ten-det}. Then $\sigma\in \thi{\frac{1}{2}+\varepsilon}$.
If $\norm{\kappa}_{\hr{-\frac12+\delta}}\le M$ and $\abs{\kappa}_*\ge m>0$, then
\begin{equation}\label{eqn:tension-higher-bound}
\begin{aligned}
    \norm{\sigma}_{\hr{\frac{1}{2}+\varepsilon}}
    \le C_{M,m}(1+C_\kappa)\norm{\kappa}_{\hr{\frac{3}{2}}}.
\end{aligned}
\end{equation}
\end{lemma}
\begin{proof}
We gain regularity by comparing the principal operator $\p_s\mc H_{\R}$ with the elliptic multiplier $A_1$, whose difference is smoothing of order $-1$, and invoking the endpoint estimate of Lemma~\ref{lem:A-orderk-interval}.
Recall that $A_1$ is the Fourier multiplier on $\R$ with symbol
\begin{equation}
    a_1(\xi) = (1+\xi^2)^{\frac{1}{2}}.
\end{equation}
Since $\sigma\in\thi{\frac12}$ has support in $\overline I$, Lemma~\ref{lem:A-orderk-interval}(1) with $k=1$ yields the a priori bound
\begin{equation}\label{eqn:sigma-A-priori}
    \norm{\sigma}_{\hr{\frac{1}{2}+\varepsilon}}
    \le C\left(
    \norm{p_I A_1\sigma}_{\hi{-\frac{1}{2}+\varepsilon}}
    +\norm{\sigma}_{\hr{\frac{1}{2}}}
    \right).
\end{equation}
Rewriting the tension equation~\eqref{eqn:ten-det} as $\p_s\mc H\sigma = -4R_\kappa\sigma + 4F(\kappa)$, we decompose the restriction to $I$ as
\begin{equation}\label{eqn:sigma-principal-decomposition}
    p_I A_1\sigma
    =p_I\paren{A_1-\p_s\mc H_{\R}}\sigma
    -4R_\kappa\sigma+4F(\kappa).
\end{equation}
The symbol of $A_1-\p_s\mc H_{\R}$ satisfies
\begin{equation}
    (1+\xi^2)^{\frac{1}{2}}-\abs{\xi} = \frac{1}{\sqrt{1+\xi^2}+\abs{\xi}} \sim \abs{\xi}^{-1},
    \quad \text{as } \abs{\xi}\to \infty.
\end{equation}
Hence $A_1-\p_s\mc H_{\R}$ is a pseudodifferential operator of order $-1$, so that
\begin{equation}\label{eqn:sigma-lower-order-bound}
    \norm{p_I\paren{A_1-\p_s\mc H_{\R}}\sigma}_{\hi{-\frac{1}{2}+\varepsilon}}
    \le C\norm{\sigma}_{\hr{-\frac{3}{2}+\varepsilon}}
    \le C\norm{\sigma}_{\hr{\frac{1}{2}}} .
\end{equation}
Applying~\eqref{eqn:sigma-lower-order-bound} and the remainder estimate of Lemma~\ref{R-est} to the first two terms on the right-hand side of~\eqref{eqn:sigma-principal-decomposition}, we obtain
\begin{equation}
    \norm{p_I A_1\sigma}_{\hi{-\frac{1}{2}+\varepsilon}} \le{} C_{M,m}\norm{\sigma}_{\hr{\frac{1}{2}}} +
    C\norm{F(\kappa)}_{\hi{-\frac{1}{2}+\varepsilon}}.
\end{equation}
Combining this bound with the forcing estimate (Lemma~\ref{F-est}), the baseline tension estimate (Lemma~\ref{sigma-est-F}), and~\eqref{eqn:sigma-A-priori}, we conclude that
\begin{equation}
    \norm{\sigma}_{\hr{\frac{1}{2}+\varepsilon}}
    \le C_{M,m}(1+C_\kappa)\norm{\kappa}_{\hr{\frac{3}{2}}}.
\end{equation}
\end{proof}


We next establish Lipschitz continuity of the hydrodynamic forcing with respect to the curvature.

\begin{lemma}\label{F-Lip-kappa}
Let $0\le \varepsilon<\delta\le \frac{1}{2}$.
Let $\kappa_1,\kappa_2 \in \thi{\frac{3}{2}}$ satisfy
\begin{equation*}
    \norm{\kappa_i}_{\hr{-\frac{1}{2}+\delta}}\le M,
    \quad i=1,2,
    \quad
    \abs{\rho \xka{1} + \paren{1-\rho}\xka{2}}_*\ge m,
    \quad \rho\in[0,1].
\end{equation*}
Then
\begin{equation}\label{eqn:forcing-difference}
\begin{aligned}
    \norm{F(\kappa_1)-F(\kappa_2)}_{\hi{-\frac{1}{2}+\varepsilon}} &\le C_{M,m}
    \norm{\kappa_1-\kappa_2}_{\hr{-\frac{1}{2}+\delta}} \times \paren{ \norm{\kappa_1}_{\hr{\frac{3}{2}}}
    +\norm{\kappa_2}_{\hr{\frac{3}{2}}} } \\
    &\quad+ C_{M,m}M \norm{\kappa_1-\kappa_2}_{\hr{\frac{3}{2}}}.
\end{aligned}
\end{equation}
\end{lemma}

The proof is provided in Appendix~\ref{app:nonlinear-estimates}.

\begin{lemma}[Lipschitz continuity of tension with respect to curvature]\label{sigma-Lip-F}
Let $0\le \varepsilon<\delta\le \frac{1}{2}$.
Let $\kappa_1,\kappa_2 \in \thi{\frac{3}{2}}$ satisfy
\begin{equation*}
    \norm{\kappa_i}_{\hr{-\frac{1}{2}+\delta}}\le M,
    \quad i=1,2,
    \quad
    \abs{\rho \xka{1} + \paren{1-\rho}\xka{2}}_*\ge m,
    \quad \rho\in[0,1].
\end{equation*}
Set $C'=(1+C_{\kappa_1})(1+C_{\kappa_2})$.
Let $\sigma_i$ be the solution to~\eqref{eqn:ten-det} with right-hand side
$F(\kappa_i)$, $i=1,2$.
Then
\begin{equation}\label{eqn:tension-difference}
\begin{aligned}
    \norm{\sigma_1 - \sigma_2}_{\hr{\frac{1}{2}}} &\le C_{M,m}C'
    \norm{\kappa_1-\kappa_2}_{\hr{-\frac{1}{2}+\delta}} \times \paren{ \norm{\kappa_1}_{\hr{\frac{3}{2}}}
    +\norm{\kappa_2}_{\hr{\frac{3}{2}}} } \\
    &\quad+ C_{M,m}C'M \norm{\kappa_1-\kappa_2}_{\hr{\frac{3}{2}}}.
\end{aligned}
\end{equation}
\end{lemma}
\begin{proof}
Subtracting the tension equations $T_{\kappa_i}\sigma_i = F(\kappa_i)$ for $i=1,2$ gives
\begin{equation}\label{eqn:tension-difference-equation}
    T_{\kappa_1}\paren{\sigma_1-\sigma_2}
    =F(\kappa_1)-F(\kappa_2)
    -\paren{R_{\kappa_1}-R_{\kappa_2}}\sigma_2.
\end{equation}
Applying the bounded invertibility estimate~\eqref{eqn:sigma-H12} for $T_{\kappa_1}$ and estimating the right-hand side using the remainder difference bound~\eqref{eqn:R-curvature-difference}, the forcing variation~\eqref{eqn:forcing-difference}, and the baseline tension bound~\eqref{eqn:tension-bound}, we obtain
\begin{equation}
\begin{aligned}
    &\norm{\sigma_1-\sigma_2}_{\hr{\frac{1}{2}}} \le C_{M,m}C'
    \norm{\kappa_1-\kappa_2}_{\hr{-\frac{1}{2}+\delta}} \paren{ \norm{\kappa_1}_{\hr{\frac{3}{2}}}
    +\norm{\kappa_2}_{\hr{\frac{3}{2}}} } \\
    &\quad+ C_{M,m}C'M \norm{\kappa_1-\kappa_2}_{\hr{\frac{3}{2}}} + C_{M,m}C' M
    \norm{\kappa_1-\kappa_2}_{\hr{-\frac{1}{2}+\delta}} \norm{\kappa_2}_{\hr{\frac{3}{2}}} \\
    &\le C_{M,m}C' \norm{\kappa_1-\kappa_2}_{\hr{-\frac{1}{2}+\delta}} \paren{
    \norm{\kappa_1}_{\hr{\frac{3}{2}}} +\norm{\kappa_2}_{\hr{\frac{3}{2}}} } + C_{M,m}C'M
    \norm{\kappa_1-\kappa_2}_{\hr{\frac{3}{2}}}.
\end{aligned}
\end{equation}
\end{proof}

\section{Local well-posedness}
\label{sec:local-theory}

We establish local well-posedness for the geometric evolution~\eqref{eqn:mild-def} via a Banach contraction argument in the time-weighted trajectory space $X_T^{(\delta)}$. Consistency under temporal restarts and variations of the regularity exponent then yields a unique maximal solution and ensures continuous dependence on initial data.

\subsection{Local well-posedness at fixed regularity}
\label{subsec:local-proof}
We first derive growth and Lipschitz estimates for the curvature nonlinearity $N(\kappa)$ on the stable arc-chord neighborhood $\mc B(\kappa_0)$, ensuring that the Duhamel operator acts as a strict contraction.

\begin{lemma}\label{lem:N-est}
Let $0 \le \varepsilon < \delta \le \frac{1}{2}$.
Let $M,m>0$ and let $\kappa\in\thi{\frac{3}{2}}$ satisfy
$\norm{\kappa}_{\hr{-\frac12+\delta}}\le M$ and $\abs{\kappa}_*\ge m$. Then
\begin{equation}\label{eqn:N-pointwise-bound}
    \norm{N(\kappa)}_{\hi{-\frac{1}{2}+\varepsilon}}
    \le C_{M,m}
    \norm{\kappa}_{\hr{-\frac{1}{2}+\delta}}
    \norm{\kappa}_{\hr{\frac{3}{2}}} .
\end{equation}
Let $\kappa_1,\kappa_2\in\thi{\frac{3}{2}}$ satisfy
$\norm{\kappa_i}_{\hr{-\frac{1}{2}+\delta}}\le M$ for $i=1,2$. Suppose
$\abs{\rho\xka{1}+\paren{1-\rho}\xka{2}}_*\ge m$ whenever
$\rho\in[0,1]$.
Then
\begin{equation}\label{eqn:N-pointwise-difference}
\begin{aligned}
    \norm{N(\kappa_1) - N(\kappa_2)}_{\hi{-\frac{1}{2}+\varepsilon}} &\le C_{M,m} \norm{\kappa_1-\kappa_2}_{\hr{-\frac{1}{2}+\delta}} \paren{\norm{\kappa_1}_{\hr{\frac{3}{2}}}+\norm{\kappa_2}_{\hr{\frac{3}{2}}}} \\
    & + C_{M,m}M \norm{\kappa_1-\kappa_2}_{\hr{\frac{3}{2}}}.
\end{aligned}
\end{equation}
\end{lemma}
\begin{proof}[Proof of the first estimate]
We first control the Frenet frame and tension, and then bound the force density primitive $\bm g$ entering $N(\kappa)$.
Set $\alpha=\frac{1}{2}\paren{\varepsilon+\delta}$, so that
$\varepsilon<\alpha<\delta$, and write
$K=\norm{\kappa}_{\hr{-\frac{1}{2}+\delta}}\le M$.
Lemma~\ref{lem:c-kappa} bounds the inverse constant $C_\kappa$ by $C_{M,m}$.
Lemmas~\ref{tau-n-est} and~\ref{lem:Xk-holder} give
\begin{equation}
    [\nk]_{\ci{0,\alpha}}
    +\norm{\p_s\nk}_{\hi{-\frac{1}{2}+\delta}}
    \le C_{M,m}K,\quad
    \norm{\nk}_{\hi{\frac{1}{2}+\delta}}
    +\norm{\tk}_{\hi{\frac{1}{2}+\delta}}
    \le C_{M,m}.
\end{equation}
Let $\sigma=\sigma(\kappa)$ and set
$\bm g=\p_s\kappa\nk+\sigma\tk$. We use fixed bounded Sobolev
extensions of the frame fields to $\mathbb R$; $\bm g$ remains supported
in $\overline I$. Lemmas~\ref{lem:wt-der}, \ref{point_mult_sobo},
and~\ref{sigma-est-F} give $\bm g\in\thi{\frac12}$ and
\begin{equation}
    \norm{\bm g}_{\hr{\frac12}}
    \le C_{M,m}
    \paren{
        \norm{\kappa}_{\hr{\frac{3}{2}}}
        +\norm{\sigma}_{\hr{\frac{1}{2}}}
    }
    \le C_{M,m}\norm{\kappa}_{\hr{\frac32}}.
\end{equation}
Lemma~\ref{dsH_f_comm} applies with this choice of $\alpha$, including when
$\varepsilon=0$. For the Stokes remainder, apply Lemma~\ref{lem:Sop-est}
with $s=\frac12$, then differentiate and use Lemma~\ref{point_mult_sobo}
to multiply by $\nk$:
\begin{equation}
\begin{aligned}
    \norm{[\p_s\mc H,\nk]\cdot\bm g}_{\hi{-\frac12+\varepsilon}}
    &\le C_{M,m}K\norm{\bm g}_{\hr{\frac12}},\\
    \norm{\nk\cdot\p_s(S_\kappa\bm g)}_{\hi{-\frac12+\varepsilon}}
    &\le C_{M,m}K\norm{\bm g}_{\hr{\frac12}}.
\end{aligned}
\end{equation}
Substituting into~\eqref{eqn:N-def} gives
\begin{equation}
    \norm{N(\kappa)}_{\hi{-\frac12+\varepsilon}}
    \le C_{M,m}K\norm{\kappa}_{\hr{\frac32}}.
\end{equation}
The proof of~\eqref{eqn:N-pointwise-difference} is given in
Appendix~\ref{app:nonlinear-estimates}.
\end{proof}

Fix $0<\delta\le\frac12$, and recall the admissible configuration set $\mc U_\delta$
and the chord-arc neighborhood $\mc B(\kappa_0)$ from
Subsection~\ref{subsec:notation}. The third-order linear semigroup gains
$2-\delta$ derivatives from $\thi{-1/2+\delta}$ to $\thi{3/2}$ with an integrable temporal
singularity $t^{-(2-\delta)/3}$, which motivates the time-weighted space below.
Define the function space
\begin{equation}\label{eqn:XT-intro}
    X_T
    \coloneqq
    \left\{
    \kappa\in C\paren{[0,T];\thi{-\frac12+\delta}}
    \cap C\paren{(0,T];\thi{\frac32}}:
    \norm{\kappa}_{X_T}<\infty
    \right\},
\end{equation}
with the time-weighted norm
\begin{equation}\label{eqn:XT-norm-intro}
    \norm{\kappa}_{X_T}
    =
    \sup_{0\le t\le T}
    \norm{\kappa(t)}_{\hr{-\frac12+\delta}}
    +
    \sup_{0<t\le T}
    t^{\frac{2-\delta}{3}}
    \norm{\kappa(t)}_{\hr{\frac32}}.
\end{equation}
We write $X_T^{(\delta)}$ when specifying the exponent.

Given $\kappa_0\in\mc U_\delta$, a \emph{mild solution} on $[0,T]$ is a function $\kappa\in X_T^{(\delta)}$ such that $\kappa(0)=\kappa_0$, $\kappa(t)\in\mc U_\delta$ for every $t\in[0,T]$,
and $\kappa=\Psi_{\kappa_0}\kappa$ on $[0,T]$, with the time integral
in~\eqref{eqn:mild-def} understood as a Bochner integral in $H^{-3/2}(I)$.
A function $\kappa$ is a global mild solution if its restriction to every
finite interval $[0,T]$ is a mild solution.

\begin{proposition}[Local well-posedness for a fixed exponent]
\label{prop:local-fixed-exponent}
Let $\kappa_0\in\mc U_\delta$. There exist $T,r_0,C>0$, depending only on
$\delta,\kappa_0$, such that for every $\vph_0\in\thi{-\frac12+\delta}$ with
$\norm{\vph_0-\kappa_0}_{\hr{-\frac12+\delta}}\le r_0$, there exists a unique mild
solution $\kappa(\cdot;\vph_0)\in X_T$. Moreover,
\begin{equation}\label{eqn:local-fixed-lipschitz}
    \norm{\kappa(\cdot;\vph_0)-\kappa(\cdot;\kappa_0)}_{X_T}
    \le C\norm{\vph_0-\kappa_0}_{\hr{-\frac12+\delta}}.
\end{equation}
\end{proposition}
\begin{proof}
Fix $\varepsilon=\delta/2$ and write $a=(2-\delta)/3$.

\emph{1. Choice of the fixed-point set and preservation of the geometry.}
We center the fixed-point ball on the reference linear trajectory $e^{t\Lambda}\kappa_0$.
Strong continuity of the semigroup keeps this evolution close to $\kappa_0$ for short times.
The radius of the ball is chosen so that every trajectory remains
in $\mc B(\kappa_0)$, where Lemma~\ref{lem:set-ok0} provides uniform geometric control. Define
\begin{equation}\label{eqn:OT-intro}
    \mc O_T
    \coloneqq
    \left\{
    \kappa\in X_T:
    \norm{\kappa-e^{\cdot\Lambda}\kappa_0}_{X_T}
    \le \frac{\abs{\kappa_0}_*}{4C_{J,\delta}}
    \right\}.
\end{equation}
For the neighborhood~\eqref{eqn:Bkappa0}, the strong
continuity property~\eqref{eqn:lambda-semigroup-cont} provides a time $T_1\in(0,1]$ such that
\begin{equation}\label{eqn:local-linear-closeness}
    \sup_{0\le t\le T_1}
    \norm{e^{t\Lambda}\kappa_0-\kappa_0}_{\hr{-\frac12+\delta}}
    \le \frac{\abs{\kappa_0}_*}{4C_{J,\delta}}.
\end{equation}
By~\eqref{eqn:OT-intro}, \eqref{eqn:XT-norm-intro}, and
\eqref{eqn:local-linear-closeness}, for any $0<T\le T_1$ and $\kappa\in\mc O_T$,
\begin{equation}
    \begin{aligned}
        \norm{\kappa(t)-\kappa_0}_{\hr{-\frac12+\delta}}
        &\le \norm{\kappa(t)-e^{t\Lambda}\kappa_0}_{\hr{-\frac12+\delta}}
        +\norm{e^{t\Lambda}\kappa_0-\kappa_0}_{\hr{-\frac12+\delta}}\\
        &\le \frac{\abs{\kappa_0}_*}{4C_{J,\delta}}
        +\frac{\abs{\kappa_0}_*}{4C_{J,\delta}}
        =\frac{\abs{\kappa_0}_*}{2C_{J,\delta}},
        \quad 0\le t\le T,
    \end{aligned}
\end{equation}
so that $\kappa(t)\in\mc B(\kappa_0)$ on $[0,T]$. We set the low-norm and arc-chord bounds
\begin{equation}
    M=\norm{\kappa_0}_{\hr{-\frac12+\delta}}
      +\frac{\abs{\kappa_0}_*}{2C_{J,\delta}},
    \quad
    m=\frac12\abs{\kappa_0}_*.
\end{equation}
Lemma~\ref{lem:set-ok0} then guarantees that
$\norm{\kappa(t)}_{\hr{-\frac12+\delta}}\le M$ and
$\abs{\kappa(t)}_*\ge m$.
Moreover, because $\kappa_i(t)\in\mc B(\kappa_0)$ for any
$\kappa_1,\kappa_2\in\mc O_T$, \eqref{eqn:neighborhood-geometry} provides the uniform arc-chord lower bound
for the convex combinations $\rho\xka{1}(t)+(1-\rho)\xka{2}(t)$ for all $\rho\in[0,1]$.
This verifies the geometric hypothesis of the difference estimate
\eqref{eqn:N-pointwise-difference} in Lemma~\ref{lem:N-est}.

\emph{2. Nonlinear and Duhamel estimates.}
From~\eqref{eqn:lambda-semigroup-est} and~\eqref{eqn:OT-intro}, every $\kappa\in\mc O_T$ satisfies
\begin{equation}\label{eqn:local-ball-norm}
    \norm{\kappa}_{X_T}
    \le \norm{e^{\cdot\Lambda}\kappa_0}_{X_T}
       +\frac{\abs{\kappa_0}_*}{4C_{J,\delta}}
    \le C(1+M).
\end{equation}
Evaluating~\eqref{eqn:N-pointwise-bound} at each positive time and taking advantage of
\eqref{eqn:XT-norm-intro} and~\eqref{eqn:local-ball-norm}, we obtain
\begin{equation*}
    t^a\norm{N(\kappa(t))}_{\hi{-\frac12+\varepsilon}}
    \le C_{M,m}M\,t^a\norm{\kappa(t)}_{\hr{\frac32}}
    \le C_{M,m}.
\end{equation*}
For the difference, both terms in~\eqref{eqn:N-pointwise-difference}
scale with the same temporal weight:
\begin{equation*}
\begin{aligned}
    t^a\norm{N(\kappa_1(t))-N(\kappa_2(t))}_{\hi{-\frac12+\varepsilon}}
    &\le C_{M,m}\norm{\kappa_1(t)-\kappa_2(t)}_{\hr{-\frac12+\delta}}
        \sum_{i=1}^2 t^a\norm{\kappa_i(t)}_{\hr{\frac32}}\\
    &\quad+C_{M,m}M t^a
        \norm{\kappa_1(t)-\kappa_2(t)}_{\hr{\frac32}}\\
    &\le C_{M,m}\norm{\kappa_1-\kappa_2}_{X_T}.
\end{aligned}
\end{equation*}
Uniformly for $T\le T_1$, we therefore conclude that
\begin{equation}\label{eqn:local-nonlinear-bounds}
\begin{aligned}
    \sup_{0<t\le T}t^a
    \norm{N(\kappa(t))}_{\hi{-\frac12+\varepsilon}}
    &\le C_{M,m},\\
    \sup_{0<t\le T}t^a
    \norm{N(\kappa_1(t))-N(\kappa_2(t))}_{\hi{-\frac12+\varepsilon}}
    &\le C_{M,m}\norm{\kappa_1-\kappa_2}_{X_T}.
\end{aligned}
\end{equation}
Furthermore, \eqref{eqn:N-pointwise-difference} implies that
$N(\kappa)\in C((0,T];\hi{-\frac12+\varepsilon})$, since
$\kappa\in C((0,T];\thi{\frac32})$.

\emph{Duhamel estimate.}
For the integral operator in~\eqref{eqn:mild-def}, we prove that
$\mc Df\in X_T$ and $\norm{\mc Df}_{X_T}\le CT^{\varepsilon/3}F_T$
whenever $f\in C((0,T];\hi{-\frac12+\varepsilon})$ satisfies
\begin{equation*}
    F_T=\sup_{0<r\le T}r^a\norm{f(r)}_{\hi{-\frac12+\varepsilon}}<\infty.
\end{equation*}
Here
\begin{equation}\label{eqn:local-D-definition}
    (\mc Df)(t)
    =\int_0^t e^{(t-r)\Lambda}\p_s f(r)\,dr.
\end{equation}
Set $X=H^{-3/2}(I)$ and $D(\Lambda)=\thi{3/2}$.
The derivative satisfies
$\p_s f\in H^{-3/2+\varepsilon}(I)=(X,D(\Lambda))_{\varepsilon/3,2}$;
the target spaces $\thi{-1/2+\delta}$ and $\thi{3/2}$ correspond to the interpolation indices
$(1+\delta)/3$ and $1$, respectively, with zero extension being bounded at the lower index.
Proposition~\ref{prop:lambda-semigroup}, applied with $\alpha=\varepsilon/3$
and these target indices, yields via~\eqref{eqn:lambda-semigroup-est}:
\begin{equation*}
\begin{aligned}
    \norm{(\mc Df)(t)}_{\hr{-\frac12+\delta}}
    &\le C\int_0^t
    (t-r)^{-\frac{1+\delta-\varepsilon}{3}}
    \norm{f(r)}_{\hi{-\frac12+\varepsilon}}\,dr,\\
    \norm{(\mc Df)(t)}_{\hr{\frac32}}
    &\le C\int_0^t
    (t-r)^{-1+\frac{\varepsilon}{3}}
    \norm{f(r)}_{\hi{-\frac12+\varepsilon}}\,dr.
\end{aligned}
\end{equation*}
Using $\norm{f(r)}_{\hi{-\frac12+\varepsilon}}\le F_T r^{-a}$ and making the change of variable $u=r/t$, we obtain
\begin{equation*}
\begin{aligned}
    \int_0^t (t-r)^{-\frac{1+\delta-\varepsilon}{3}}r^{-a}\,dr
    &=t^{\frac{\varepsilon}{3}}
    \int_0^1 (1-u)^{-\frac{1+\delta-\varepsilon}{3}}u^{-a}\,du,\\
    t^a\int_0^t (t-r)^{-1+\frac{\varepsilon}{3}}r^{-a}\,dr
    &=t^{\frac{\varepsilon}{3}}
    \int_0^1 (1-u)^{-1+\frac{\varepsilon}{3}}u^{-a}\,du.
\end{aligned}
\end{equation*}
Both beta-type integrals converge: the condition $a<1$ controls the integrable singularity at $u=0$, while $(1+\delta-\varepsilon)/3<1$ and $\varepsilon/3>0$ ensure integrability at $u=1$.
Consequently, both components of the $X_T$-norm~\eqref{eqn:XT-norm-intro} gain a factor of $t^{\varepsilon/3}$.
Taking suprema produces the contraction factor $T^{\varepsilon/3}$:
\begin{equation}\label{eqn:local-duhamel-est}
    \norm{\mc Df}_{X_T}
    \le C T^{\frac{\varepsilon}{3}}
    \sup_{0<t\le T}t^a\norm{f(t)}_{\hi{-\frac12+\varepsilon}}.
\end{equation}
The constant $C$ depends on $\delta,\varepsilon$, but is independent of $T\le T_1$.

Beyond~\eqref{eqn:local-duhamel-est}, membership in $X_T$ requires verifying the continuity requirements in~\eqref{eqn:XT-intro}.
For continuity in $\thi{3/2}$ at any $t_0\in(0,T]$, fix $0<\eta<t_0/4$. For $t\in(0,T]$ with $\abs{t-t_0}<\eta/2$, we decompose the integral at $t_0-\eta$:
\begin{equation*}
    (\mc Df)(t)
    =\int_0^{t_0-\eta}e^{(t-r)\Lambda}\p_s f(r)\,dr
    +\int_{t_0-\eta}^t e^{(t-r)\Lambda}\p_s f(r)\,dr.
\end{equation*}
In the first integral, $t-r\ge\eta/2>0$, so analyticity of the semigroup combined with the integrable bound $\norm{f(r)}_{\hi{-\frac12+\varepsilon}}\le F_T r^{-a}$ gives continuity into $\thi{\frac32}$ by dominated convergence.
For the second integral, since $r\ge t_0-\eta>t_0/2$,
\begin{equation*}
\begin{aligned}
    \norm{\int_{t_0-\eta}^t e^{(t-r)\Lambda}\p_s f(r)\,dr}_{\hr{\frac32}}
    &\le C F_T t_0^{-a}
    \int_{t_0-\eta}^t(t-r)^{-1+\frac{\varepsilon}{3}}\,dr\\
    &\le C F_T t_0^{-a}\eta^{\frac{\varepsilon}{3}}.
\end{aligned}
\end{equation*}
This bound also holds at $t_0$. Sending $t\to t_0$ and subsequently letting $\eta\downarrow0$ proves continuity in $\thi{3/2}$ on $(0,T]$.
Continuity at zero follows from the lower-norm estimate
\begin{equation*}
    \norm{(\mc Df)(t)}_{\hr{-\frac12+\delta}}
    \le C F_T t^{\frac{\varepsilon}{3}}\longrightarrow0
    \quad\text{as }t\downarrow0.
\end{equation*}
With $(\mc Df)(0)=0$, we have $\mc Df\in X_T$.

\emph{3. Contraction and dependence on the initial data.}
For $\kappa,\kappa_1,\kappa_2\in\mc O_T$, definitions
\eqref{eqn:mild-def} and~\eqref{eqn:local-D-definition} give
\begin{equation*}
    \begin{aligned}
        \Psi_{\vph_0}\kappa-e^{\cdot\Lambda}\kappa_0
        &=e^{\cdot\Lambda}(\vph_0-\kappa_0)+\mc D N(\kappa),\\
        \Psi_{\vph_0}\kappa_1-\Psi_{\vph_0}\kappa_2
        &=\mc D\paren{N(\kappa_1)-N(\kappa_2)}.
    \end{aligned}
\end{equation*}
Equation~\eqref{eqn:lambda-semigroup-est}, together with
\eqref{eqn:local-nonlinear-bounds} and~\eqref{eqn:local-duhamel-est}, gives
\begin{equation}\label{eqn:local-map-bounds}
\begin{aligned}
    \norm{\Psi_{\vph_0}\kappa-e^{\cdot\Lambda}\kappa_0}_{X_T}
    &\le C_{M,m}\paren{
        T^{\frac{\varepsilon}{3}}
        +\norm{\vph_0-\kappa_0}_{\hr{-\frac12+\delta}}
    },\\
    \norm{\Psi_{\vph_0}\kappa_1-\Psi_{\vph_0}\kappa_2}_{X_T}
    &\le C_{M,m}T^{\frac{\varepsilon}{3}}
        \norm{\kappa_1-\kappa_2}_{X_T}.
\end{aligned}
\end{equation}
Choose $T\le T_1$ and $r_0>0$ sufficiently small such that
\begin{equation}\label{eqn:local-T-r-choice}
    C_{M,m}\paren{T^{\frac{\varepsilon}{3}}+r_0}
    \le\frac{\abs{\kappa_0}_*}{4C_{J,\delta}},
    \quad
    C_{M,m}T^{\frac{\varepsilon}{3}}\le\frac12.
\end{equation}
By~\eqref{eqn:local-map-bounds}, the first condition in
\eqref{eqn:local-T-r-choice} ensures that $\Psi_{\vph_0}$ maps the closed ball
$\mc O_T$ from~\eqref{eqn:OT-intro} into itself, while the second guarantees that $\Psi_{\vph_0}$ is a strict contraction with contraction constant at most $1/2$.
Both conditions hold for every $\vph_0$ in the stated neighborhood with the same
choice of $T$. The Banach fixed-point theorem
thus gives a unique mild solution $\kappa(\cdot;\vph_0)\in\mc O_T$.

For two fixed points with initial data $\vph_0,\psi_0$ in this neighborhood,
subtracting the corresponding Duhamel formulas gives
\begin{equation*}
\begin{aligned}
    \kappa(\cdot;\vph_0)-\kappa(\cdot;\psi_0)
    &=e^{\cdot\Lambda}(\vph_0-\psi_0)\\
    &\quad+\mc D\paren{N(\kappa(\cdot;\vph_0))-N(\kappa(\cdot;\psi_0))}.
\end{aligned}
\end{equation*}
The contraction estimate~\eqref{eqn:local-map-bounds}--\eqref{eqn:local-T-r-choice} yields
\begin{equation}
\begin{aligned}
    &\norm{\kappa(\cdot;\vph_0)-\kappa(\cdot;\psi_0)}_{X_T}\\
    &\quad\le C\norm{\vph_0-\psi_0}_{\hr{-\frac12+\delta}}
    +\frac12
    \norm{\kappa(\cdot;\vph_0)-\kappa(\cdot;\psi_0)}_{X_T}.
\end{aligned}
\end{equation}
Absorbing the last term and taking $\psi_0=\kappa_0$ proves
\eqref{eqn:local-fixed-lipschitz}.

\emph{4. Uniqueness among all mild solutions.}
The contraction gives uniqueness only within the ball $\mc O_T$
in~\eqref{eqn:OT-intro}. We now show that every mild solution with the same
initial datum enters such a ball on a sufficiently short interval.
Let $\kappa_1,\kappa_2\in X_T$ be mild solutions with the same initial
value $f\in\mc U_\delta$. By continuity, both lie in
$\mc B(f)$ near zero.
Equations~\eqref{eqn:N-pointwise-bound} and
\eqref{eqn:local-duhamel-est} give, for sufficiently small $h>0$,
\begin{equation}
    \norm{\kappa_i-e^{\cdot\Lambda}f}_{X_h}
    \le C h^{\frac{\varepsilon}{3}}\norm{\kappa_i}_{X_h}
    \longrightarrow0,
    \quad i=1,2,
\end{equation}
where $C$ is independent of $h$. The convergence follows from
$\norm{\kappa_i}_{X_h}\le\norm{\kappa_i}_{X_T}$. For small $h$, both
solutions lie in the ball~\eqref{eqn:OT-intro} with $\kappa_0=f$ and $T=h$, so they
agree near zero. Set
\begin{equation*}
    T_*=\sup\{t\in[0,T]:\kappa_1=\kappa_2\text{ on }[0,t]\}.
\end{equation*}
If $T_*<T$, then $T_*>0$ and the semigroup property applied to
\eqref{eqn:mild-def} gives the same formula after shifting time by $T_*$,
with initial value
$g=\kappa_1(T_*)=\kappa_2(T_*)$.
The definition of a mild solution gives $g\in\mc U_\delta$, so $g$
is admissible initial data for the local uniqueness argument.
For $0<h\le T-T_*$, the shifted
solutions $u_i(t)=\kappa_i(T_*+t)$ belong to $X_h$. Indeed, they are
continuous in $\thi{\frac32}$ up to the new initial time, and
\begin{equation*}
    \sup_{0<t\le h}t^a\norm{u_i(t)}_{\hr{\frac32}}
    \le h^a\sup_{T_*\le r\le T_*+h}\norm{\kappa_i(r)}_{\hr{\frac32}}
    <\infty.
\end{equation*}
The uniqueness argument at zero, applied with initial value $g$,
gives $u_1=u_2$ for small $h$. Their equality extends beyond
$T_*$, contradicting the definition of $T_*$. Hence $\kappa_1=\kappa_2$ on $[0,T]$.
\end{proof}

\subsection{Maximal lifespan and continuous dependence}
\label{subsec:lifespan}

To extend local solutions to a maximal lifespan and establish global-in-time
properties, we analyze the consistency of mild solutions under temporal restarts
and across differing regularity exponents $\delta$.

\begin{proposition}[Consistency of local solutions]
\label{lem:restart-uniqueness}
Fix $0<\delta\le\frac12$. Let $\kappa$ be a mild solution on $[0,T]$.
For $t_0\in[0,T)$, let $\vph$ be the local fixed-point solution from
Proposition~\ref{prop:local-fixed-exponent} with initial value
$\vph(0)=\kappa(t_0)$.
Then
\begin{equation}\label{eqn:restart-identity}
    \vph(t)=\kappa(t_0+t)
\end{equation}
throughout their common interval of existence.

In addition, let $0<\delta_1\le\delta_2\le\frac12$, and let
$f\in\thi{-\frac12+\delta_2}$ satisfy $\abs{f}_*>0$. If $\vph_i$ is the
local fixed-point solution with initial value $f$ constructed at exponent
$\delta_i$, then $\vph_1=\vph_2$ throughout their common interval of
existence.
\end{proposition}
\begin{proof}
Set $\kappa^{[t_0]}(t)=\kappa(t_0+t)$. The semigroup property and the Duhamel
formula~\eqref{eqn:mild-def} for $\kappa$ yield
\begin{equation}\label{eqn:shifted-duhamel}
    \kappa^{[t_0]}(t)
    =e^{t\Lambda}\kappa(t_0)
    +\int_0^t e^{(t-r)\Lambda}
    \p_s\paren{N(\kappa^{[t_0]}(r))}\,dr.
\end{equation}
The shifted solution inherits the uniform arc-chord bound of $\kappa$ and belongs to
$X_h$ on each common interval $[0,h]$; when $t_0>0$, this membership follows
from positive-time continuity in $\thi{\frac32}$. Uniqueness in
Proposition~\ref{prop:local-fixed-exponent} therefore establishes~\eqref{eqn:restart-identity}.

For consistency across exponents, let $a_i=(2-\delta_i)/3$ for $i=1,2$.
The continuous Sobolev embedding
$\thi{-\frac12+\delta_2}\hookrightarrow\thi{-\frac12+\delta_1}$ and the identity
\begin{equation}
    t^{a_1}\norm{\vph_2(t)}_{\hr{\frac32}}
    =t^{a_1-a_2}
    \paren{t^{a_2}\norm{\vph_2(t)}_{\hr{\frac32}}}
\end{equation}
imply that $X_T^{(\delta_2)}\hookrightarrow X_T^{(\delta_1)}$ on every finite
interval, because $a_1\ge a_2$. Since both $\vph_1$ and $\vph_2$ satisfy the Duhamel equation~\eqref{eqn:mild-def} with initial data $f$, uniqueness at exponent $\delta_1$ guaranteed by
Proposition~\ref{prop:local-fixed-exponent} proves that $\vph_1=\vph_2$.
\end{proof}

\begin{proof}[Proof of Theorem~\ref{thm:main-local}]
The proof proceeds in two steps: we first demonstrate that the maximal lifespan
is independent of the regularity exponent $\delta$, and then propagate the local
Lipschitz stability along the reference trajectory up to any compact time $T<\tau_{\mathrm{max}}(\kappa_0)$.

\emph{Independence of the lifespan.}
For each admissible exponent $\delta\in(0,1/2]$,
Propositions~\ref{prop:local-fixed-exponent} and~\ref{lem:restart-uniqueness}
provide a unique maximal mild solution $\kappa_\delta$ with lifespan $\tau_\delta\in(0,\infty]$.
If $0<\delta_1\le\delta_2\le\frac12$ and
$\kappa_0\in\thi{-\frac12+\delta_2}$, the consistency property ensures that
$\kappa_{\delta_1}$ and $\kappa_{\delta_2}$ agree on their common interval of existence.
Because the $\delta_2$-solution class is embedded in the
$\delta_1$-solution class, we have $\tau_{\delta_2}\le\tau_{\delta_1}$.
If this inequality were strict, choose numbers
$0<a<\tau_{\delta_2}<b<\tau_{\delta_1}$.
On $[0,a]$, equality with $\kappa_{\delta_2}$ gives initial continuity
and the weighted bound at exponent $\delta_2$. On $[a,b]$, positive-time
regularity $\kappa_{\delta_1}\in C([a,b];\thi{3/2})$ supplies the same
bounds, since the temporal weight is bounded above and bounded away from zero on $[a,b]$.
Hence $\kappa_{\delta_1}\in X_b^{(\delta_2)}$. Because it satisfies the Duhamel identity
\eqref{eqn:mild-def} and the arc-chord condition in~\eqref{eqn:U-intro},
it extends $\kappa_{\delta_2}$ beyond its maximal lifespan $\tau_{\delta_2}$, a
contradiction. Consequently, all admissible exponents have the
same lifespan $\tau_{\mathrm{max}}(\kappa_0)$.
Choosing any finite $T=T(\kappa_0)\in(0,\tau_{\mathrm{max}}(\kappa_0))$, this choice is independent of $\delta$, and the solution satisfies
\eqref{eqn:main-local-class} on $[0,T]$ for all admissible exponents.

\emph{Lipschitz dependence on the initial data.}
Fix $\delta\in(0,1/2]$ and set $a=(2-\delta)/3$. We propagate the local Lipschitz
estimate to time $T$ by covering the compact reference trajectory
$\mc K=\{\kappa(t):0\le t\le T\}\subset\mc U_\delta$ with finitely
many local well-posedness neighborhoods and chaining the local Lipschitz bounds across restarts.
At each point $g\in\mc K$, Proposition~\ref{prop:local-fixed-exponent}
provides an open ball of initial data with a uniform existence time, and its proof
yields a Lipschitz estimate for any pair of initial data in that ball.
The balls with half these radii still cover $\mc K$. Extracting a finite
subcover, let $r>0$ be the minimum of these half-radii, bounding
the admissible initial perturbation. Choose a uniform restart time $0<h\le T$
such that $2h$ does not exceed any local existence time, and let $L\ge1$ be a common
Lipschitz constant. Then
\begin{equation}\label{eqn:uniform-local-lipschitz}
    \norm{\kappa(\cdot;q)-\kappa(\cdot;g)}_{X_{2h}^{(\delta)}}
    \le L\norm{q-g}_{\hr{-\frac12+\delta}}
\end{equation}
whenever $g\in\mc K$, $q\in\thi{-\frac12+\delta}$, and
$\norm{q-g}_{\hr{-\frac12+\delta}}\le r$.
Indeed, since $g$ belongs to one of the smaller balls, both $q$ and $g$ lie
in the corresponding ball of full radius. The parameters $h,r,L$ may
depend on $\delta,\kappa_0,T$.

We construct the nearby solution by successive restarts at intervals of length $h$.
Set $n=\lceil T/h\rceil$ and choose $r_0=rL^{-n}$, so that compounding up to
$n$ factors of $L$ keeps the accumulated perturbation within radius $r$. For any initial datum satisfying
\begin{equation*}
    d_0=\norm{\vph_0-\kappa_0}_{\hr{-\frac12+\delta}}\le r_0,
\end{equation*}
let $u(t)=\kappa(t;\vph_0)$ be the local solution. At each restart,
\eqref{eqn:uniform-local-lipschitz} amplifies the difference
in $\hr{-\frac12+\delta}$ by at most a factor $L$. Thus, inductively,
\begin{equation*}
    \norm{u(jh)-\kappa(jh)}_{\hr{-\frac12+\delta}}
    \le L^j d_0\le r,
    \quad j=0,\ldots,n-1.
\end{equation*}
Each local solution exists for time $2h$, and uniqueness
identifies the solutions on their overlaps. They therefore extend $u$
to $[0,T]$, with
\begin{equation*}
    \sup_{0\le t\le T}\norm{u(t)-\kappa(t)}_{\hr{-\frac12+\delta}}
    \le L^n d_0.
\end{equation*}

A restart at $t_0$ controls the high norm with the local weight
$(t-t_0)^a$, whereas $X_T$ uses $t^a$. On $(0,h]$,
\eqref{eqn:uniform-local-lipschitz} gives
\begin{equation*}
    \sup_{0<t\le h}t^a\norm{u(t)-\kappa(t)}_{\hr{\frac32}}
    \le Ld_0.
\end{equation*}
On $[jh,(j+1)h]\cap[0,T]$, use the solution restarted one interval
earlier, at $(j-1)h$. Then $h\le t-(j-1)h\le2h$, so the ratio of
global to local weights is at most $(T/h)^a$:
\begin{equation*}
\begin{aligned}
    t^a\norm{u(t)-\kappa(t)}_{\hr{\frac32}}
    &\le \left(\frac{t}{t-(j-1)h}\right)^a L^j d_0\\
    &\le \left(\frac{T}{h}\right)^a L^j d_0,
    \quad j=1,\ldots,n-1.
\end{aligned}
\end{equation*}
The lower-norm estimate and these weighted estimates give
\begin{equation*}
    \norm{u-\kappa}_{X_T^{(\delta)}}
    \le C_{\delta,\kappa_0,T}
    \norm{\vph_0-\kappa_0}_{\hr{-\frac12+\delta}}.
\end{equation*}
Since $T$ depends only on $\kappa_0$, this proves
\eqref{eqn:main-local-lipschitz} with $r_0,C$ depending only on
$\delta,\kappa_0$. Local uniqueness gives uniqueness on the whole interval.
\end{proof}

\section{Positive-time regularity and reconstruction}
\label{sec:regularity-theory}

This section establishes the positive-time regularity and endpoint asymptotics of Theorem~\ref{thm:main-regularity} and carries out the physical and hydrodynamic reconstruction of Corollary~\ref{cor:main-reconstruction}.

\subsection{Sobolev regularity and endpoint asymptotics}
\label{subsec:regularity-proof}

Parabolic smoothing of the mild curvature trajectory requires higher-order Sobolev control on the non-local nonlinearity $N(\kappa)$.
\begin{lemma}\label{lem:N-higher-est}
Let $0<\varepsilon<\delta< \frac{1}{2}$.
Let $L,m>0$ and let $\kappa\in\thi{\frac{3}{2}+\delta}$ satisfy
$\norm{\kappa}_{\hr{\frac32+\delta}}\le L$ and $\abs{\kappa}_*\ge m$.
Let $\sigma$ be the solution of~\eqref{eqn:ten-det}. Then
$N(\kappa)\in \hi{\frac{1}{2}+\varepsilon}$ and
\begin{equation}\label{eqn:N-higher-est}
    \norm{N(\kappa)}_{\hi{\frac{1}{2}+\varepsilon}}
    \le C_{L,m}
    \norm{\kappa}_{\hr{\frac{3}{2}+\delta}} .
\end{equation}
\end{lemma}
\begin{proof}
Since the lower Sobolev norm is also bounded by $L$, the uniform coercivity estimate~\eqref{eqn:inverse-uniform-bound} yields $C_\kappa\le C_{L,m}$.
Throughout this proof, constants depending only on $L$ and the fixed exponents are denoted by $C_L$.
Fix $\alpha=(\varepsilon+\delta)/2$ and consider the hydrodynamic force primitive
$\bm q=\p_s\kappa\,\nk+\sigma\tk$.
We first establish higher Sobolev control for the Frenet frame and the tension.
By~\eqref{eqn:theta-Hs+1},
$\norm{\theta_\kappa}_{\hi{\frac{5}{2}+\delta}}
\le C\norm{\kappa}_{\hr{\frac{3}{2}+\delta}}$.
Set $r=\frac12+\delta\in(\frac12,1)$ and write $\theta=\theta_\kappa$.
The chain rule bounds $\sin\theta$ and $\cos\theta$ in $H^1(I)$,
and hence in $H^r(I)$, by $C_L$. Since $H^r(I)$ is an algebra by
Lemma~\ref{point_mult_sobo} and $\theta',\theta''\in H^r(I)$, the identities
\begin{align*}
    \p_s^2(\sin\theta)
    &=\cos\theta\,\theta''-\sin\theta\,(\theta')^2,\\
    \p_s^2(\cos\theta)
    &=-\sin\theta\,\theta''-\cos\theta\,(\theta')^2
\end{align*}
bound both second derivatives in $H^r(I)$ by $C_L$. Thus both
compositions are bounded in $H^{r+2}(I)$, giving
\begin{equation}\label{eqn:tau-n-high}
    \norm{\tk}_{\hi{\frac{5}{2}+\delta}}
    +\norm{\nk}_{\hi{\frac{5}{2}+\delta}}
    \le C_L .
\end{equation}
Lemma~\ref{lem:Xk-holder}, with $m=2$, also gives
\begin{equation}\label{eqn:X-high-C}
    \norm{\xk}_{\ci{3,\alpha}}
    \le C_L .
\end{equation}
The higher tension estimate
\eqref{eqn:tension-higher-bound} gives
\begin{equation}\label{eqn:sigma-high-for-N}
    \norm{\sigma}_{\hr{\frac{1}{2}+\varepsilon}}
    \le C_{L,m}
        \norm{\kappa}_{\hr{\frac{3}{2}+\delta}} .
\end{equation}

For $\bm q$, Lemma~\ref{point_mult_sobo} and
\eqref{eqn:tau-n-high}--\eqref{eqn:sigma-high-for-N} give
\begin{equation}\label{eqn:q-high}
\begin{aligned}
    &\norm{\bm q}_{\hr{\frac{1}{2}+\varepsilon}} \le \norm{\p_s\kappa\,\nk}_{\hr{\frac{1}{2}+\varepsilon}} +
    \norm{\sigma\tk}_{\hr{\frac{1}{2}+\varepsilon}} \\
    &\le C\norm{\p_s\kappa}_{\hr{\frac{1}{2}+\delta}} \norm{\nk}_{\hi{\frac{1}{2}+\varepsilon}} +
    C\norm{\sigma}_{\hr{\frac{1}{2}+\varepsilon}} \norm{\tk}_{\hi{\frac{1}{2}+\varepsilon}} \le C_{L,m}
    \norm{\kappa}_{\hr{\frac{3}{2}+\delta}} .
\end{aligned}
\end{equation}
In particular,
\begin{equation}\label{eqn:q-L2-high}
    \norm{\bm q}_{L^2(\R)}
    \le C_{L,m}
        \norm{\kappa}_{\hr{\frac{3}{2}+\delta}} .
\end{equation}

For the commutator term in~\eqref{eqn:N-def} with $\sigma=\sigma(\kappa)$,
Lemma~\ref{lem:comm-psH}, \eqref{eqn:tau-n-high}, and~\eqref{eqn:q-high} give
\begin{equation}\label{eqn:N-comm-high}
    \norm{[\p_s\mc H,\nk]\cdot\bm q}_{\hi{\frac{1}{2}+\varepsilon}} \le C\norm{\nk}_{\hi{\frac{5}{2}+\delta}}
    \norm{\bm q}_{\hr{\frac{1}{2}+\varepsilon}} \le C_{L,m} \norm{\kappa}_{\hr{\frac{3}{2}+\delta}} .
\end{equation}

For the Stokes-remainder term, write
$r_{ij}(s,\eta;\kappa)=\phi_{ij}(V(s,\eta))$, where
$V(s,\eta)=\frac{\xk(s)-\xk(\eta)}{s-\eta}$ and
$\phi_{ij}(v)=\frac{1}{2}\log(\abs{v}^2)\delta_{ij}-\frac{v_i v_j}{|v|^2}$.
The arc-chord bound and $|\p_s\xk|=1$ give $m\le|V|\le1$, including
on the diagonal by continuity. Thus the divided differences stay in a
compact subset of $\R^2\setminus\{0\}$, where $\phi_{ij}$ is smooth.
Together with $\xk\in\ci{3,\alpha}$ from~\eqref{eqn:X-high-C} and
$\varepsilon<\alpha$, this verifies the hypotheses of
Lemma~\ref{lem:K-higher-order} with $k=3$. Applied componentwise, it gives
\begin{equation}\label{eqn:R-kernel-high}
    \norm{\int_I \p_\eta r(s,\eta;\kappa)\bm q(\eta)\,d\eta}_{\hi{2+\varepsilon}}
    \le C_{L,m}
        \norm{\bm q}_{L^2(\R)} .
\end{equation}
After differentiating once in $s$,
\begin{equation}\label{eqn:dR-kernel-high}
\begin{aligned}
    &\norm{\p_s\int_I \p_\eta r(s,\eta;\kappa)\bm q(\eta)\,d\eta}_{\hi{\frac{1}{2}+\varepsilon}} \le C\norm{\int_I
    \p_\eta r(s,\eta;\kappa)\bm q(\eta)\,d\eta}_{\hi{\frac{3}{2}+\varepsilon}}  \\
    &\le C\norm{\int_I \p_\eta r(s,\eta;\kappa)\bm q(\eta)\,d\eta}_{\hi{2+\varepsilon}} \le C_{L,m}
    \norm{\bm q}_{L^2(\R)} \le C_{L,m} \norm{\kappa}_{\hr{\frac{3}{2}+\delta}} ,
\end{aligned}
\end{equation}
where the last two inequalities use \eqref{eqn:X-high-C} and
\eqref{eqn:q-L2-high}. By~\eqref{eqn:Sop-definition},
\begin{equation}
    \p_s S_\kappa\bm q
    =
    \p_s\int_I\p_\eta r(s,\eta;\kappa)\bm q(\eta)\,d\eta .
\end{equation}
Multiplying by $\nk$ and using Lemma~\ref{point_mult_sobo},
\eqref{eqn:tau-n-high}, and \eqref{eqn:dR-kernel-high} gives
\begin{equation}\label{eqn:N-rem-high}
    \norm{\nk\cdot\p_s S_\kappa\bm q}_{\hi{\frac{1}{2}+\varepsilon}} \le C\norm{\nk}_{\hi{\frac{1}{2}+\varepsilon}}
    \norm{\p_s S_\kappa\bm q}_{\hi{\frac{1}{2}+\varepsilon}} \le C_{L,m} \norm{\kappa}_{\hr{\frac{3}{2}+\delta}} .
\end{equation}
Combining \eqref{eqn:N-comm-high} and \eqref{eqn:N-rem-high} in the formula
\eqref{eqn:N-def} proves~\eqref{eqn:N-higher-est}.
\end{proof}

\begin{proof}[Proof of Theorem~\ref{thm:main-regularity}]
Fix $0<\tau<T$ and $0\le\varepsilon<\frac12$, and set
$t_0=\tau/4$ and $t_1=\tau/2$.

\emph{1. Positive-time spatial Sobolev regularity.}
Restarting the evolution at a positive time provides initial curvature data in $\thi{3/2}$ and ensures a strictly positive, uniform arc-chord lower bound on the remaining compact interval. Choose intermediate exponents
\begin{equation*}
    \varepsilon<\eta<\eta_1<\frac12.
\end{equation*}
where $\eta$ represents an intermediate spatial gain and the condition $\eta_1>\eta$ ensures that the temporal Duhamel convolution kernel is integrable.
The continuity in~\eqref{eqn:XT-intro}, the condition $\kappa(t)\in\mc U_\delta$ from~\eqref{eqn:U-intro}, and~\eqref{eqn:arc-chord-C1-continuity} give finite bounds
\begin{equation*}
    B=\sup_{t_0\le r\le T}\norm{\kappa(r)}_{\hr{\frac32}},
    \quad m=\inf_{t_0\le r\le T}\abs{\kappa(r)}_*>0.
\end{equation*}
Constants in this step may depend on $B,m,T-t_0$ and the fixed
exponents, but are independent of $t,r$.
Choose an auxiliary exponent $\delta_1\in(\eta_1,\frac12]$.
The $\thi{3/2}$ bound controls the $\thi{-1/2+\delta_1}$ norm, so
Lemma~\ref{lem:N-est} applies uniformly with this auxiliary exponent:
\begin{equation}\label{eqn:pN-low}
    \p_sN(\kappa)\in
    L^\infty\paren{[t_0,T];\hi{-\frac32+\eta_1}}.
\end{equation}
For $t\in(t_0,T]$, \eqref{eqn:shifted-duhamel} reads
\begin{equation}\label{eqn:mild-reg}
    \kappa(t)=e^{(t-t_0)\Lambda}\kappa(t_0)
    +\int_{t_0}^t e^{(t-r)\Lambda}\p_sN(\kappa(r))\,dr.
\end{equation}
For forcing in $H^{-3/2+\eta_1}(I)$ and target $H^{-3/2+\eta}(I)$,
Proposition~\ref{prop:lambda-semigroup} with $n=1$ gives the exponent
$-1+(\eta_1-\eta)/3$ in~\eqref{eqn:lambda-semigroup-est}, while the
initial datum in $D(\Lambda)$ gives $-\eta/3$:
\begin{equation}\label{eqn:Lam-32}
\begin{aligned}
    \norm{\Lambda\kappa(t)}_{\hi{-\frac32+\eta}}
    &\le C(t-t_0)^{-\eta/3}\norm{\kappa(t_0)}_{\hr{\frac32}}\\
    &\quad+C\int_{t_0}^t
        (t-r)^{-1+(\eta_1-\eta)/3}
        \norm{\p_sN(\kappa(r))}_{\hi{-\frac32+\eta_1}}\,dr\\
    &\le C(t-t_0)^{-\eta/3}\norm{\kappa(t_0)}_{\hr{\frac32}}\\
    &\quad+C\int_{t_0}^t
        (t-r)^{-1+(\eta_1-\eta)/3}
        \norm{\kappa(r)}_{\hr{\frac32}}\,dr\\
    &\le C\paren{1+(t-t_0)^{-\eta/3}}
        \sup_{t_0\le r\le t}\norm{\kappa(r)}_{\hr{\frac32}}.
\end{aligned}
\end{equation}
The second inequality uses~\eqref{eqn:N-pointwise-bound}.

Recall that $A_3$, defined by~\eqref{eqn:orderk-symbol-factorization} with $k=3$, has symbol
$\langle\xi\rangle^3$. Write
$A_3\kappa=-4\Lambda\kappa+(A_3+4\Lambda)\kappa$.
The difference $A_3+4\Lambda$ is first order, hence lower order than
$A_3$. Its bound on $\thi{3/2}$ and the estimate for $\Lambda\kappa$
give $p_I A_3\kappa(t)\in H^{-3/2+\eta}(I)$.
Lemma~\ref{lem:A-orderk-interval}(1), with $k=3$, then yields the
supported Sobolev gain in~\eqref{eq:A-orderk-interval-est}:
\begin{equation}\label{eqn:k-est}
    \norm{\kappa(t)}_{\hr{\frac32+\eta}}
    \le C\paren{
        \norm{\Lambda\kappa(t)}_{\hi{-\frac32+\eta}}
        +\norm{\kappa(t)}_{\hr{\frac32}}}.
\end{equation}
Interpolating the uniform $\thi{3/2+\eta}$ bound on $[t_1,T]$ from
\eqref{eqn:Lam-32}--\eqref{eqn:k-est} with continuity in $\thi{3/2}$ gives
\begin{equation}\label{eqn:reconstruction-positive-time-sobolev}
    \kappa\in C\paren{[\tau,T];\thi{\frac32+\varepsilon}}.
\end{equation}
The case $\varepsilon=0$ is already part of the mild solution class.
Lemma~\ref{lem:theta-kappa} and Sobolev composition also give uniform
bounds for $\theta_\kappa$ in $H^{5/2+\eta}(I)$ and $\xk$ in
$H^{7/2+\eta}(I)$ on $[t_1,T]$. Hence the normalized curves are
uniformly bounded in $C^{3,\alpha}(\overline I)$ for $0<\alpha<\eta$.

\emph{2. Time differentiability.}
The graph norm on $D(\Lambda)=\thi{3/2}$ is equivalent to its
supported Sobolev norm, so $\kappa\in C([\tau,T];D(\Lambda))$.
We use the evolution equation to prove differentiability first in the
base space. For the semigroup argument, write
$E=H^{-3/2}(I)$, $S(t)=e^{t\Lambda}$, and
$g(t)=\p_sN(\kappa(t))$.
Choose $M>0$ with $\sup_{\tau\le t\le T}\norm{\kappa(t)}_{L^2(I)}\le M$.
Continuity in $\thi{3/2}$ gives lower-norm closeness and $C^1$ closeness
of the normalized curves. Their convex combinations retain a common
arc-chord bound by~\eqref{eqn:neighborhood-geometry}, so
\eqref{eqn:N-pointwise-difference} applies uniformly and gives
$g\in C([\tau,T];E)$.

For $\tau<t<T$ and $h>0$ with $t+h\le T$, \eqref{eqn:shifted-duhamel} gives
\begin{equation}\label{eqn:kappa-time-right-quotient}
    \frac{\kappa(t+h)-\kappa(t)}{h}
    =
    \frac{S(h)-\Id}{h}\kappa(t)
    +\frac1h\int_0^h S(h-r)g(t+r)\,dr.
\end{equation}
The first term converges to $\Lambda\kappa(t)$ in $E$ because
$\kappa(t)\in D(\Lambda)$. Strong continuity and local boundedness of
$S$ and continuity of $g$ imply that the second term tends to $g(t)$.
For the left difference quotient, write
\begin{equation}\label{eqn:kappa-time-left-quotient}
    \frac{\kappa(t)-\kappa(t-h)}{h}
    =
    \frac1h\int_0^h S(r)\Lambda\kappa(t-h)\,dr
    +\frac1h\int_0^h S(h-r)g(t-h+r)\,dr.
\end{equation}
Continuity of $\Lambda\kappa$ and $g$ in $E$ gives the same limit.
Thus $\p_t\kappa=\Lambda\kappa+g$ in $E$, with a continuous right-hand
side. At $\tau$ and $T$, the corresponding one-sided limits apply.

Since $\Lambda$ has order three on the zero extension of $\kappa$,
\begin{equation*}
    \norm{\Lambda\kappa(t)}_{\hi{-\frac32+\varepsilon}}
    \le C_\varepsilon
        \norm{\kappa(t)}_{\hr{\frac32+\varepsilon}}.
\end{equation*}
By \eqref{eqn:reconstruction-positive-time-sobolev}, $\Lambda\kappa$ is
also continuous in this space. Lemma~\ref{lem:N-est}, with
the auxiliary parameter $\delta_1=\frac12$, gives
\begin{equation*}
    \norm{g(t)}_{\hi{-\frac32+\varepsilon}}
    \le C_{\varepsilon,M,m}\norm{\kappa(t)}_{L^2(I)}
        \norm{\kappa(t)}_{\hr{\frac32}},
\end{equation*}
and~\eqref{eqn:N-pointwise-difference} gives continuity in the same space.
Consequently $\Lambda\kappa+g$ belongs to
$C([\tau,T];H^{-3/2+\varepsilon}(I))$, and
\begin{equation}\label{eqn:kappa-time-est}
    \norm{\p_t\kappa(t)}_{\hi{-\frac32+\varepsilon}}
    \le C_{\varepsilon,M,m}
        \norm{\kappa(t)}_{\hr{\frac32+\varepsilon}},
    \quad \tau\le t\le T.
\end{equation}
Integrating the identity in $E$ gives
\begin{equation}\label{eqn:kappa-time-integral}
    \kappa(t)-\kappa(\tau)
    =\int_\tau^t\paren{\Lambda\kappa(r)+g(r)}\,dr.
\end{equation}
Both sides lie in $H^{-3/2+\varepsilon}(I)$. The integral defines a
continuously differentiable primitive in this space, so
\begin{equation}\label{eqn:kappa-time-regularity}
    \kappa\in C^1\paren{[\tau,T];\hi{-\frac32+\varepsilon}},
    \quad
    \p_t\kappa=\Lambda\kappa+\p_sN(\kappa).
\end{equation}

\emph{3. Endpoint asymptotics.}
According to Lemma~\ref{lem:A-orderk-interval}(2) with $k=3$, the $d^{3/2}$
boundary expansion follows once
\begin{equation*}
    p_I A_3\kappa(t)\in H^{-1/2+\varepsilon}(I).
\end{equation*}
It therefore suffices to upgrade the bound on
$\Lambda\kappa$ from $H^{-3/2+\eta}(I)$ to $H^{-1/2+\varepsilon}(I)$.
Choose $\varepsilon<\rho<\eta$ to apply Lemma~\ref{lem:N-higher-est}
with the available curvature gain $\eta$ and nonlinear gain $\rho$;
the strict inequality $\rho>\varepsilon$ makes the Duhamel kernel integrable.
Equations~\eqref{eqn:Lam-32} and~\eqref{eqn:k-est} give
\begin{equation*}
    L=\sup_{t_1\le r\le T}
        \norm{\kappa(r)}_{\hr{\frac32+\eta}}<\infty.
\end{equation*}
Lemma~\ref{lem:N-higher-est}, with
$\delta=\eta$ and $\varepsilon=\rho$, gives
\begin{equation}
    N(\kappa(r))\in H^{1/2+\rho}(I),\quad t_1\le r\le T,
\end{equation}
with a uniform bound depending on $L,m$. Hence
\begin{equation}\label{eqn:pN-hi}
    \p_sN(\kappa)\in
    L^\infty\paren{[t_1,T];\hi{-\frac12+\rho}}.
\end{equation}
For $t\in[\tau,T]$, restart the mild formulation at $t_1$:
\begin{equation}
    \kappa(t)=e^{(t-t_1)\Lambda}\kappa(t_1)
        +\int_{t_1}^t e^{(t-r)\Lambda}\p_sN(\kappa(r))\,dr.
\end{equation}
For forcing in $H^{-1/2+\rho}(I)$ and target
$H^{-1/2+\varepsilon}(I)$, Proposition~\ref{prop:lambda-semigroup}
with $n=1$ gives the exponent $-1+(\rho-\varepsilon)/3$ in~\eqref{eqn:lambda-semigroup-est}:
\begin{equation}\label{eqn:Lam-12}
\begin{aligned}
    \norm{\Lambda\kappa(t)}_{\hi{-\frac12+\varepsilon}}
    &\le C(t-t_1)^{-(1+\varepsilon-\eta)/3}
        \norm{\kappa(t_1)}_{\hr{\frac32+\eta}}\\
    &\quad+C\int_{t_1}^t
        (t-r)^{-1+(\rho-\varepsilon)/3}
        \norm{\p_sN(\kappa(r))}_{\hi{-\frac12+\rho}}\,dr.
\end{aligned}
\end{equation}
For the initial term, commute $\Lambda$ with the semigroup and use
$\norm{\Lambda\kappa(t_1)}_{\hi{-3/2+\eta}}
\le C\norm{\kappa(t_1)}_{\hr{3/2+\eta}}$.
Since $t-t_1\ge\tau/2$ and the kernel is integrable,
\eqref{eqn:Lam-12} is uniform on $[\tau,T]$.

The order-one term satisfies
\begin{equation}
    \norm{(A_3+4\Lambda)\kappa(t)}_{\hi{-\frac12+\varepsilon}}
    \le C\norm{\kappa(t)}_{\hr{\frac32+\eta}}.
\end{equation}
Together with \eqref{eqn:Lam-12}, this gives
$p_I A_3\kappa(t)\in H^{-1/2+\varepsilon}(I)$.
Lemma~\ref{lem:A-orderk-interval}(2) gives~\eqref{eq:A-orderk-interval-split},
whose endpoint power is $k/2=3/2$. This is the expansion
\eqref{eqn:main-endpoint-split}, since replacing $e^{-d_\pm}d_\pm^{3/2}$
by $d_\pm^{3/2}$ changes only the remainder:
$(e^{-d_\pm}-1)d_\pm^{3/2}=O(d_\pm^{5/2})$, and
\begin{equation*}
    \chi_\pm(e^{-d_\pm}-1)d_\pm^{3/2}
    \in H^{5/2+\varepsilon}(\R),\quad 0\le\varepsilon<\frac12.
\end{equation*}
The estimate~\eqref{eq:A-orderk-interval-split-est} gives
\begin{equation}\label{eqn:split-est}
    |c_-(t)|+|c_+(t)|+\norm{\kappa_r(t)}_{\hr{\frac52+\varepsilon}}
    \le C\paren{
        \norm{\Lambda\kappa(t)}_{\hi{-\frac12+\varepsilon}}
        +\norm{\kappa(t)}_{\hr{\frac32+\eta}}}.
\end{equation}
This bound is uniform on $[\tau,T]$. Since $\tau>0$ is arbitrary, the
expansion holds at every $t\in(0,T]$.
\end{proof}

Reconstruction of the mean orientation at $t=0$ relies on an additional small
Sobolev gain whose temporal singularity remains integrable near the initial time.
Define the exponents
\begin{equation*}
    a=\frac{2-\delta}{3},\quad
    \beta=\frac{\delta}{4},\quad
    \gamma=\frac{\delta}{2},\quad T_0=\min\{T,1\}.
\end{equation*}
If $M\ge\norm{\kappa}_{X_T}$ and
$0<m\le\inf_{0\le t\le T}\abs{\kappa(t)}_*$, then
\begin{equation}\label{eqn:reconstruction-small-time-smoothing}
    \norm{\kappa(t)}_{\hr{\frac32+\beta}}
    \le C_{M,m}t^{-a-\beta/3},
    \quad 0<t\le T_0.
\end{equation}
Repeat the argument for~\eqref{eqn:Lam-32} with initial time $t/2$.
The $X_T$ norm gives the bound $\norm{\kappa(r)}_{\hr{3/2}}\le Mr^{-a}$.
For forcing in $H^{-3/2+\gamma}(I)$ and target $H^{-3/2+\beta}(I)$,
\eqref{eqn:lambda-semigroup-est} gives the exponent
$-1+(\gamma-\beta)/3$; with~\eqref{eqn:N-pointwise-bound}, this yields
\begin{equation*}
\begin{aligned}
    \norm{\Lambda\kappa(t)}_{\hi{-\frac32+\beta}}
    &\le Ct^{-\beta/3}\norm{\kappa(t/2)}_{\hr{\frac32}}\\
    &\quad+C\int_{t/2}^t
        (t-r)^{-1+(\gamma-\beta)/3}
        \norm{\p_sN(\kappa(r))}_{\hi{-\frac32+\gamma}}\,dr\\
    &\le C_{M,m}\paren{
        t^{-a-\beta/3}+t^{-a+(\gamma-\beta)/3}}\\
    &\le C_{M,m}t^{-a-\beta/3}.
\end{aligned}
\end{equation*}
The kernel is integrable because $\gamma>\beta$, and $r^{-a}\le
2^at^{-a}$ on the integration interval. Applying
\eqref{eqn:k-est} with exponent $\beta$ gives
\eqref{eqn:reconstruction-small-time-smoothing}, with constants also
depending on the fixed $\delta$. Since $\beta=\delta/4$ gives
$a+\beta/3<1$, the bound is integrable at $0$.

\subsection{Reconstruction of the Stokes--filament solution}
\label{subsec:reconstruction-proof}

With the regularized curvature and tension fields established, we now invert the geometric and kinematic reductions to reconstruct the physical filament curve $\bm X$, the tension multiplier $\lambda$, and the ambient fluid fields $(\bm u, p)$.

\begin{proof}[Proof of Corollary~\ref{cor:main-reconstruction}]
Fix $M\ge\norm{\kappa}_{X_T}$ and
$0<m\le\inf_{0\le t\le T}\abs{\kappa(t)}_*$, and let
$\sigma=\sigma(\kappa)$ be given by
Proposition~\ref{prop:main-tension}.

\emph{Step 1. Reconstruct the angle, curve, and tension.}
For $t>0$, set
\begin{equation}\label{eqn:q-kappa-physical}
    \bm q_\kappa
    =\p_s\kappa\,\nk+\sigma\tk
    \in\thi{\frac12}.
\end{equation}
The derivative $\p_s\bm q_\kappa$ represents the Stokes single-layer force
density. By~\eqref{eqn:stokeslet-splitting}, its boundary velocity in
normalized coordinates, given by~\eqref{eqn:filament-boundary-evolution}
on $\xk$, is
\begin{equation}\label{eqn:reconstruction-boundary-integral}
    \bm v_\kappa(s,t)
    =-\frac14\mc H\bm q_\kappa(s,t)
    +\frac{1}{4\pi}\int_I
        \p_\eta r(s,\eta;\kappa(t))\bm q_\kappa(\eta,t)\,d\eta.
\end{equation}
For smooth supported densities, integration by parts in
$\int_I G(\Delta\xk)\p_\eta\bm q_\kappa\,d\eta$ gives this formula.
Equations~\eqref{eqn:hilbert-bound} and~\eqref{eqn:Sop-bound} extend~\eqref{eqn:reconstruction-boundary-integral} to
$\bm q_\kappa\in\thi{1/2}$. Define also
\begin{equation}
    h_\kappa=-\nk\cdot\p_s\bm v_\kappa.
\end{equation}
The angular velocity $h_\kappa$ determines the evolution of the mean orientation.
Put $a=(2-\delta)/3$, $\beta=\delta/4$, and $T_0=\min\{T,1\}$.
The $X_T$ bound supplies the basic $t^{-a}$ singularity; applying
\eqref{eqn:tension-bound}, \eqref{eqn:product-supported},
\eqref{eqn:product-restriction}, and~\eqref{eqn:Sop-bound} gives
\begin{equation}\label{eqn:reconstruction-time-bound}
    \norm{\bm q_\kappa(t)}_{\hr{\frac12}}
    +\norm{\bm v_\kappa(t)}_{\hi{\frac12}}
    +\norm{h_\kappa(t)}_{\hi{-\frac12}}
    \le C_{M,m}t^{-a},
    \quad 0<t\le T.
\end{equation}
The constants are uniform by~\eqref{eqn:inverse-uniform-bound} and~\eqref{eqn:tangent-Sobolev-bound}, which also applies to $\nk$.
For the mean angle, we use the additional gain in
\eqref{eqn:reconstruction-small-time-smoothing}.
Equation~\eqref{eqn:tension-higher-bound}, with exponent $\beta<\delta$,
gives $\norm{\sigma(t)}_{\hr{1/2+\beta}}\le C_{M,m}t^{-a}$.
Multiplication by the frame preserves $H^{1/2+\beta}$, and
the boundary integral in \eqref{eqn:reconstruction-boundary-integral}
is bounded at this index.
Consequently,
\begin{equation}\label{eqn:reconstruction-mean-bound}
\begin{aligned}
    &\norm{\bm q_\kappa(t)}_{\hr{\frac12+\beta}}
    +\norm{\bm v_\kappa(t)}_{\hi{\frac12+\beta}}
    +\norm{h_\kappa(t)}_{\hi{-\frac12+\beta}}\\
    &\hspace{3em}\le C_{M,m}t^{-a-\beta/3},
    \quad 0<t\le T_0.
\end{aligned}
\end{equation}
The zero extension of $1$ belongs to $\thi{1/2-\beta}$, so
$\dual{h_\kappa}{1}$ is a well-defined
$H^{-1/2+\beta}(I)$--$\thi{1/2-\beta}$ pairing.
Since $a+\beta/3<1$, its time integral converges at $0$
by~\eqref{eqn:reconstruction-mean-bound}.

Curvature determines the filament curve only up to rigid planar motions;
the mean-angle and centroid ODEs restore these degrees of freedom.
Let $\bar\theta_0$ be the initial mean angle in
\eqref{eqn:reconstruction-initial-means}. Following
\eqref{eqn:theta-bar} and \eqref{eqn:reconstruction-endpoint-angle}, set
\begin{equation}\label{eqn:reconstructed-mean-angle}
    \bar\theta(t)
    =\bar\theta_0+\frac12\int_0^t\dual{h_\kappa(r)}{1}\,dr,
\end{equation}
and
\begin{equation}\label{eqn:reconstructed-angle}
\begin{aligned}
    \theta(s,t)
    &=\bar\theta(t)+J\kappa(s,t)
        -\frac12\int_I J\kappa(\eta,t)\,d\eta,\\
    \theta_*(t)
    &=\bar\theta(t)-\frac12\int_I J\kappa(\eta,t)\,d\eta.
\end{aligned}
\end{equation}
These formulas agree with~\eqref{eqn:theta-bar} and~\eqref{eqn:reconstruction-endpoint-angle}, since
$\int_I J\kappa=\int_I(1-\eta)\kappa(\eta)\,d\eta$.
For the initial data, this identity is understood by density.
The mild solution is continuous in $\thi{-1/2+\delta}$, so
Lemma~\ref{lem:theta-kappa} gives
$\theta\in C([0,T];C(\overline I))$ and $\theta_*\in C([0,T])$.
Set
\begin{equation}
    \bm\tau=\bm e_r(\theta)=Q_{\theta_*}\tk,\quad
    \bm n=Q_{\theta_*}\nk,\quad
    \bm v=Q_{\theta_*}\bm v_\kappa.
\end{equation}
The centroid prescription \eqref{eqn:X-bar} and the curve formula
\eqref{eqn:curve-reconstruction} give
\begin{equation}\label{eqn:reconstructed-centered-curve}
\begin{aligned}
    \bm Y(s,t)
    &=-\frac12\int_I(1-\eta)\bm\tau(\eta,t)\,d\eta
        +\int_{-1}^s\bm\tau(\eta,t)\,d\eta,\\
    \overline{\bm X}(t)
    &=\overline{\bm X}_0
        +\frac12\int_0^t\int_I\bm v(s,r)\,ds\,dr,\\
    \bm X(s,t)&=\overline{\bm X}(t)+\bm Y(s,t).
\end{aligned}
\end{equation}
The centroid integral converges by \eqref{eqn:reconstruction-time-bound}.
For $t>0$, set
\begin{equation}\label{eqn:reconstructed-tension}
    \lambda=\sigma-\kappa^2.
\end{equation}

\emph{Step 2. Reconstruct the Stokes fields.}
Let $(\bm u_\kappa,p_\kappa)$ be the single-layer fields of Lemma~\ref{lem:stokes-layer-open-arc}, with velocity~\eqref{eqn:stokes-single-layer}, on
$\Gamma_\kappa=\xk(I)$ with density
$\bm f_\kappa=\p_s\bm q_\kappa$.
Write
\begin{equation}
    \overline{\bm X}_\kappa(t)=\frac12\int_I\xk(s,t)\,ds,
    \quad
    \bm Y=Q_{\theta_*}\paren{\xk-\overline{\bm X}_\kappa}.
\end{equation}
The physical curve is obtained by rotating the centered normalized
curve and adding the prescribed centroid. Applying the corresponding
coordinate transformation to the fluid fields yields
\begin{equation}\label{eqn:reconstructed-fluid}
\begin{aligned}
    \bm u(\bm x,t)
    &=Q_{\theta_*(t)}\bm u_\kappa\!\left(
        Q_{\theta_*(t)}^{-1}
        \paren{\bm x-\overline{\bm X}(t)}
        +\overline{\bm X}_\kappa(t),t\right),\\
    p(\bm x,t)
    &=p_\kappa\!\left(
        Q_{\theta_*(t)}^{-1}
        \paren{\bm x-\overline{\bm X}(t)}
        +\overline{\bm X}_\kappa(t),t\right).
\end{aligned}
\end{equation}
By translation invariance and rotational covariance of the Stokeslet, these
are the single-layer fields on $\Gamma(t)=\bm X(I,t)$ with density
$\bm f=\p_s\bm q$, where
\begin{equation}
    \bm q=\p_s\kappa\,\bm n+\sigma\bm\tau
        =Q_{\theta_*}\bm q_\kappa.
\end{equation}

\emph{Step 3. Establish positive-time regularity.}
Fix $0<\tau<T$ and $0<\alpha<\frac12$, and choose
$\alpha<\varepsilon<\varepsilon_1<\frac12$.
We first obtain regularity up to the endpoints for $\bm X$, $\lambda$, and the velocity
trace sufficient to interpret the endpoint and kinematic conditions classically.
Localization away from the endpoints will then provide additional interior regularity.
By~\eqref{eqn:main-positive-time-regularity}, $\kappa$ is continuous in
$\thi{3/2+\varepsilon_1}$ on $[\tau,T]$.
Equation~\eqref{eqn:tension-difference} gives continuity of $\sigma$ in $\thi{1/2}$,
since the normalized arcs at nearby times are close in $C^1$ and all their
convex combinations satisfy~\eqref{eqn:neighborhood-geometry}. Equation~\eqref{eqn:tension-higher-bound}, with its parameter
$\delta=\frac12$ and the positive-time $L^2$ bound for $\kappa$, gives a
uniform $\thi{1/2+\varepsilon_1}$ bound. Interpolation gives
\begin{equation}\label{eqn:reconstruction-positive-time-tension}
    \sigma\in C\paren{[\tau,T];\thi{\frac12+\varepsilon}},
    \quad
    \bm q\in C\paren{[\tau,T];\thi{\frac12+\varepsilon}}.
\end{equation}
The same continuity holds for $\bm q_\kappa$.
Lemma~\ref{lem:theta-kappa} and Sobolev composition give
\begin{equation*}
    \theta,\bm\tau,\bm n
    \in C\paren{[\tau,T];H^{5/2+\varepsilon_1}(I)}.
\end{equation*}
Equations~\eqref{eqn:reconstructed-centered-curve} and~\eqref{eqn:reconstructed-tension}, together with Sobolev embedding, give
\begin{equation}
    \bm X\in C\paren{[\tau,T];C^{3,\alpha}(\overline I)},
    \quad
    \lambda\in C\paren{[\tau,T];C^{0,\alpha}(\overline I)}.
\end{equation}

Globally, the supported Sobolev spaces retain the endpoint singularity
in Theorem~\ref{thm:main-regularity}. On compact subintervals of $I$,
the singular terms are smooth, so~\eqref{eqn:main-endpoint-split} gives
\begin{equation}
    \kappa(\cdot,t)\in H^{\frac52+\varepsilon}_{\mathrm{loc}}(I).
\end{equation}
To improve the tension regularity in the interior, we localize its
order-one equation and invert the whole-line principal part.
Lemma~\ref{lem:comm-psH} applies to the commutators in
\eqref{eqn:Rop-def}--\eqref{eqn:F-kappa}, since
$\bm\tau,\bm n\in H^{5/2+\varepsilon_1}(I)$ and
$\p_s\kappa\,\bm n,\sigma\bm\tau\in\thi{1/2+\varepsilon}$.
Choose $\gamma\in(\varepsilon,\varepsilon_1)$. The curve belongs to
$C^{3,\gamma}(\overline I)$, so Lemma~\ref{lem:K-higher-order} with $k=3$
maps the $L^2$ densities $\p_s\kappa\,\bm n$ and $\sigma\bm\tau$ into
$H^{2+\varepsilon}(I)$. Differentiating once and multiplying by the
tangent controls the Stokes-remainder terms, as in
\eqref{eqn:R-kernel-high}--\eqref{eqn:dR-kernel-high}.
Consequently $F(\kappa),R_\kappa\sigma\in H^{1/2+\varepsilon}(I)$.
Localizing~\eqref{eqn:ten-det} with $\chi\in C_c^\infty(I)$ removes
the endpoints and gives
\begin{equation}\label{eqn:localized-tension-equation}
    \p_s\mc H_{\R}(\chi\sigma)
    =4\chi\bigl(F(\kappa)-R_\kappa\sigma\bigr)
    +[\p_s\mc H_{\R},\chi]\sigma
    \in H^{1/2+\varepsilon}(\R).
\end{equation}
Here $\p_s\mc H_{\R}$ denotes the whole-line multiplier $|\xi|$, and
$\chi(F(\kappa)-R_\kappa\sigma)$ is extended by zero. The cutoff
commutator has order zero and preserves $H^{1/2+\varepsilon}(\R)$:
the Fourier kernel
$(|\xi|-|\eta|)\widehat\chi(\xi-\eta)$ is bounded using
$\bigl||\xi|-|\eta|\bigr|\le|\xi-\eta|$ and weighted Young's inequality.
Let $A_1$ have symbol $\langle\xi\rangle$.
Since $A_1-\p_s\mc H_{\R}$ has order $-1$ and
$\chi\sigma\in H^{1/2+\varepsilon}(\R)$,
\eqref{eqn:localized-tension-equation} gives
$A_1(\chi\sigma)\in H^{1/2+\varepsilon}(\R)$ and, by inversion,
$\chi\sigma\in H^{3/2+\varepsilon}(\R)$, proving
\begin{equation}\label{eqn:interior-tension-regularity}
    \sigma(\cdot,t)\in H^{\frac32+\varepsilon}_{\mathrm{loc}}(I).
\end{equation}
Consequently,
\begin{equation}
    \bm X(\cdot,t)\in H^{\frac92+\varepsilon}_{\mathrm{loc}}(I),
    \quad
    \lambda(\cdot,t)\in H^{\frac32+\varepsilon}_{\mathrm{loc}}(I),
\end{equation}
and Sobolev embedding gives
\begin{equation}\label{eqn:main-interior-classical-regularity}
    \bm X(\cdot,t)\in C_{\mathrm{loc}}^{4,\alpha}(I),
    \quad \lambda(\cdot,t)\in C_{\mathrm{loc}}^{1,\alpha}(I).
\end{equation}
With a fixed interior cutoff, these Sobolev bounds are
uniform on $[\tau,T]$, by~\eqref{eqn:split-est} and the argument leading to~\eqref{eqn:interior-tension-regularity}. Interpolating with the positive-time continuity proved above
gives time continuity in the local H\"older spaces in
\eqref{eqn:main-interior-classical-regularity}.
Globally, $\bm q\in\thi{1/2+\varepsilon}$ and
$\bm f=\p_s \bm q\in\thi{-1/2+\varepsilon}$, while locally
$\bm q\in H^{3/2+\varepsilon}$ and $\bm f\in H^{1/2+\varepsilon}$.
Thus $\bm q\in C^{1,\alpha}_{\mathrm{loc}}(I)$ and
$\bm f\in C^{0,\alpha}_{\mathrm{loc}}(I)$.

For the fluid, Lemma~\ref{lem:stokes-layer-open-arc} supplies the
finite-energy solution and its common velocity trace.
Together with~\eqref{eqn:reconstructed-fluid}, it gives
\begin{equation}\label{eqn:reconstruction-fluid-energy}
    \bm u(\cdot,t)\in H^1_{\mathrm{loc}}(\R^2),
    \quad \grad\bm u(\cdot,t),\ p(\cdot,t)\in L^2(\R^2),
    \quad t>0.
\end{equation}
The single-layer representation in Lemma~\ref{lem:stokes-layer-open-arc} shows that $\bm u,p$ are smooth off $\overline{\Gamma(t)}$.

For the classical one-sided regularity, we use local layer-potential theory.
Along a $C^{2,\alpha}$ curve, a $C^{0,\alpha}$ single-layer density
gives velocity and pressure of class $C^{1,\alpha}$ and $C^{0,\alpha}$,
respectively, up to the curve from each side; see
\cite[Section~2.3]{Hsiao2021} and
\cite[Theorem~2.1 and Appendix~A]{KuoLaiMoriRodenberg2023Tension}.
For a compact interior subarc, cut off $\bm f$ on a slightly larger
subarc away from the free endpoints and extend the latter to a closed
$C^{2,\alpha}$ curve. The cutoff density extends by zero in $C^{0,\alpha}$,
so the cited regularity applies; the remaining density is supported
away from the smaller subarc and produces a smooth field there.
This gives one-sided $C^{1,\alpha}$ regularity for $\bm u$ and
$C^{0,\alpha}$ regularity for $p$ on every compact interior subarc.

By Lemma~\ref{lem:stokes-layer-open-arc}, the common trace of
$\bm u_\kappa$ is \eqref{eqn:reconstruction-boundary-integral}.
The whole-line Hilbert transform is bounded on $H^{1/2+\varepsilon}$.
Lemmas~\ref{lem:Sop-est} and~\ref{lem:Sop-Lip}, with input index
$1/2+\varepsilon<1$ and gain zero, give boundedness and continuity
of the remainder at the same index. The density and rotation are
continuous, so
\begin{equation}\label{eqn:reconstruction-positive-time-trace}
    \bm v=(\bm u|_\Gamma)\circ\bm X
    \in C\paren{[\tau,T];H^{\frac12+\varepsilon}(I)}
    \hookrightarrow C\paren{[\tau,T];C^{0,\alpha}(\overline I)}.
\end{equation}

In particular, $h_\kappa\in C([\tau,T];H^{-1/2+\varepsilon}(I))$.
Equations~\eqref{eqn:reconstructed-mean-angle} and~\eqref{eqn:reconstructed-centered-curve} give
$\bar\theta\in C^1([\tau,T])$ and
$\overline{\bm X}\in C^1([\tau,T])$.

To differentiate~\eqref{eqn:reconstructed-centered-curve} in time, we
write both the angle and the centered curve using the mean-zero
primitive $J_0$. It gains one spatial derivative in the negative
Sobolev scales needed here. Subtracting the spatial mean from $J$ gives,
for smooth $f$,
\begin{equation}\label{eqn:centered-primitive}
\begin{aligned}
    (J_0f)(s)
    &=(Jf)(s)-\frac12\int_I(Jf)(r)\,dr\\
    &=\int_{-1}^s f(\eta)\,d\eta
        -\frac12\int_I(1-\eta)f(\eta)\,d\eta.
\end{aligned}
\end{equation}
It extends to bounded maps
\begin{equation}\label{eqn:centered-primitive-bounds}
\begin{aligned}
    J_0:H^{-3/2+\varepsilon}(I)&\longrightarrow
        H^{-1/2+\varepsilon}(I),\\
    J_0:H^{-1/2+\varepsilon}(I)&\longrightarrow
        H^{1/2+\varepsilon}(I).
\end{aligned}
\end{equation}
To prove the first bound, consider the adjoint
\begin{equation*}
    (J_0^*g)(\eta)
    =\int_\eta^1 g(s)\,ds
        -\frac{1-\eta}{2}\int_I g(s)\,ds.
\end{equation*}
It vanishes at both endpoints and satisfies
$(J_0^*g)'=-g+\frac12\int_I g$. Thus
\begin{equation*}
    \norm{J_0^*g}_{\hr{3/2-\varepsilon}}
    \le C_\varepsilon\norm{g}_{\hi{1/2-\varepsilon}}.
\end{equation*}
The norm on the left is taken after zero extension. Since
$1<3/2-\varepsilon<3/2$, the zero endpoint values give
$J_0^*g\in\thi{3/2-\varepsilon}$.
Duality against $g\in\thi{1/2-\varepsilon}$ proves the first bound in
\eqref{eqn:centered-primitive-bounds}.
For the second bound, zero extension is bounded at the index
$-1/2+\varepsilon\in(-1/2,0)$; apply Lemma~\ref{lem:theta-kappa} and
subtract the mean. Both extensions satisfy
$\p_sJ_0f=f$ and $\dual{J_0f}{1}=0$.

Using~\eqref{eqn:centered-primitive}, we can write~\eqref{eqn:reconstructed-angle} and~\eqref{eqn:reconstructed-centered-curve} as
$\theta=\bar\theta+J_0\kappa$ and
$\bm X=\overline{\bm X}+J_0\bm\tau$.
Theorem~\ref{thm:main-regularity} and
\eqref{eqn:centered-primitive-bounds} give
\begin{equation}\label{eqn:reconstruction-angle-time}
    \p_t\theta=\p_t\bar\theta+J_0\p_t\kappa
    \in C\paren{[\tau,T];H^{-1/2+\varepsilon}(I)}.
\end{equation}
The clockwise normal convention gives $\p_\theta\bm e_r=-\bm n$.
Since $\p_t\theta$ lies only in a negative Sobolev space, a pointwise
chain rule is not available. We therefore mollify $\theta$ in time on
compact subintervals of $(0,T)$. The mollified angles converge in
$C_tH^{5/2+\varepsilon_1}$, and their derivatives converge in
$C_tH^{-1/2+\varepsilon}$. By Sobolev composition and~\eqref{eqn:product-restriction}, the products $-\bm n\,\p_t\theta$ converge in
$C_tH^{-1/2+\varepsilon}$.
Passing to the limit in the smooth chain rule yields
\begin{equation}\label{eqn:tangent-time-chain-rule}
    \p_t\bm\tau=-\bm n\,\p_t\theta.
\end{equation}
Differentiating~\eqref{eqn:reconstructed-centered-curve} gives
\begin{equation}\label{eqn:reconstruction-curve-time}
    \p_t\bm X
    =\p_t\overline{\bm X}-J_0(\bm n\,\p_t\theta)
    \in C\paren{[\tau,T];H^{1/2+\varepsilon}(I)}.
\end{equation}
The derivative is continuous on $[\tau,T]$, with endpoint values given
by the one-sided time derivatives of $\bm X$. Integrating the
distributional identity recovers $\bm X$.
Sobolev embedding therefore gives
\begin{equation}\label{eqn:main-physical-regularity}
\begin{aligned}
    \bm X&\in C\paren{[\tau,T];C^{3,\alpha}(\overline I)}
        \cap C^1\paren{[\tau,T];C^{0,\alpha}(\overline I)},\\
    \lambda&\in C\paren{[\tau,T];C^{0,\alpha}(\overline I)}.
\end{aligned}
\end{equation}

\emph{Step 4. Verify the original Stokes--filament equations.}
The identities~\eqref{eqn:reduced-velocity-tan}--\eqref{eqn:reduced-velocity-normal}
connect the reduced equation to the kinematic condition: the tangential
velocity derivative vanishes, while the normal component determines
the angle evolution. We justify them by approximation before recovering
the curve evolution.

The reconstruction gives $\p_s\theta=\kappa$ and
$\p_s\bm X=\bm\tau$, so $\abs{\p_s\bm X}=1$.
The Frenet identities give
\begin{equation}\label{eqn:reconstructed-force}
    \lambda\p_s\bm X-\p_s^3\bm X
    =\p_s\kappa\,\bm n+\sigma\bm\tau
    =\bm q.
\end{equation}

For the boundary integral $\bm v_\kappa$, the splitting~\eqref{eqn:stokeslet-splitting} gives, in
$H^{-1/2}(I)$ and $H^{-3/2}(I)$, respectively,
\begin{subequations}\label{eqn:reduced-velocity}
\begin{align}
    \tk\cdot\p_s\bm v_\kappa&=0,
        \label{eqn:reduced-velocity-tan}\\
    -\p_s\paren{\nk\cdot\p_s\bm v_\kappa}
        &=\Lambda\kappa+\p_sN(\kappa).
        \label{eqn:reduced-velocity-normal}
\end{align}
\end{subequations}
To justify these identities, first take smooth supported $\kappa$ and
$\sigma$, without imposing the tension equation. In the tangential
component of $\p_s\bm v_\kappa$, the field generated by
$\p_s(\p_s\kappa\,\nk)$ contributes $F(\kappa)$, while
\eqref{eqn:tension-tangential-id} gives the contribution
$-T_\kappa\sigma$ from $\p_s(\sigma\tk)$. For the normal component, use
$\nk\cdot\bm q_\kappa=\p_s\kappa$ in
\eqref{eqn:reconstruction-boundary-integral}. Thus the unreduced
identities are
\begin{equation}\label{eqn:unreduced-velocity-identities}
\begin{aligned}
    \tk\cdot\p_s\bm v_\kappa
    &=F(\kappa)-T_\kappa\sigma,\\
    -\p_s\paren{\nk\cdot\p_s\bm v_\kappa}
    &=\frac14\p_s^2\mc H\p_s\kappa
        +\p_s\wt N(\kappa,\sigma),
\end{aligned}
\end{equation}
where $\wt N$ is defined in \eqref{eqn:N-def}.

For $\kappa\in\thi{3/2}$ and $\sigma\in\thi{1/2}$, choose smooth
supported functions $\kappa_j\to\kappa$ and $\sigma_j\to\sigma$ in
these two spaces. For all sufficiently large $j$, the normalized curves
have a uniform arc-chord lower bound. Equations~\eqref{eqn:hilbert-bound}, \eqref{eqn:product-supported}, \eqref{eqn:Sop-bound}, and~\eqref{eqn:Sop-difference} give convergence of
$\bm q_j=(\p_s\kappa_j)\bm n_{\kappa_j}+\sigma_j\bm\tau_{\kappa_j}$
to $\bm q_\kappa$ in $\thi{1/2}$ and of the corresponding boundary
integrals in $H^{1/2}(I)$. After differentiation, the multiplication
bound~\eqref{eqn:product-restriction} controls the tangential and normal
projections in $H^{-1/2}(I)$. The operator estimates
\eqref{eqn:commutator-low-bound}, \eqref{eqn:Sop-bound}, and
\eqref{eqn:Sop-difference} give convergence of the commutator and
remainder terms there, and hence of their derivatives in $H^{-3/2}(I)$.
Thus both identities in~\eqref{eqn:unreduced-velocity-identities}
pass to the limit in the stated spaces. Imposing~\eqref{eqn:ten-det} and setting $N(\kappa)=\wt N(\kappa,\sigma(\kappa))$ in~\eqref{eqn:N-def} give
\eqref{eqn:reduced-velocity-tan}--\eqref{eqn:reduced-velocity-normal}.
Equation~\eqref{eqn:kappa-time-regularity} therefore gives, for every $t>0$,
\begin{equation}\label{eqn:kappa-velocity-form}
    \p_t\kappa
    =-\p_s\paren{\nk\cdot\p_s\bm v_\kappa}
    =\p_sh_\kappa
    \quad\text{in }H^{-3/2}(I).
\end{equation}

Since $\p_s\theta=\kappa$, \eqref{eqn:kappa-velocity-form} implies
$\p_s(\p_t\theta-h_\kappa)=0$. On $[\tau,T]$, both terms belong to
$H^{-1/2+\varepsilon}(I)$, so their means are well defined.
Equation~\eqref{eqn:reconstructed-mean-angle} gives
\begin{equation*}
    \dual{\p_t\theta-h_\kappa}{1}
    =2\p_t\bar\theta-\dual{h_\kappa}{1}=0.
\end{equation*}
The mean-angle equation therefore removes the constant of integration.
By rotational covariance and~\eqref{eqn:reduced-velocity-tan},
\begin{equation*}
    \p_t\theta=h_\kappa=-\bm n\cdot\p_s\bm v,
    \qquad \bm\tau\cdot\p_s\bm v=0.
\end{equation*}
The chain rule~\eqref{eqn:tangent-time-chain-rule} now gives
\begin{equation*}
    \p_t\bm\tau
    =-\bm n\,\p_t\theta
    =(\bm n\cdot\p_s\bm v)\bm n
    =\p_s\bm v.
\end{equation*}
Since $\p_s\bm X=\bm\tau$, this implies
$\p_s(\p_t\bm X-\bm v)=0$. The curve $\bm Y$ has zero mean,
and \eqref{eqn:reconstructed-centered-curve} prescribes
$\p_t\overline{\bm X}=\frac12\int_I\bm v$.
The centroid equation fixes precisely the remaining spatially constant
velocity: $\p_t\bm X-\bm v$ has zero mean and therefore vanishes. By
\eqref{eqn:reconstruction-positive-time-trace} and
\eqref{eqn:main-physical-regularity},
\begin{equation}
    \p_t\bm X(s,t)=\bm v(s,t)=\bm u(\bm X(s,t),t)
\end{equation}
holds pointwise for every $s\in\overline I$ and $t>0$, using the
continuous velocity trace at the endpoints.

For fixed $t>0$, Lemma~\ref{lem:stokes-layer-open-arc} gives
\eqref{eqn:intro-stokes-bulk} off $\overline{\Gamma(t)}$, pointwise
because the single-layer fields are smooth there.
The classical jump relations for the localized layer potentials, with
the convention in
\cite[Theorem~2.1 and Appendix~A]{KuoLaiMoriRodenberg2023Tension}, give
\begin{equation}
    \jump{\bm u}=0,\quad
    \jump{\Sigma\bm n}=\p_s\bm q
        =\p_s\paren{\lambda\p_s\bm X-\p_s^3\bm X}
\end{equation}
pointwise at every interior point of the filament.

\emph{Step 5. Verify endpoint conditions, initial data, far-field behavior,
and uniqueness.}
Since $\kappa\in\thi{3/2+\varepsilon}$, its distributional derivative
belongs to $\thi{1/2+\varepsilon}$. For $\varepsilon>0$, Sobolev
embedding makes the zero extensions of both $\kappa$ and $\p_s\kappa$
continuous; their support in $\overline I$ forces them to vanish at
both endpoints. Likewise, $\sigma\in\thi{1/2+\varepsilon}$ has a
continuous zero extension and zero endpoint traces. Hence
\begin{equation}
    \kappa=\p_s\kappa=\sigma=0\quad\text{on }\p I.
\end{equation}
The identities $\p_s^2\bm X=-\kappa\bm n$ and
$\lambda\p_s\bm X-\p_s^3\bm X=\p_s\kappa\,\bm n+\sigma\bm\tau$
from~\eqref{eqn:reconstructed-force} give the pointwise free-end conditions
\begin{equation}
    \p_s^2\bm X=0,\quad
    \lambda\p_s\bm X-\p_s^3\bm X=0
    \quad\text{at }s=\pm1.
\end{equation}
At the initial time, continuity of $\theta$ and the centroid, together
with \eqref{eqn:reconstruction-initial-means}, gives
\begin{equation}\label{eqn:main-X-regularity}
    \bm X\in C\paren{[0,T];C^1(\overline I)},\quad
    \bm X(\cdot,0)=\bm X_0,\quad
    \abs{\p_s\bm X}=1.
\end{equation}

Since $\p_s\bm q$ is the derivative of a supported distribution, it has
zero total force. The logarithmic far field cancels, and
Lemma~\ref{lem:stokes-layer-open-arc} gives
$\bm u(\bm x,t),p(\bm x,t)\to0$ as $\abs{\bm x}\to\infty$.

For a fixed curvature, Proposition~\ref{prop:main-tension} determines
$\sigma$. Equation~\eqref{eqn:reconstructed-mean-angle} fixes $\bar\theta$ from its prescribed
initial value, hence also $\theta$ and $\theta_*$.
The centroid equation~\eqref{eqn:reconstructed-centered-curve} fixes the
translation. If the initial angle and centroid are left unspecified,
any two reconstructions differ by a fixed rotation and translation.

Finally, Lemma~\ref{lem:stokes-layer-open-arc} and
\eqref{eqn:reconstructed-fluid} uniquely fix the fluid fields with the
prescribed far-field normalization. All components of
\eqref{eqn:intro-stokes-bulk}--\eqref{eqn:intro-kinematic} and
\eqref{eqn:intro-free-end}, together with the initial data, have now been
verified in the sense stated in Corollary~\ref{cor:main-reconstruction}.
\end{proof}

\section{Global solutions and long-time behavior}
\label{sec:global-theory}

This section proves Theorem~\ref{thm:main-global} and Corollary~\ref{cor:main-straightening} by combining energy dissipation, small-data contraction in exponentially weighted spaces, and asymptotic rigidity of zero-dissipation states.

\subsection{Energy dissipation and geometric control}
\label{subsec:energy-geometry}

\begin{proposition}[Energy dissipation and behavior at the initial time]
\label{prop:main-energy-identity}
Fix $0<\delta\le\frac12$ and $\kappa_0\in\mc U_\delta$, and let $\kappa$ be its
maximal mild solution. For $0<a<b<\tau_{\mathrm{max}}(\kappa_0)$,
the constructed solution to the Stokes--filament system satisfies
\begin{equation}\label{eqn:main-energy-positive}
    \mc E(b)
    +2\int_a^b\int_{\R^2}\abs{\grad_S\bm u(\bm x,t)}^2\,d\bm x\,dt
    =\mc E(a).
\end{equation}
If $\kappa_0\in\wt L^2(I)$, then~\eqref{eqn:main-energy-positive} extends to the initial time:
\begin{equation}\label{eqn:main-energy-initial}
    \mc E(t)
    +2\int_0^t\int_{\R^2}
        \abs{\grad_S\bm u(\bm x,r)}^2\,d\bm x\,dr
    =\mc E(0),
    \quad 0\le t<\tau_{\mathrm{max}}(\kappa_0).
\end{equation}
If $\kappa_0\notin\wt L^2(I)$, then
\begin{equation}\label{eqn:main-infinite-initial-energy}
    \lim_{t\downarrow0}\mc E(t)=\infty.
\end{equation}
\end{proposition}
\begin{proof}
\emph{Positive-time energy identity.}
Fix $0<a<b<\tau_{\mathrm{max}}(\kappa_0)$ and define the Gelfand triple
\begin{equation}
    V=\thi{\frac32},\quad H=L^2(I),\quad
    V^*=H^{-\frac32}(I).
\end{equation}
The continuity property in~\eqref{eqn:XT-intro} guarantees $\kappa\in C([a,b];V)$, while the parabolic smoothing~\eqref{eqn:main-positive-time-regularity} yields $\p_t\kappa\in L^2(a,b;V^*)$. By standard chain rule calculus for abstract evolution equations in Gelfand triples, the energy mapping $t\mapsto\mc E(t)$ is absolutely continuous on $[a,b]$, with
\begin{equation}
    \frac{d}{dt}\mc E(t)
    =
    \dual{\p_t\kappa}{\kappa}_{V^*,V}
\end{equation}
for almost every $t\in(a,b)$. To identify this time derivative with the Stokes energy dissipation, we express the curvature equation via the force primitive $\bm q_\kappa$. Using~\eqref{eqn:kappa-velocity-form}, distributional integration by parts, and the tangential identity~\eqref{eqn:reduced-velocity-tan}, we calculate
\begin{equation}\label{eqn:weak-curvature-energy}
    \dual{\p_t\kappa}{\kappa}_{V^*,V} = \dual{\nk\cdot\p_s\bm v_\kappa}{\p_s\kappa} =\dual{\p_s\bm
    v_\kappa}{\p_s\kappa\,\nk} = \dual{\p_s\bm v_\kappa} {\p_s\kappa\,\nk+\sigma\tk} =-\dual{\p_s \bm q_\kappa}{\bm
    v_\kappa}.
\end{equation}
The first three duality pairings in~\eqref{eqn:weak-curvature-energy} are evaluated in $H^{-\frac12}(I)\times\wt H^{\frac12}(I)$. Because $\p_s \bm q_\kappa\in\wt H^{-\frac12}(I)$, the final pairing corresponds to the $\wt H^{-\frac12}(I)\times H^{\frac12}(I)$ duality established in Lemma~\ref{lem:stokes-layer-open-arc}. Distributional integration by parts with these supported densities produces no boundary terms. Applying the weak energy identity~\eqref{eqn:weak-stokes-energy} with $\bm g=\bm q_\kappa$, and noting that the Dirichlet dissipation integral is invariant under the rigid motions in Corollary~\ref{cor:main-reconstruction}, we obtain
\begin{equation}
    \frac{d}{dt}\mc E(t)
    =
    -2\int_{\R^2}\abs{\grad_S\bm u(\bm x,t)}^2\,d\bm x.
\end{equation}
Integrating this identity over $(a,b)$ yields~\eqref{eqn:main-energy-positive}.

\emph{Behavior at $t=0$ for $L^2$ data.}
When $\kappa_0\in\wt L^2(I)$, the mild solution constructed at exponent $\delta=\frac12$ coincides with the maximal solution on its lifespan by Proposition~\ref{lem:restart-uniqueness}. Consequently, $\kappa\in C([0,b];\wt L^2(I))$. Taking the limit $a\downarrow0$ in~\eqref{eqn:main-energy-positive}, strong continuity of the energy and the monotone convergence theorem for the non-negative dissipation integrand establish~\eqref{eqn:main-energy-initial}.

\emph{Non-$L^2$ initial data.}
Suppose conversely that $\kappa_0\notin\wt L^2(I)$. By monotonicity in~\eqref{eqn:main-energy-positive}, $\lim_{t\downarrow0}\mc E(t)$ exists in $[0,\infty]$. Were this limit finite, the family $\{\kappa(t_j)\}$ would be uniformly bounded in $\wt L^2(I)$ along any sequence $t_j\downarrow0$, extracting a weakly convergent subsequence in $L^2(\R)$. However, continuity at the initial time ensures that
\begin{equation}
    \kappa(t_j)\longrightarrow\kappa_0
    \quad\text{in }\thi{-\frac12+\delta}.
\end{equation}
By uniqueness of distributional limits, $\kappa_0$ must belong to $\wt L^2(I)$, which contradicts the hypothesis and proves~\eqref{eqn:main-infinite-initial-energy}.
\end{proof}

The next estimate controls chord lengths in terms of total turning, bounding the arc-chord constant away from zero whenever the bending energy remains below a universal threshold.

\begin{lemma}
\label{lem:energy-arc-chord}
Let $\kappa\in\wt L^2(I)$. For $-1\le s<\eta\le1$, set
\begin{equation}
    h=\eta-s,
    \quad
    V_{s,\eta}=\int_s^\eta\abs{\kappa(r)}\,dr.
\end{equation}
If $V_{s,\eta}<\pi$, then
\begin{equation}\label{eqn:turning-chord-bound}
    \frac{\abs{\xk(\eta)-\xk(s)}}{h}
    \ge\cos\paren{\frac{V_{s,\eta}}2}.
\end{equation}
Consequently, if $\sqrt h\norm{\kappa}_{L^2(s,\eta)}<\pi$, then
\begin{equation}\label{eqn:l2-chord-bound}
    \frac{\abs{\xk(\eta)-\xk(s)}}{h}
    \ge
    \cos\paren{
        \frac{\sqrt h}{2}\norm{\kappa}_{L^2(s,\eta)}
    }.
\end{equation}
In particular,
\begin{equation}\label{eqn:energy-arc-chord-bound}
    \frac12\norm{\kappa}_{L^2(\R)}^2<\mc E_{\mathrm{ac}}
    \quad\Longrightarrow\quad
    \abs{\kappa}_*
    \ge
    \cos\paren{\frac{\norm{\kappa}_{L^2(\R)}}{\sqrt2}}
    >0.
\end{equation}
Moreover, if a sequence $(\kappa_j)$ is bounded in $\wt L^2(I)$ and
$\abs{\kappa_j}_*\to0$, then
\begin{equation}\label{eqn:sequential-energy-barrier}
    \liminf_{j\to\infty}
    \frac12\norm{\kappa_j}_{L^2(\R)}^2
    \ge\mc E_{\mathrm{ac}}.
\end{equation}
\end{lemma}
\begin{proof}
When $V_{s,\eta}<\pi$, the tangent field $\tk$ is confined to an angular sector of opening at most $V_{s,\eta}$, so projecting the chord onto the central bisector provides a uniform geometric lower bound.
Because the tangent angle $\theta_\kappa$ is absolutely continuous with $\p_s\theta_\kappa=\kappa$, its oscillation on $[s,\eta]$ satisfies
\begin{equation}
    \theta_+-\theta_-\le V_{s,\eta},
    \quad
    \theta_-=\min_{[s,\eta]}\theta_\kappa,
    \quad
    \theta_+=\max_{[s,\eta]}\theta_\kappa.
\end{equation}
Whenever $V_{s,\eta}<\pi$, the unit vector $\bm e$ with orientation $(\theta_-+\theta_+)/2$ satisfies $\tk(r)\cdot\bm e\ge\cos(V_{s,\eta}/2)>0$ for all $r\in[s,\eta]$. Therefore,
\begin{equation}
    \abs{\xk(\eta)-\xk(s)}
    \ge\paren{\xk(\eta)-\xk(s)}\cdot\bm e
    =\int_s^\eta\tk(r)\cdot\bm e\,dr
    \ge h\cos\paren{\frac{V_{s,\eta}}2}.
\end{equation}
which proves~\eqref{eqn:turning-chord-bound}. The Cauchy--Schwarz inequality $V_{s,\eta}\le\sqrt h\norm{\kappa}_{L^2(s,\eta)}$ combined with the monotonicity of the cosine yields~\eqref{eqn:l2-chord-bound}.

Next, suppose that $\frac12\norm{\kappa}_{L^2(\R)}^2<\mc E_{\mathrm{ac}}=\pi^2/4$. Since $h=\eta-s\le2$, every subinterval satisfies
\begin{equation}
    \frac{\sqrt h}{2}\norm{\kappa}_{L^2(s,\eta)}
    \le\frac{\norm{\kappa}_{L^2(\R)}}{\sqrt2}<\frac\pi2.
\end{equation}
Applying~\eqref{eqn:l2-chord-bound} and taking the infimum over all $-1\le s<\eta\le 1$ establishes~\eqref{eqn:energy-arc-chord-bound}.

Finally, if the sequential barrier~\eqref{eqn:sequential-energy-barrier} were violated, some subsequence would satisfy $\frac12\norm{\kappa_j}_{L^2(\R)}^2\le E_0<\mc E_{\mathrm{ac}}$. The lower bound~\eqref{eqn:energy-arc-chord-bound} would then force $\abs{\kappa_j}_*\ge\cos\sqrt{E_0}>0$, contradicting the assumption $\abs{\kappa_j}_*\to0$.
\end{proof}

In particular, bending energy below $\mc E_{\mathrm{ac}}=\pi^2/4$ prevents arc-chord degeneration.

\begin{proposition}\label{prop:std-stt}
The only stationary mild solution of~\eqref{eqn:kap_evo} is
$\kappa_\infty=0$, which represents a straight filament.
\end{proposition}
\begin{proof}
For any stationary mild solution $\kappa$, rigidity follows from the vanishing of the dissipation rate and the traction-free boundary conditions.
Let $(\bm X,\bm u,p,\lambda)$ be the hydrodynamic reconstruction associated with $\kappa$ by Corollary~\ref{cor:main-reconstruction}. Applying the dissipation identity~\eqref{eqn:main-energy-positive} on any positive-time interval yields $\grad_S\bm u=0$ almost everywhere in $\R^2$. By the far-field condition $\bm u(\bm x)\to0$ as $|\bm x|\to\infty$, we conclude that $\bm u\equiv0$.
Because the complement $\R^2\setminus\overline{\Gamma}$ of the open arc is connected, the homogeneous Stokes system~\eqref{eqn:intro-stokes-bulk} implies that the pressure $p$ is constant on $\R^2\setminus\overline{\Gamma}$, and the decay at infinity forces $p\equiv0$.
Consequently, the hydrodynamic stress jump across the filament vanishes identically. The interface condition~\eqref{eqn:intro-stokes-interface} and the bending force formula~\eqref{eqn:bending-force} then reduce to
\begin{equation}
    \p_s\paren{\lambda\p_s\bm X-\p_s^3\bm X}=0.
\end{equation}
Integrating and applying the free-end boundary condition~\eqref{eqn:intro-free-end} at $\p I$, we obtain $\lambda\p_s\bm X-\p_s^3\bm X=0$ on all of $I$. Expanding this combination in the Frenet frame via $\lambda\p_s\bm X-\p_s^3\bm X=(\lambda+\kappa^2)\bm\tau+\p_s\kappa\,\bm n$ and projecting onto the normal vector $\bm n$ yields $\p_s\kappa=0$ on $I$. Thus $\kappa$ is constant on $I$. Finally, the moment-free boundary condition in~\eqref{eqn:intro-free-end}, namely $\p_s^2\bm X=-\kappa\bm n=0$ at $\p I$, enforces $\kappa\equiv0$ on $I$.
\end{proof}

\subsection{Small-data global existence and exponential decay}
\label{subsec:small-data-proof}

To capture both the short-time singularity $t^{-(2-\delta)/3}$ and large-time exponential relaxation, we adapt the contraction mapping argument of Section~\ref{sec:local-theory} to an exponentially weighted trajectory space.
Fix $0<\delta\le\frac12$ and $\omega\in(\omega_\Lambda,0)$, where $\omega_\Lambda<0$ is the spectral bound from Proposition~\ref{prop:lambda-semigroup}. We define the solution space
\begin{equation}\label{eqn:Y-intro}
    Y
    \coloneqq
    \left\{
    \kappa\in C\paren{[0,\infty);\thi{-\frac12+\delta}}
    \cap C\paren{(0,\infty);\thi{\frac32}}:
    \norm{\kappa}_Y<\infty
    \right\},
\end{equation}
with norm
\begin{equation}\label{eqn:Y-norm-intro}
\begin{aligned}
    \norm{\kappa}_Y
    &=\norm{\kappa}_{Y_1}+\norm{\kappa}_{Y_2},\quad
    \norm{\kappa}_{Y_1}
    =\sup_{t\ge0}e^{-t\omega}
    \norm{\kappa(t)}_{\hr{-\frac12+\delta}},\\
    \norm{\kappa}_{Y_2}
    &=\sup_{t>0}e^{-t\omega}
    \min\left\{t^{\frac{2-\delta}{3}},1\right\}
    \norm{\kappa(t)}_{\hr{\frac32}}.
\end{aligned}
\end{equation}

We write $Y^{(\delta,\omega)}$ when the parameters need to be specified.

\begin{proof}[Proof of Theorem~\ref{thm:main-global}(1)]
Fix $0<\delta\le\frac12$ and $\omega\in(\omega_\Lambda,0)$. In the exponentially weighted space $Y$, consider the closed ball of radius $M>0$:
\begin{equation}\label{eqn:BM-intro}
    \mc B_M=\left\{\kappa\in Y:\norm{\kappa}_Y\le M\right\}.
\end{equation}
Set $\varepsilon=\delta/2$. Throughout this proof, constants $C$ may depend on $\delta$ and $\omega$, but remain independent of time and of the radius $M$. We write
\begin{equation}
    w(t)=\min\{t^\frac{2-\delta}{3},1\}.
\end{equation}

\emph{Time-convolution bounds.}
We establish the exponentially weighted analogue~\eqref{eqn:global-duhamel-est} of the local Duhamel estimate~\eqref{eqn:local-duhamel-est}. Define the exponents
\begin{equation}
    a=\frac{2-\delta}{3},\quad
    \beta=\frac{1+\delta-\varepsilon}{3},\quad
    \gamma=1-\frac{\varepsilon}{3}.
\end{equation}
Then $0<a,\beta,\gamma<1$, and we claim that
\begin{equation}\label{eqn:glb-conv-low}
    \sup_{t>0}\int_0^t e^{\omega\tau}
    \max\{(t-\tau)^{-\beta},1\}w(\tau)^{-1}\,d\tau <\infty,
\end{equation}
\begin{equation}\label{eqn:glb-conv-high}
    \sup_{t>0}w(t)\int_0^t e^{\omega\tau}
    \max\{(t-\tau)^{-\gamma},1\}w(\tau)^{-1}\,d\tau <\infty.
\end{equation}
For $0<t\le1$, we have $w(\tau)=\tau^a$ on $(0,t)$, so the short-time convolution satisfies
\begin{equation}
    \int_0^t (t-\tau)^{-\beta}\tau^{-a}\,d\tau
    \le C t^{1-a-\beta}
    = Ct^{\varepsilon/3}\le C,
\end{equation}
while
\begin{equation}
    w(t)\int_0^t (t-\tau)^{-\gamma}\tau^{-a}\,d\tau
    \le C t^a t^{1-a-\gamma}
    = Ct^{\varepsilon/3}\le C.
\end{equation}
For $t>1$, we decompose the integration interval $(0,t)$ at $t/2$. On the initial segment $(0,t/2)$, the kernel factor $\max\{(t-\tau)^{-\cdot},1\}$ is bounded uniformly, while $\int_0^\infty e^{\omega\tau}w(\tau)^{-1}\,d\tau<\infty$ thanks to $\omega<0$. On the second segment $(t/2,t)$, the weight $w(\tau)^{-1}$ is uniformly bounded and $e^{\omega\tau}\le e^{\omega t/2}$, with
\begin{equation}
    \int_0^{t/2}\max\{r^{-\beta},1\}\,dr
    +\int_0^{t/2}\max\{r^{-\gamma},1\}\,dr
    \le C(1+t).
\end{equation}
The late-time contribution is therefore bounded by $Ce^{\omega t/2}(1+t)\le C$, completing the verification of~\eqref{eqn:glb-conv-low}--\eqref{eqn:glb-conv-high}.

The weight factor $e^{2\omega t}w(t)^{-1}$ corresponds to the quadratic nature of the nonlinearity, with the low and high curvature components contributing $e^{\omega t}$ and $e^{\omega t}w(t)^{-1}$, respectively. For any $f\in C((0,\infty);\hi{-\frac12+\varepsilon})$, the semigroup bound~\eqref{eqn:lambda-semigroup-est} combined with~\eqref{eqn:glb-conv-low}--\eqref{eqn:glb-conv-high} yields
\begin{equation}\label{eqn:global-duhamel-est}
    \norm{\int_0^\cdot e^{(\cdot-\tau)\Lambda}\p_s f(\tau)\,d\tau}_Y
    \le C\sup_{t>0}e^{-2\omega t}w(t)
        \norm{f(t)}_{\hi{-\frac12+\varepsilon}},
\end{equation}
whenever the right-hand side is finite. Time continuity follows by the same arguments as in Proposition~\ref{prop:local-fixed-exponent}.

\emph{Invariance and contraction.}
For any $\kappa\in\mc B_M$, the definition of the $Y$-norm~\eqref{eqn:Y-norm-intro} gives
\begin{equation}
    \norm{\kappa(t)}_{\hr{-\frac12+\delta}}\le Me^{\omega t}\quad(t\ge0),
    \quad
    \norm{\kappa(t)}_{\hr{\frac32}}\le Mw(t)^{-1}e^{\omega t}\quad(t>0).
\end{equation}
Assume $M\le\min\{1,(2C_{J,\delta})^{-1}\}$. Because the straight reference filament satisfies $\abs{\kappa_\infty}_*=1$ for $\kappa_\infty\equiv0$, smallness of the low norm keeps the reconstructed curves $C^1$-close to the straight segment. Thus, by~\eqref{eqn:arc-chord-C1-continuity},
\begin{equation}\label{eqn:global-ball-geometry}
    \abs{\rho\xka{1}(t)+(1-\rho)\xka{2}(t)}_*
    \ge1-C_{J,\delta}M\ge\frac12,
    \quad 0\le\rho\le1,
\end{equation}
for all $\kappa_1,\kappa_2\in\mc B_M$ and $t\ge0$. In particular, the uniform coercivity bound~\eqref{eqn:inverse-uniform-bound} holds uniformly in time and across all $\kappa\in\mc B_M$. Absorbing $C_{M,1/2}\le C_{1,1/2}$ into a generic constant $C$, estimates~\eqref{eqn:N-pointwise-bound} and~\eqref{eqn:N-pointwise-difference} imply, for all $t>0$,
\begin{equation}\label{eqn:global-nonlinear-bounds}
    \begin{aligned}
        \norm{N(\kappa(t))}_{\hi{-\frac12+\varepsilon}}
        &\le CM^2e^{2\omega t}w(t)^{-1},\\
        \norm{N(\kappa_1(t))-N(\kappa_2(t))}_{\hi{-\frac12+\varepsilon}}
        &\le CMe^{2\omega t}w(t)^{-1}\norm{\kappa_1-\kappa_2}_Y.
    \end{aligned}
\end{equation}
In this global setting, quadratic smallness in $M$ replaces the small-time factor $T^{\varepsilon/3}$ of the local argument. Applying~\eqref{eqn:global-duhamel-est} to~\eqref{eqn:global-nonlinear-bounds} and incorporating the linear bound $\norm{e^{\cdot\Lambda}\kappa_0}_Y\le C\norm{\kappa_0}_{\hr{-\frac12+\delta}}$ from~\eqref{eqn:lambda-semigroup-est}, we arrive at
\begin{equation}\label{eqn:global-map-bounds}
    \begin{aligned}
        \norm{\Psi_{\kappa_0}\kappa}_Y
        &\le C\norm{\kappa_0}_{\hr{-\frac12+\delta}}+CM^2,\\
        \norm{\Psi_{\kappa_0}\kappa_1-\Psi_{\kappa_0}\kappa_2}_Y
        &\le CM\norm{\kappa_1-\kappa_2}_Y.
    \end{aligned}
\end{equation}
Fix $C\ge1$ such that~\eqref{eqn:global-map-bounds} holds, and select $M_*>0$ satisfying both the geometric and contraction thresholds:
\begin{equation}
    M_*\le\min\{1,(2C_{J,\delta})^{-1}\},\quad CM_*\le\frac12,
\end{equation}
Set the ball radius $M=2C\norm{\kappa_0}_{\hr{-\frac12+\delta}}$ and the corresponding initial data threshold $M_0=M_*/(2C)$. If $\norm{\kappa_0}_{\hr{-\frac12+\delta}}\le M_0$, then $M\le M_*$ and
\begin{equation}
    C\norm{\kappa_0}_{\hr{-\frac12+\delta}}+CM^2\le M,
    \quad CM\le\frac12.
\end{equation}
Therefore, $\Psi_{\kappa_0}$ defines a strict contraction mapping on $\mc B_M$. The Banach fixed-point theorem ensures a unique fixed point $\kappa\in\mc B_M$, which represents a global mild solution in view of the non-degeneracy condition~\eqref{eqn:global-ball-geometry}. Uniqueness in the full class $Y$ follows from Theorem~\ref{thm:main-local} and Proposition~\ref{lem:restart-uniqueness}. The bound $\norm{\kappa}_Y\le M=2C\norm{\kappa_0}_{\hr{-\frac12+\delta}}$ directly delivers the decay estimate~\eqref{eqn:main-global-decay}.
\end{proof}

\subsection{Continuation and asymptotic alternatives}
\label{subsec:global-behavior-proof}

We now prove the continuation criterion, asymptotic alternatives, and energy gap in Theorem~\ref{thm:main-global}(2), combining energy monotonicity, compactness of time translates, and zero-dissipation rigidity.

\begin{proof}[Proof of Theorem~\ref{thm:main-global}(2)]
\emph{Positive-time energy bound and maximal solution.}
Fix $0<\delta\le\frac12$ and $\kappa_0\in\mc U_\delta$. The maximal trajectory constructed in Theorem~\ref{thm:main-local} and Proposition~\ref{lem:restart-uniqueness} satisfies the regularity class~\eqref{eqn:main-local-class} on every interval $[0,T]$ with $T<\tau_{\text{max}}(\kappa_0)$.

By the energy identity~\eqref{eqn:main-energy-positive}, the mapping $t\mapsto\mc E(t)$ is non-increasing for $t>0$. Consequently, for every $0<a<\tau_{\text{max}}(\kappa_0)$,
\begin{equation}\label{eqn:global-positive-time-bound}
    \sup_{a\le t<\tau_{\text{max}}(\kappa_0)}
    \norm{\kappa(t)}_{L^2(\R)}
    \le\norm{\kappa(a)}_{L^2(\R)}.
\end{equation}
The asymptotic energy limit $\mc E_* = \lim_{t\uparrow\tau_{\text{max}}(\kappa_0)}\mc E(t)\in[0,\infty)$ is therefore well-defined. Estimate~\eqref{eqn:global-positive-time-bound} provides the uniform $L^2$ control needed for the compactness arguments that follow.

\emph{Finite maximal time.}
Suppose that $\tau_{\text{max}}(\kappa_0)<\infty$ and that the arc-chord constant does not degenerate to zero. Then there exist a constant $m>0$ and an increasing sequence of times $t_j\uparrow\tau_{\text{max}}(\kappa_0)$ such that $\abs{\kappa(t_j)}_*\ge m$. Fixing any $a\in(0,\tau_{\text{max}}(\kappa_0))$, estimate~\eqref{eqn:global-positive-time-bound} bounds the sequence $\{\kappa(t_j)\}$ uniformly in $\wt L^2(I)$. Fix an auxiliary exponent $0<\delta_0<\frac12$ and set $s_0=-\frac12+\delta_0$. By the compact Sobolev embedding $\wt L^2(I)\hookrightarrow\thi{s_0}$, a subsequence satisfies
\begin{equation*}
    \kappa(t_j)\longrightarrow\kappa_*
    \quad\text{in }\thi{s_0}.
\end{equation*}
By $C^1$-continuity of the arc-chord functional~\eqref{eqn:arc-chord-C1-continuity}, $\abs{\kappa_*}_*\ge m$, so $\kappa_*\in\mc U_{\delta_0}$ remains an admissible initial state. Because convergence holds in the topology of the local well-posedness theory, the data $\kappa(t_j)$ eventually lie in a single local existence neighborhood of $\kappa_*$. Theorem~\ref{thm:main-local} at exponent $\delta_0$ therefore guarantees a common lifespan $T_*>0$ for all sufficiently large $j$. By Proposition~\ref{lem:restart-uniqueness}, these local solutions coincide with restarts of the original maximal solution. For large $j$, we have $t_j+T_*>\tau_{\text{max}}(\kappa_0)$, which contradicts the maximality of $\tau_{\text{max}}(\kappa_0)$. Hence
\begin{equation}
    \abs{\kappa(t)}_*\longrightarrow0
    \quad\text{as }t\uparrow\tau_{\text{max}}(\kappa_0).
\end{equation}
Applying the sequential energy barrier~\eqref{eqn:sequential-energy-barrier} along any sequence $t_j\uparrow\tau_{\text{max}}(\kappa_0)$ and invoking the uniform $L^2$ bound~\eqref{eqn:global-positive-time-bound}, we conclude that $\mc E_*\ge\mc E_{\mathrm{ac}}$.

\emph{Rigidity of nondegenerate time translates.}
Next, consider a global solution. We prove the implication~\eqref{eqn:uniform-geometry-zero-energy}: an eventual uniform arc-chord lower bound forces the limiting energy to vanish, $\mc E_*=0$. Assume that
\begin{equation}\label{eqn:global-uniform-geometry}
    \liminf_{t\to\infty}\abs{\kappa(t)}_*>0.
\end{equation}
Fix $0<\delta_0<\frac12$ and set $s_0=-\frac12+\delta_0$. Let $t_j\to\infty$ be an arbitrary sequence of times. Discarding finitely many terms, we have $\abs{\kappa(t_j)}_*\ge m_0$ for some $m_0>0$. The uniform $L^2$ bound~\eqref{eqn:global-positive-time-bound}, weak compactness in $\wt L^2(I)$, and the compact embedding into $\thi{s_0}$ extract a subsequence satisfying
\begin{equation}\label{eqn:global-translate-data}
    \kappa(t_j)\rightharpoonup\kappa_*
    \quad\text{in }\wt L^2(I),
    \quad
    \kappa(t_j)\longrightarrow\kappa_*
    \quad\text{in }\thi{s_0}.
\end{equation}
Equation~\eqref{eqn:arc-chord-C1-continuity} ensures that $\abs{\kappa_*}_*\ge m_0$.

Applying Theorem~\ref{thm:main-local} around $\kappa_*$ at exponent $\delta_0$, the mild solutions $\vph_j$ and $\vph_*$ with respective initial data $\kappa(t_j)$ and $\kappa_*$ exist on a common time interval $[0,T]$ for all sufficiently large $j$, with
\begin{equation}\label{eqn:global-translate-solutions}
    \norm{\vph_j-\vph_*}_{X_T^{(\delta_0)}}\longrightarrow0.
\end{equation}
By the shifted Duhamel identity~\eqref{eqn:shifted-duhamel} and the restart uniqueness principle in Proposition~\ref{lem:restart-uniqueness}, the sequence $\vph_j$ represents exact time translates:
\begin{equation}\label{eqn:time-shift-global}
    \vph_j(t)=\kappa(t_j+t),
    \quad 0\le t\le T.
\end{equation}
At each fixed positive time $t\in(0,T]$, the weighted convergence~\eqref{eqn:global-translate-solutions} implies strong convergence in $\thi{3/2}$, and hence in $L^2(I)$. Since $\mc E(t_j+t)\to\mc E_*$, the limiting trajectory $\vph_*$ possesses constant energy:
\begin{equation}\label{eqn:global-limit-energy}
    \frac12\norm{\vph_*(t)}_{L^2(\R)}^2
    =\lim_{j\to\infty}\mc E(t_j+t)
    =\mc E_*,
    \quad 0<t\le T.
\end{equation}
Let $(\bm X_*,\bm u_*,p_*,\lambda_*)$ be the hydrodynamic reconstruction associated with $\vph_*$. The positive-time energy identity~\eqref{eqn:main-energy-positive} then implies, for every $0<a<b\le T$,
\begin{equation}
    \int_a^b\int_{\R^2}
    \abs{\grad_S\bm u_*(\bm x,t)}^2\,d\bm x\,dt=0.
\end{equation}
The symmetric gradient $\grad_S\bm u_*$ therefore vanishes almost everywhere in $\R^2\times(a,b)$. At every regular time, the stationary rigidity argument in Proposition~\ref{prop:std-stt} shows that $\vph_*(t)\equiv0$. By continuity on $(0,T]$, we obtain
\begin{equation}
    \vph_*(t)=0,
    \quad 0<t\le T.
\end{equation}
Combined with~\eqref{eqn:global-limit-energy}, this verifies the rigidity assertion:
\begin{equation}\label{eqn:uniform-geometry-zero-energy}
    \liminf_{t\to\infty}\abs{\kappa(t)}_*>0
    \quad\Longrightarrow\quad
    \mc E_*=0.
\end{equation}

\emph{Energy gap and asymptotic alternatives.}
Now assume the solution is global with $\mc E_*<\mc E_{\mathrm{ac}}$. Then there exists $t_0>0$ such that $\mc E(t_0)<\mc E_{\mathrm{ac}}$. By monotonicity~\eqref{eqn:main-energy-positive}, $\mc E(t)\le \mc E(t_0)$ for all $t\ge t_0$. Lemma~\ref{lem:energy-arc-chord} applied at each $t\ge t_0$ gives
\begin{equation}\label{eqn:eventual-energy-geometry}
    \abs{\kappa(t)}_*
    \ge\cos\sqrt{\mc E(t)}
    \ge\cos\sqrt{\mc E(t_0)}>0,
    \quad t\ge t_0.
\end{equation}
This establishes~\eqref{eqn:global-uniform-geometry}, so the rigidity implication~\eqref{eqn:uniform-geometry-zero-energy} forces $\mc E_*=0$. Together with the finite-time blowup analysis, this completes the proof of the energy gap~\eqref{eqn:main-energy-gap}.

If a global solution satisfies $\mc E_*>0$, implication~\eqref{eqn:uniform-geometry-zero-energy} immediately yields $\liminf_{t\to\infty}\abs{\kappa(t)}_*=0$, and the gap inequality~\eqref{eqn:main-energy-gap} enforces $\mc E_*\ge\mc E_{\mathrm{ac}}$.

Conversely, if $\mc E_*=0$, the finite-time criterion already rules out finite lifespan, ensuring global existence with
\begin{equation}
    \norm{\kappa(t)}_{L^2(\R)}\longrightarrow0.
\end{equation}
For all sufficiently large $t$, Lemma~\ref{lem:energy-arc-chord} gives
\begin{equation}
    1\ge\abs{\kappa(t)}_*
    \ge\cos\sqrt{\mc E(t)}.
\end{equation}
which implies $\abs{\kappa(t)}_*\to1$ as $t\to\infty$. The two asymptotic alternatives are mutually disjoint and exhaust all global trajectories.

\emph{Eventual exponential decay.}
Once the bending energy drops sufficiently low, the trajectory enters the small-data regime and the exponentially weighted contraction theory takes effect. Specifically, assuming $\mc E_*=0$ and fixing $\omega\in(\omega_\Lambda,0)$, we choose $T_\omega>0$ large enough that $\norm{\kappa(T_\omega)}_{L^2(\R)}$ falls below the smallness threshold in Theorem~\ref{thm:main-global}(1) with exponent $\delta=\frac12$. By Proposition~\ref{lem:restart-uniqueness}, the subsequent trajectory coincides with the unique small-data solution $\kappa(T_\omega+\cdot)$, and estimate~\eqref{eqn:main-global-decay} provides the exponential decay. Curve reconstruction then establishes exponential relaxation toward a straight filament modulo rigid motions.
\end{proof}

\begin{proof}[Proof of Theorem~\ref{thm:main-global}(3)]
For initial data satisfying $\mc E(0)<\mc E_{\mathrm{ac}}$ (or $\mc E(t_0)<\mc E_{\mathrm{ac}}$ for $t_0>0$), Lemma~\ref{lem:energy-arc-chord} guarantees $\abs{\kappa_0}_*>0$, so the maximal mild solution exists. Proposition~\ref{prop:main-energy-identity} implies that $\mc E(t)\le\mc E(t_0)<\mc E_{\mathrm{ac}}$ for all $t_0\le t<\tau_{\mathrm{max}}$, including $t_0=0$ whenever $\kappa_0\in\wt L^2(I)$. By the energy gap~\eqref{eqn:main-energy-gap}, the limiting energy must be zero, $\mc E_*=0$.
Part~\emph{(2)} then guarantees global existence and exponential straightening. The lower bound~\eqref{eqn:main-post-threshold-arc-chord} follows directly from~\eqref{eqn:energy-arc-chord-bound} and energy monotonicity~\eqref{eqn:main-energy-positive}.
\end{proof}

\begin{proof}[Proof of Corollary~\ref{cor:main-straightening}]
Fix $\omega\in(\omega_\Lambda,0)$. By Theorem~\ref{thm:main-global}(2), there exists $T>0$ such that
\begin{equation*}
    \norm{\kappa(t)}_{\hr{\frac32}}\le Ce^{\omega t},
    \quad |\kappa(t)|_*\ge\frac12,\quad t\ge T.
\end{equation*}
Let $\bm v=(\bm u|_{\Gamma})\circ\bm X$ denote the reconstructed filament velocity. Combining~\eqref{eqn:inverse-uniform-bound}, \eqref{eqn:tension-bound}, \eqref{eqn:product-supported}, and~\eqref{eqn:stokes-layer-bound} as in~\eqref{eqn:reconstruction-time-bound}, we obtain
\begin{equation}\label{eqn:reconstructed-velocity-decay}
    \norm{\bm v(t)}_{\hi{\frac12}}
    +\norm{\p_t\theta(t)}_{\hi{-\frac12}}
    \le C\norm{\bm q_\kappa(t)}_{\hr{\frac12}}
    \le C\norm{\kappa(t)}_{\hr{\frac32}}
    \le Ce^{\omega t},\quad t\ge T.
\end{equation}
The curvature decay and the uniform lower bound on the arc-chord constant ensure that the constant $C$ is independent of time. By the centroid evolution equation~\eqref{eqn:X-bar},
\begin{equation*}
    |\p_t\overline{\bm X}(t)|
    =\left|\frac12\int_I\bm v(s,t)\,ds\right|
    \le Ce^{\omega t}.
\end{equation*}
Integrating in time establishes convergence to an asymptotic center of mass $\overline{\bm X}_\infty$ with $|\overline{\bm X}(t)-\overline{\bm X}_\infty|\le Ce^{\omega t}$.

To track the asymptotic orientation, we use a compactly supported spatial average because $\p_t\theta$ belongs to a negative Sobolev space. Choose a non-negative mollifier $\phi\in C_c^\infty(I)$ with $\int_I\phi=1$, and set $a(t)=\int_I\phi(s)\theta(s,t)\,ds$. Since $\phi\in\thi{\frac12}$, estimate~\eqref{eqn:reconstructed-velocity-decay} implies
\begin{equation*}
    |a'(t)|=|\dual{\p_t\theta(t)}{\phi}|\le Ce^{\omega t}.
\end{equation*}
Thus $a(t)$ converges exponentially to an asymptotic angle $\theta_\infty\in\R$, with $|a(t)-\theta_\infty|\le Ce^{\omega t}$.
Furthermore, using $\p_s\theta=\kappa$ and $\int_I\phi=1$, we have
\begin{equation*}
    \norm{\theta(\cdot,t)-a(t)}_{C(\overline I)}
    \le\norm{\kappa(t)}_{L^1(I)}\le Ce^{\omega t}.
\end{equation*}
It follows that the full angle field $\theta(\cdot,t)$, and in particular the average orientation $\bar\theta(t)$, converges to $\theta_\infty$ at rate $e^{\omega t}$. The tangent relation $\p_s\bm X=(\cos\theta,\sin\theta)$ combined with the convergence of the centroid then proves the claimed $C^1$ convergence toward the straight filament $\bm X_\infty(s) = \overline{\bm X}_\infty + s(\cos\theta_\infty,\sin\theta_\infty)$. By continuity of the trajectory on $[0,T]$, the estimate extends to all $t\ge0$ upon adjusting $C_\omega$.
\end{proof}

\appendix

\section{Proofs of auxiliary estimates}
\label{app:operator-estimates}
\label{app:divided-differences}

This appendix collects the proofs of the geometric, commutator, and boundary integral estimates invoked in Sections~\ref{sec:formulation} and~\ref{sec:tension}.

\subsection{Geometric estimates}
\label{app:geometry-estimates}

\begin{proof}[Proof of Lemma~\ref{tau-n-est}]
The argument reduces to composition estimates for the radial frame vector $\bm e_r$ on $H^{s+1}(I)$. The assumption on $s$ ensures $\frac{1}{2}<s+1\le1$, so $H^{s+1}(I)$ is a Banach algebra. For $f\in H^{s+1}(I)$ and $-\frac{1}{2}<s<0$, the Lipschitz continuity of $\bm e_r$ yields
\begin{equation}
\begin{aligned}
    [\bm e_r\circ f]_{\hi{s+1}}^2
    &= \int_I\int_I
    \frac{\abs{\bm e_r(f(x))-\bm e_r(f(y))}^2}
    {\abs{x-y}^{1+2(s+1)}}\,dx\,dy \\
    &\le \int_I\int_I
    \frac{\abs{f(x)-f(y)}^2}
    {\abs{x-y}^{1+2(s+1)}}\,dx\,dy
    = [f]_{\hi{s+1}}^2.
\end{aligned}
\end{equation}
At $s=0$, the chain rule gives the corresponding estimate:
\begin{equation}
    \p_s(\bm e_r\circ f)
    =\bm e_r'(f)\p_s f,
    \quad
    \norm{\p_s(\bm e_r\circ f)}_{L^2(I)}
    \le \norm{\p_s f}_{L^2(I)}.
\end{equation}
Since $\abs{\bm e_r}\equiv 1$,
$\norm{\bm e_r\circ f}_{L^2(I)}$ is uniformly bounded. Thus
\begin{equation}\label{eqn:er-comp-Hr}
    [\bm e_r\circ f]_{\hi{s+1}}
    \le [f]_{\hi{s+1}},
    \quad
    \norm{\bm e_r\circ f}_{\hi{s+1}}
    \le C\paren{1+\norm{f}_{\hi{s+1}}}.
\end{equation}
Equation~\eqref{eqn:er-comp-Hr} also holds for $\bm e_\theta\coloneqq\bm e_r'=Q_{\frac{\pi}{2}}\bm e_r$, since
$\bm e_\theta$ is a fixed rotation of $\bm e_r$.

To estimate differences, write
\begin{equation}
    \bm e_r(f_1)-\bm e_r(f_2)
    = \paren{f_1-f_2}\bm G(f_1,f_2),
\end{equation}
where
\begin{equation}
    \bm G(f_1,f_2)
    \coloneqq \int_0^1 \bm e_\theta\paren{\eta f_1+(1-\eta)f_2}\,d\eta.
\end{equation}
Applying~\eqref{eqn:er-comp-Hr} to $\bm e_\theta$ gives
\begin{equation}
    \norm{\bm G(f_1,f_2)}_{\hi{s+1}}
    \le C\paren{1+\norm{f_1}_{\hi{s+1}}+\norm{f_2}_{\hi{s+1}}}.
\end{equation}
Since $H^{s+1}(I)$ is an algebra, it follows that
\begin{equation}\label{eqn:er-comp-diff}
\begin{aligned}
    \norm{\bm e_r(f_1)-\bm e_r(f_2)}_{\hi{s+1}}
    &\le C\paren{1+\norm{f_1}_{\hi{s+1}}+\norm{f_2}_{\hi{s+1}}}
    \norm{f_1-f_2}_{\hi{s+1}}.
\end{aligned}
\end{equation}

Equation~\eqref{eqn:theta-Hs+1} gives
\begin{equation}
    \norm{\thk}_{\hi{s+1}} \le C\norm{\kappa}_{\hr{s}},
    \quad
    \norm{\thka{1}-\thka{2}}_{\hi{s+1}}
    \le C\norm{\kappa_1-\kappa_2}_{\hr{s}}.
\end{equation}
Since $\xk'=\bm e_r(\thk)$, estimates~\eqref{eqn:er-comp-Hr}
and~\eqref{eqn:er-comp-diff} imply
\begin{equation}
    \norm{\xk'}_{\hi{s+1}}
    \le C\paren{1+\norm{\kappa}_{\hr{s}}},
    \quad
    \norm{\xka{1}'-\xka{2}'}_{\hi{s+1}}
    \le C\paren{1+M}\norm{\kappa_1-\kappa_2}_{\hr{s}}.
\end{equation}
The seminorm estimate in~\eqref{eqn:er-comp-Hr} also gives the sharper bound
\begin{equation}
    [\xk']_{\hi{s+1}}
    \le [\thk]_{\hi{s+1}}
    \le C\norm{\kappa}_{\hr{s}}.
\end{equation}

It remains to estimate the second derivative. Taking $f_1=\thk$ and $f_2=0$
in~\eqref{eqn:er-comp-diff} gives
\begin{equation}
    \norm{\xk'-\bm e_r(0)}_{\hi{s+1}}
    \le C\paren{1+\norm{\thk}_{\hi{s+1}}}\norm{\thk}_{\hi{s+1}}
    \le C\paren{1+\norm{\kappa}_{\hr{s}}}\norm{\kappa}_{\hr{s}}.
\end{equation}
Differentiating gives
\begin{equation}
    \norm{\xk''}_{\hi{s}}
    \le C \norm{\xk'-\bm e_r(0)}_{\hi{s+1}}
    \le C\paren{1+\norm{\kappa}_{\hr{s}}}\norm{\kappa}_{\hr{s}}.
\end{equation}
Applying the same differentiation estimate to the difference gives
\begin{equation}
    \norm{\xka{1}''-\xka{2}''}_{\hi{s}}
    \le C \norm{\xka{1}'-\xka{2}'}_{\hi{s+1}}
    \le C\paren{1+M}\norm{\kappa_1-\kappa_2}_{\hr{s}}.
\end{equation}
\end{proof}

\begin{proof}[Proof of Lemma~\ref{lem:Xk-holder}]
We first employ the Sobolev regularity of $\kappa$ to bound $\theta_\kappa$ in H\"older spaces, and then invoke the reconstruction identity $\xk'=\bm e_r(\theta_\kappa)$. Equation~\eqref{eqn:theta-Hs+1} gives
\begin{equation}
    \norm{\theta_\kappa}_{\hi{s+1}} \le C\norm{\kappa}_{\hr{s}}.
\end{equation}
The assumptions on $s$ and $\alpha$ imply
$s+1>m+\alpha+\frac{1}{2}$. The one-dimensional Sobolev embedding gives
\begin{equation}\label{eqn:Xk-holder-higher-theta}
    \norm{\theta_\kappa}_{\ci{m,\alpha}} \le C\norm{\kappa}_{\hr{s}}.
\end{equation}

To estimate the reconstructed curve, Fa\`a di Bruno's formula shows that,
for $0\le j\le m$, each derivative
$\xk^{(j+1)}=\p_s^j(\bm e_r(\theta_\kappa))$ is a finite sum of terms of the
form
\begin{equation}
    G\paren{\theta_\kappa}\prod_{\ell=1}^j \paren{\p_s^\ell\theta_\kappa}^{\beta_\ell},
    \quad
    \sum_{\ell=1}^j \ell\beta_\ell = j,
\end{equation}
where $G$ is a smooth bounded function. Since $\ci{m,\alpha}$ is an algebra,
estimate~\eqref{eqn:Xk-holder-higher-theta} gives
\begin{equation}
    \norm{\xk'}_{\ci{m,\alpha}}
    \le C\paren{1+\norm{\theta_\kappa}_{\ci{m,\alpha}}}^{m+1}
    \le C\paren{1+\norm{\kappa}_{\hr{s}}}^{m+1}.
\end{equation}
Since $\xk(-1)=0$, it follows that
\begin{equation}
    \norm{\xk}_{\ci{m+1,\alpha}}
    \le C \norm{\xk'}_{\ci{m,\alpha}}
    \le C\paren{1+\norm{\kappa}_{\hr{s}}}^{m+1}.
\end{equation}
This proves~\eqref{eqn:Xk-holder-higher}.

When $m=0$, the fact that $\bm e_r$ is $1$-Lipschitz gives, for $x,y\in I$,
\begin{equation}
    \abs{\xk'(x)-\xk'(y)}
    = \abs{\bm e_r(\theta_\kappa(x))-\bm e_r(\theta_\kappa(y))}
    \le \abs{\theta_\kappa(x)-\theta_\kappa(y)}.
\end{equation}
By~\eqref{eqn:Xk-holder-higher-theta},
\begin{equation}
    [\xk']_{\ci{0,\alpha}} \le [\theta_\kappa]_{\ci{0,\alpha}}
    \le C\norm{\kappa}_{\hr{s}},
\end{equation}
which proves~\eqref{eqn:Xk-holder}.

For the difference estimate, let
$\theta_i=\theta_{\kappa_i}$ for $i=1,2$. By linearity of $J$ and
\eqref{eqn:Xk-holder-higher-theta},
\begin{equation}
    \norm{\theta_i}_{\ci{m,\alpha}} \le C M,
    \quad
    \norm{\theta_1-\theta_2}_{\ci{m,\alpha}}
    \le C \norm{\kappa_1-\kappa_2}_{\hr{s}}.
\end{equation}
The fundamental theorem of calculus gives
\begin{equation}
    \xka{1}'-\xka{2}'
    = \paren{\theta_1-\theta_2}\bm G(\theta_1,\theta_2),
    \quad
    \bm G(\theta_1,\theta_2)
    \coloneqq \int_0^1 \bm e_\theta\paren{\eta\theta_1+(1-\eta)\theta_2}\,d\eta.
\end{equation}
The smoothness of $\bm e_\theta$ and the algebra property of
$\ci{m,\alpha}$ imply
\begin{equation}
    \norm{\bm G(\theta_1,\theta_2)}_{\ci{m,\alpha}} \le C\paren{1+M}^{m+1}.
\end{equation}
Combining the last three estimates gives
\begin{equation}
    \norm{\xka{1}'-\xka{2}'}_{\ci{m,\alpha}}
    \le C\paren{1+M}^{m+1}\norm{\kappa_1-\kappa_2}_{\hr{s}}.
\end{equation}
Since $(\xka{1}-\xka{2})'=\xka{1}'-\xka{2}'$ and
$\xka{1}(-1)=\xka{2}(-1)=0$, it follows that
\begin{equation}
    \norm{\xka{1}-\xka{2}}_{\ci{m+1,\alpha}}
    \le C\paren{1+M}^{m+1}\norm{\kappa_1-\kappa_2}_{\hr{s}},
\end{equation}
which proves~\eqref{eqn:Xk-holder-higher-lip}.
\end{proof}

\subsection{Commutator estimates}
\label{app:commutator-estimates}

\begin{proof}[Proof of Lemma~\ref{lem:comm-psH}]
It suffices to consider $g\in C_c^\infty(I)$ and conclude by density. Set
$T_fg\coloneqq[\p_s\mc H,f]g$. Differentiating the kernel in~\eqref{eqn:finite-hilbert-definition} gives
\begin{equation}\label{eq:comm-kernel}
    (T_fg)(s)
    =\frac{1}{\pi}\operatorname{p.v.}\int_I
    \frac{f(s)-f(\eta)}{(s-\eta)^2}\,g(\eta)\,d\eta.
\end{equation}
To separate the singular Hilbert-transform part, define
$V_0(s,\eta)=\frac{f(s)-f(\eta)}{s-\eta}$. The identity
\begin{equation}\label{eq:kernel-decomp}
    \frac{f(s)-f(\eta)}{(s-\eta)^2}
    =\p_\eta V_0(s,\eta)+\frac{\p_s f(\eta)}{s-\eta}
\end{equation}
gives
\begin{equation}\label{eq:comm-split}
    T_fg=I_1+I_2, \quad I_1(s)\coloneqq\frac{1}{\pi}\int_I \p_\eta V_0(s,\eta)g(\eta)\,d\eta, \quad
    I_2\coloneqq\mc H(\p_s f\cdot g).
\end{equation}

We first estimate the smoother term $I_1$ at two endpoints. Choose
$0<\alpha<\delta$. The embedding
$H^{\frac52+\delta}(I)\hookrightarrow C^{2,\alpha}(I)$ gives
\begin{equation}
    \abs{\p_sV_0(s,\eta)}+\abs{\p_\eta V_0(s,\eta)}
    \le C\norm{f}_{H^{\frac52+\delta}(I)}.
\end{equation}
Thus $I_1$ has a bounded kernel and
\begin{equation}
    \norm{I_1}_{L^2(I)}
    \le C\norm{f}_{H^{\frac52+\delta}(I)}\norm{g}_{L^2(\R)}.
\end{equation}

For the $H^1$ endpoint, integration by parts in $\eta$ gives
\begin{equation}
    I_1(s)=-\frac{1}{\pi}\int_I
    V_0(s,\eta)\p_\eta g(\eta)\,d\eta,
\end{equation}
where the boundary terms vanish because $g\in C_c^\infty(I)$. Since
\begin{equation}
    \abs{V_0(s,\eta)}+\abs{\p_sV_0(s,\eta)}
    \le C\norm{f}_{H^{\frac52+\delta}(I)},
\end{equation}
we obtain
\begin{equation}
    \norm{I_1}_{H^1(I)}
    \le C\norm{f}_{H^{\frac52+\delta}(I)}\norm{g}_{H^1(\R)}.
\end{equation}

Interpolation between the bounds
$I_1:L^2(I)\to L^2(I)$ and $I_1:\thi{1}\to H^1(I)$ gives
\begin{equation}\label{eqn:commutator-smooth-bound}
    \norm{I_1}_{\hi{\frac12+\varepsilon}}
    \le C\norm{f}_{H^{\frac52+\delta}(I)}
    \norm{g}_{\hr{\frac12+\varepsilon}}.
\end{equation}

For $I_2$, let $F\in H^{\frac52+\delta}(\R)$ be a bounded extension of
$f$. Because $g$ is supported in $\overline I$, the function
$(\p_sF)g$ restricts to $(\p_sf)g$ on $I$. Lemma~\ref{point_mult_sobo}
therefore gives
\begin{equation}
    \norm{(\p_sF)g}_{\hr{\frac12+\varepsilon}}
    \le C\norm{f}_{H^{\frac52+\delta}(I)}
    \norm{g}_{\hr{\frac12+\varepsilon}}.
\end{equation}
Lemma~\ref{Hf-bounded} then implies
\begin{equation}\label{eqn:commutator-hilbert-bound}
    \norm{I_2}_{H^{\frac12+\varepsilon}(I)}
    \le C\norm{f}_{H^{\frac52+\delta}(I)}
    \norm{g}_{\hr{\frac12+\varepsilon}}.
\end{equation}
Combining~\eqref{eqn:commutator-smooth-bound} and~\eqref{eqn:commutator-hilbert-bound} proves
\eqref{eq:comm-psH-bound}; density gives the result for every
$g\in\thi{\frac12+\varepsilon}$.
\end{proof}

\subsection{Divided differences and boundary integral estimates}
\label{app:boundary-integral-estimates}

We establish the divided difference bounds and boundary integral estimates on the reference interval $I=[-1,1]$, retaining explicit tracking of the arc-chord constant throughout.

\begin{lemma}\label{lem:diff-est}
Let $\bm Z \in \ci{1,\alpha}$ for $\alpha\in(0,1)$. Then, for
$s,\eta\in I$ with $s\ne\eta$,
\begin{align}
    &\abs{\Delta \bm Z} \le \norm{\bm Z' }_{\ci{0}} |s-\eta|, \\
    &\abs{\bm Z '(\eta) - \frac{\Delta \bm Z }{s-\eta}}, \abs{\bm Z '(s) - \frac{\Delta \bm Z }{s-\eta}} \le [\bm Z ']_{\ci{0,\alpha}} |s-\eta|^\alpha.\label{eqn:divided-difference-remainder}
\end{align}
If $h\in\R$, $s+h\in I$, and $s+h\ne\eta$, then
\begin{align}
    &|\Delta_{s,h}\Delta \bm Z | \le C\norm{\bm Z' }_{\ci{0}} \abs{h}, \\
    &\abs{\Delta_{s,h} \paren{\bm Z'(s) - \frac{\Delta \bm Z}{s-\eta}} }   \le C [\bm Z ']_{\ci{0,\alpha}}\paren{\abs{h}^\alpha+\abs{h}\abs{s-\eta}^{\alpha-1}},\\
    &\abs{\Delta_{s,h} \paren{\bm Z'(\eta) - \frac{\Delta \bm Z}{s-\eta}} }  \le C [\bm Z ']_{\ci{0,\alpha}}\abs{h}\abs{s-\eta}^{\alpha-1}.
\end{align}
In particular, if $\abs{h}\le \frac{1}{2}\abs{s-\eta}$, then
\begin{equation}\label{eqn:diff-est1}
    \abs{\Delta_{s,h} \paren{\bm Z'(s) - \frac{\Delta \bm Z}{s-\eta}} }, \abs{\Delta_{s,h} \paren{\bm Z'(\eta) - \frac{\Delta \bm Z}{s-\eta}} }  \le C [\bm Z ']_{\ci{0,\alpha}} \abs{h}^\alpha ,
\end{equation}
\end{lemma}
\begin{proof}
The fundamental theorem of calculus gives
\begin{equation}
    \abs{\Delta\bm Z} = \abs{\paren{s-\eta} \int_0^1 \bm Z'(\eta + \rho(s-\eta))\,d\rho  } \le \norm{\bm Z'}_{\ci{0}} \abs{s-\eta},
\end{equation}
\begin{equation}
    \abs{\frac{\Delta \bm Z}{s-\eta} - \bm Z'(\eta)} = \abs{ \int_0^1 \bm Z'\paren{\eta + \rho(s-\eta)} - \bm Z'(\eta)\,d\rho } \le C [\bm Z']_{\ci{0,\alpha}} \abs{s-\eta}^\alpha .
\end{equation}
\begin{equation}
    \abs{\Delta_{s,h} \Delta \bm Z} = \abs{ h \int_0^1 \bm Z'(s+\rho h) \,d\rho } \le \norm{\bm Z'}_{\ci{0}} \abs{h}.
\end{equation}
First suppose that $\abs h\le\frac12\abs{s-\eta}$. Since
\begin{equation}
    \p_s\paren{\frac{\Delta\bm Z}{s-\eta}} = \frac{1}{s-\eta}\paren{\bm Z'(s) - \frac{\Delta\bm Z}{s-\eta}},
\end{equation}
we have
\begin{equation}
\begin{aligned}
    \abs{\Delta_{s,h}\paren{\frac{\Delta\bm Z}{s-\eta}}} &= \abs{h \int_0^1\p_s\paren{\frac{\Delta\bm Z}{s-\eta}} (s+\rho h, \eta) \,d\rho } \le C[\bm Z']_{\ci{0,\alpha}}\abs{h}\abs{s-\eta}^{\alpha-1} ,
\end{aligned}
\end{equation}
Hence
\begin{equation}
\begin{aligned}
    \abs{\Delta_{s,h} \paren{\bm Z'(s) - \frac{\Delta\bm Z}{s-\eta}} } \le \abs{\Delta_{s,h}\bm Z'(s)} + \abs{\Delta_{s,h}\paren{\frac{\Delta\bm Z}{s-\eta}}} \\
    \le C [\bm Z']_{\ci{0,\alpha}}\paren{\abs{h}^\alpha+\abs{h} \abs{s-\eta}^{\alpha-1}}.
\end{aligned}
\end{equation}
\begin{equation}
    \abs{\Delta_{s,h} \paren{\bm Z'(\eta) - \frac{\Delta\bm Z}{s-\eta}} } \le\abs{\Delta_{s,h}\paren{\frac{\Delta\bm Z}{s-\eta}}} \le C [\bm Z']_{\ci{0,\alpha}}\abs{h} \abs{s-\eta}^{\alpha-1}.
\end{equation}
This proves the three finite-difference estimates in this regime, including
\eqref{eqn:diff-est1}. If
$\abs h>\frac12\abs{s-\eta}$, apply~\eqref{eqn:divided-difference-remainder}
at $(s,\eta)$ and $(s+h,\eta)$. Since
$\abs{s-\eta}<2\abs h$ and
$\abs{s+h-\eta}\le3\abs h$, both finite differences are bounded by
$C[\bm Z']_{\ci{0,\alpha}}\abs h^\alpha$. In this regime,
$\abs h^\alpha\le C\abs h\abs{s-\eta}^{\alpha-1}$, which proves the
stated general estimates.
\end{proof}

\begin{lemma}\label{lem-phi-ZV}
Let $V, \bm Z =(Z_1,Z_2)^T\in \ci{1,\alpha}$ for $\alpha\in(0,1)$ with
$\abs{\bm Z}_*>0$. For $s,\eta\in I$ with $s\ne\eta$, define
\begin{equation}
    \phi(s,\eta) = \frac{\Delta V}{\abs{\Delta\bm Z}}.
\end{equation}
Then
\begin{align}
    &\abs{\phi}\le C\abs{\bm Z}_*^{-1}\norm{V'}_{\ci{0,\alpha}},\label{eqn:phi-est-1}\\
    &\abs{\p_s\phi}, \abs{\p_\eta\phi} \le C \abs{\bm Z}_*^{-2}\norm{V'}_{\ci{0,\alpha}}\norm{\bm Z'}_{\ci{0,\alpha}}  \abs{s-\eta}^{\alpha-1}.\label{eqn:phi-est-2}
\end{align}
If $s,\eta,s+h\in I$, $h\in\R$, and
$\abs{s-\eta}\ge2\abs{h}$, then
\begin{align}
    &\abs{\Delta_{s,h}\phi} \le C  \abs{\bm Z}_*^{-2}\norm{V'}_{\ci{0,\alpha}}\norm{\bm Z'}_{\ci{0,\alpha}} \abs{h} \abs{s-\eta}^{\alpha-1},\label{eqn:phi-est-3}\\
    &\abs{\Delta_{s,h}\p_s\phi}, \abs{\Delta_{s,h}\p_\eta\phi} \le C\abs{\bm Z}_*^{-3}\norm{V'}_{\ci{0,\alpha}}\norm{\bm Z'}_{\ci{0,\alpha}}^2\abs{h}^\alpha\abs{s-\eta}^{-1} .\label{eqn:phi-est-4}
\end{align}
\end{lemma}

\begin{proof}
Estimate~\eqref{eqn:phi-est-1} follows directly from~\eqref{eqn:arc-chord-intro}. Differentiating with respect to $s$ gives
\begin{align}
    \p_s\phi(s,\eta) &= \frac{V'(s)}{\abs{\Delta\bm Z}} - \frac{\Delta V\paren{\Delta \bm Z\cdot\bm Z'(s)}}{\abs{\Delta\bm Z}^3} \coloneqq A + B,
\end{align}
where
\begin{align}
    A  = \frac{1}{\abs{\Delta\bm Z}} \paren{V'(s) - \frac{\Delta V}{s-\eta}} ,\quad
    B = \frac{ \Delta V\Delta\bm Z}{\abs{\Delta\bm Z}^3} \cdot\paren{\frac{\Delta \bm Z}{s-\eta} - \bm Z'(s)}.
\end{align}
By Lemma~\ref{lem:diff-est},
\begin{equation}
    \abs{A} \le \abs{\bm Z}_*^{-1}[V']_{\ci{0,\alpha}} \abs{s-\eta}^{\alpha-1},\quad \abs{B} \le \abs{\bm Z}_*^{-2}\norm{V'}_{\ci{0}} [\bm Z']_{\ci{0,\alpha}} \abs{s-\eta}^{\alpha-1}.
\end{equation}
This proves \eqref{eqn:phi-est-2}. Integrating
\eqref{eqn:phi-est-2} along the segment $s+\rho h$, $0\le\rho\le1$,
gives \eqref{eqn:phi-est-3}:
\begin{equation}
\begin{aligned}
    \abs{\Delta_{s,h}\phi} &= \abs{ h\int_0^1 \p_s \phi(s+\rho h, \eta)\,d\rho }  \\
    &\le C \abs{\bm Z}_*^{-2}\paren{[V']_{\ci{0,\alpha}}\norm{\bm Z'}_{\ci{0}} + [\bm
    Z']_{\ci{0,\alpha}}\norm{V'}_{\ci{0}}} \abs{h} \abs{s-\eta}^{\alpha-1} .
\end{aligned}
\end{equation}
To estimate the finite differences of $\p_s\phi$, expand
\begin{align}
    \Delta_{s,h} A &= \frac{1}{\abs{T_{s,h}\Delta\bm Z}} \Delta_{s,h}\paren{V'(s) - \frac{\Delta V}{s-\eta}} + \paren{V'(s) - \frac{\Delta V}{s-\eta}}\Delta_{s,h}\paren{\frac{1}{\abs{\Delta\bm Z}} },\\
    \Delta_{s,h} B &= T_{s,h}\paren{\frac{ \Delta V\Delta\bm Z}{\abs{\Delta\bm Z}^3}} \cdot\Delta_{s,h}\paren{\frac{\Delta \bm Z}{s-\eta} - \bm Z'(s)} + \Delta_{s,h}\paren{\frac{ \Delta V\Delta\bm Z}{\abs{\Delta\bm Z}^3}} \cdot \paren{\frac{\Delta \bm Z}{s-\eta} - \bm Z'(s)}
\end{align}
For $\abs{h} \le \frac{1}{2}\abs{s-\eta}$, we have
\begin{align}
    &\frac{1}{\abs{T_{s,h}\Delta\bm Z}} \le \abs{\bm Z}_*^{-1}\abs{s+h-\eta}^{-1} \le C\abs{\bm Z}_*^{-1} \abs{s-\eta}^{-1},\\
    &\abs{T_{s,h}\paren{\frac{ \Delta V\Delta\bm Z}{\abs{\Delta\bm Z}^3}}} \le C\abs{\bm Z}_*^{-2}\norm{V'}_{\ci{0}} \abs{s-\eta}^{-1}.
\end{align}
The remaining factors satisfy
\begin{align}
    &\abs{\partial_s\paren{\frac{1}{\abs{\Delta \bm Z}}} }=\abs{ \frac{\Delta \bm Z\cdot\p_s\bm
    Z}{\abs{\Delta\bm Z}^3} }\le \abs{\bm Z}_*^{-2}\norm{\bm Z'}_{\ci{0}}\abs{s-\eta}^{-2}, \\
    &\p_s\paren{\frac{ \Delta V\Delta\bm Z}{\abs{\Delta\bm Z}^3}} = V'(s) \frac{\Delta\bm Z}{\abs{\Delta\bm
    Z}^3} + \Delta V \paren{\frac{\bm Z'(s)}{\abs{\Delta\bm Z}^3} -3 \Delta\bm Z \frac{\Delta\bm Z\cdot\bm
    Z'(s)}{\abs{\Delta\bm Z}^5} }, \\
    &\abs{\p_s\paren{\frac{ \Delta V\Delta\bm Z}{\abs{\Delta\bm Z}^3}} } \le
    C\paren{\frac{\abs{V'(s)}}{\abs{\Delta\bm Z}^2} + \frac{\abs{\Delta V}\abs{\bm Z'(s)}}{\abs{\Delta\bm Z}^3}
    + \frac{\abs{\bm Z'}\abs{\Delta V}}{\abs{\Delta\bm Z}^3} }\notag \\
    &\quad \le C\abs{\bm Z}_*^{-3}\norm{V'}_{\ci{0}}\norm{\bm Z'}_{\ci{0}}\abs{s-\eta}^{-2}.
\end{align}
Hence
\begin{equation}
\begin{aligned}
    &\abs{\Delta_{s,h}\paren{\frac{1}{\abs{\Delta\bm Z}}}} = \abs{ h \int_0^1 \partial_s \paren{\frac{1}{\abs{\Delta \bm Z}}}(s+\rho h,\eta) \,d\rho }\\
    &\le \abs{h}\abs{\bm Z}_*^{-2}\norm{\bm Z'}_{\ci{0}}
    \int_0^1 \abs{s+\rho h-\eta}^{-2} \,d\rho  \le C \abs{\bm Z}_*^{-2}\norm{\bm Z'}_{\ci{0}} \abs{h}\abs{s-\eta}^{-2},
\end{aligned}
\end{equation}
and
\begin{equation}
\begin{aligned}
    \abs{\Delta_{s,h}\paren{\frac{ \Delta V\Delta\bm Z}{\abs{\Delta\bm Z}^3}}} &= \abs{h\int_0^1
    \p_s\paren{\frac{ \Delta V\Delta\bm Z}{\abs{\Delta\bm Z}^3}} (s+\rho h,\eta) \,d\rho }  \\
    &\le C\abs{\bm Z}_*^{-3}\norm{V'}_{\ci{0}}\norm{\bm Z'}_{\ci{0}}\abs{h}\abs{s-\eta}^{-2}.
\end{aligned}
\end{equation}
Combining these estimates with Lemma~\ref{lem:diff-est} gives
\begin{equation}
\begin{aligned}
    \abs{\Delta_{s,h}A} &\le C\abs{\bm Z}_*^{-2}\norm{\bm Z'}_{\ci{0}}[V']_{\ci{0,\alpha}} \paren{ \abs{h}^\alpha\abs{s-\eta}^{-1} + \abs{h}\abs{s-\eta}^{\alpha-2} } \\
    &\le C\abs{\bm Z}_*^{-2}\norm{\bm Z'}_{\ci{0}}[V']_{\ci{0,\alpha}}  \abs{h}^\alpha\abs{s-\eta}^{-1},
\end{aligned}
\end{equation}
and
\begin{equation}
\begin{aligned}
    \abs{\Delta_{s,h}B} \le C \abs{\bm Z}_*^{-3}\norm{V'}_{\ci{0}}\norm{\bm Z'}_{\ci{0}}[\bm Z']_{\ci{0,\alpha}} \abs{h}^\alpha  \abs{s-\eta}^{-1}  .
\end{aligned}
\end{equation}
This is \eqref{eqn:phi-est-4} for $\p_s\phi$.

The derivative with respect to $\eta$ has the analogous form
\begin{align}
    \p_\eta\phi(s,\eta) &= -\frac{V'(\eta)}{\abs{\Delta\bm Z}} + \frac{\Delta V\paren{\Delta \bm Z\cdot\bm Z'(\eta)}}{\abs{\Delta\bm Z}^3}.
\end{align}
The estimates for $\p_\eta\phi$ follow from the same decomposition,
with $s$ and $\eta$ interchanged.
\end{proof}

If $V$ is a component of $\bm Z$, then
$\abs{\Delta V}\le\abs{\Delta\bm Z}$. This improves the arc-chord dependence
in Lemma~\ref{lem-phi-ZV}.
\begin{lemma}\label{lem-phi-Z}
Let $\bm Z = (V, W)^T\in \ci{1,\alpha}$ for $\alpha\in(0,1)$ with
$\abs{\bm Z}_*>0$. For $s,\eta\in I$ with $s\ne\eta$, define
\begin{equation}
    \phi(s,\eta) = \frac{\Delta V}{\abs{\Delta\bm Z}}.
\end{equation}
Then
\begin{align}
    &\abs{\phi}\le 1,\\
    &\abs{\p_s\phi}, \abs{\p_\eta\phi} \le C \abs{\bm Z}_*^{-1}[\bm Z']_{\ci{0,\alpha}}  \abs{s-\eta}^{\alpha-1},
\end{align}
If $s,\eta,s+h\in I$ and $\abs{s-\eta}\ge2\abs h$, then
\begin{align}
    &\abs{\Delta_{s,h}\phi} \le C \abs{\bm Z}_*^{-1}[\bm Z']_{\ci{0,\alpha}} \abs{h} \abs{s-\eta}^{\alpha-1},\\
    &\abs{\Delta_{s,h}\p_s\phi}, \abs{\Delta_{s,h}\p_\eta\phi} \le C \abs{\bm Z}_*^{-2}\norm{\bm Z'}_{\ci{0}}[\bm Z']_{\ci{0,\alpha}} \abs{h}^\alpha\abs{s-\eta}^{-1} .
\end{align}
\end{lemma}
\begin{proof}
Refine the proof of Lemma~\ref{lem-phi-ZV} using the same decomposition as
in \eqref{eqn:phi-est-2}:
\begin{equation}
    \p_s\phi
    =
    \frac{1}{\abs{\Delta\bm Z}}
    \paren{V'(s)-\frac{\Delta V}{s-\eta}}
    +
    \frac{\Delta V\Delta\bm Z}{\abs{\Delta\bm Z}^3}
    \cdot
    \paren{\frac{\Delta\bm Z}{s-\eta}-\bm Z'(s)}.
\end{equation}
Since $V$ is a component of $\bm Z$,
$\abs{\Delta V}\le \abs{\Delta\bm Z}$.  Combining this improvement with
\eqref{eqn:divided-difference-remainder} gives
\begin{equation}
    \abs{\p_s\phi}
    \le C\abs{\bm Z}_*^{-1}[\bm Z']_{\ci{0,\alpha}}
    \abs{s-\eta}^{\alpha-1}.
\end{equation}
The formula for $\p_\eta\phi$ is the same with $s$ and $\eta$
interchanged, so the same bound holds for $\p_\eta\phi$.  The estimate
$\abs{\phi}\le 1$ follows directly from
$\abs{\Delta V}\le\abs{\Delta\bm Z}$.

For the first difference estimate, integrate the $\p_s\phi$-bound
along the segment $s+\rho h$, $0\le\rho\le1$.  Because
$\abs{s+\rho h-\eta}\ge \frac12\abs{s-\eta}$ under
$\abs{s-\eta}\ge2\abs h$,
\begin{equation}
    \abs{\Delta_{s,h}\phi}
    \le |h|\int_0^1 \abs{\p_s\phi(s+\rho h,\eta)}\,d\rho
    \le C\abs{\bm Z}_*^{-1}[\bm Z']_{\ci{0,\alpha}}
    |h|\abs{s-\eta}^{\alpha-1}.
\end{equation}
For $\Delta_{s,h}\p_s\phi$, use the decompositions of
$\Delta_{s,h}A$ and $\Delta_{s,h}B$ in the proof of
Lemma~\ref{lem-phi-ZV}. The component estimate
$\abs{\Delta V}\le\abs{\Delta\bm Z}$ again lowers the power of
$\abs{\bm Z}_*^{-1}$. By \eqref{eqn:diff-est1},
\begin{equation}
    \abs{\Delta_{s,h}\p_s\phi}
    \le C\abs{\bm Z}_*^{-2}\norm{\bm Z'}_{\ci{0}}[\bm Z']_{\ci{0,\alpha}}|h|^\alpha\abs{s-\eta}^{-1}.
\end{equation}
The same argument, with the formula for $\p_\eta\phi$, gives the last
estimate.
\end{proof}

\begin{proof}[Proof of Lemma~\ref{LM-kernel-est}]
Differentiating the kernels gives
\begin{equation}
\begin{aligned}
    \p_\eta \ell(s,\eta;\bm Z) = \frac{1}{s-\eta} - \frac{\bm Z'(\eta)\cdot \Delta \bm Z}{|\Delta \bm Z|^2} = \frac{\Delta \bm Z}{\abs{\Delta\bm Z}^2}\cdot\paren{ \frac{\Delta \bm Z}{s-\eta} -\bm Z'(\eta) },
\end{aligned}
\end{equation}
\begin{equation}
\begin{aligned}
    \p_\eta m_{ij}(s,\eta;\bm Z) &= \p_\eta\paren{\frac{\Delta Z_i}{|\Delta\bm Z|}} \frac{\Delta Z_j}{|\Delta\bm Z|} + \frac{\Delta Z_i}{|\Delta\bm Z|}\p_\eta\paren{\frac{\Delta Z_j}{|\Delta\bm Z|}} .
\end{aligned}
\end{equation}
Lemmas~\ref{lem:diff-est} and~\ref{lem-phi-Z} immediately imply
\eqref{eqn:LM-1} and \eqref{eqn:LM-2}. For the finite-difference bounds,
expand
\begin{equation}
\begin{aligned}
    \Delta_{s,h}(\p_\eta \ell(s,\eta;\bm Z))
    &= T_{s,h}\paren{\frac{\Delta \bm Z}{\abs{\Delta\bm Z}^2}}
    \cdot\Delta_{s,h}\paren{
    \frac{\Delta \bm Z}{s-\eta} -\bm Z'(\eta) } \\
    &\quad
    + \Delta_{s,h}\paren{\frac{\Delta \bm Z}{\abs{\Delta\bm Z}^2}}
    \cdot\paren{ \frac{\Delta \bm Z}{s-\eta} -\bm Z'(\eta) }.
\end{aligned}
\end{equation}
and
\begin{equation}
\begin{aligned}
    \Delta_{s,h}(\p_\eta m_{ij}(s,\eta;\bm Z)) &= T_{s,h}\p_\eta\paren{\frac{\Delta Z_i}{|\Delta\bm Z|}} \Delta_{s,h}\paren{ \frac{\Delta Z_j}{|\Delta\bm Z|} } + \frac{\Delta Z_j}{|\Delta\bm Z|}  \Delta_{s,h}\p_\eta\paren{\frac{\Delta Z_i}{|\Delta\bm Z|}} \\
    &+ T_{s,h}\frac{\Delta Z_i}{|\Delta\bm Z|}\Delta_{s,h}\p_\eta\paren{\frac{\Delta Z_j}{|\Delta\bm Z|}} + \p_\eta\paren{\frac{\Delta Z_j}{|\Delta\bm Z|}}\Delta_{s,h}\frac{\Delta Z_i}{|\Delta\bm Z|} .
\end{aligned}
\end{equation}
Moreover,
\begin{equation}
    \abs{\p_s \paren{\frac{\Delta \bm Z}{\abs{\Delta\bm Z}^2}}} = \abs{\frac{\bm Z'(s)}{\abs{\Delta\bm Z}^2} - 2 \frac{\Delta\bm Z\paren{\Delta\bm Z\cdot\bm Z'(s)}}{\abs{\Delta\bm Z}^4}} \le C \abs{\bm Z}_*^{-2}\norm{\bm Z'}_{\ci{0}}\abs{s-\eta}^{-2},
\end{equation}
and hence
\begin{equation}
    \abs{\Delta_{s,h} \paren{\frac{\Delta \bm Z}{\abs{\Delta\bm Z}^2}}} =\abs{h\int_0^1 \p_s \paren{\frac{\Delta \bm Z}{\abs{\Delta\bm Z}^2}}(s+\rho h,\eta) \,d\rho}\le C \abs{\bm Z}_*^{-2}\norm{\bm Z'}_{\ci{0}}\abs{h}\abs{s-\eta}^{-2}.
\end{equation}
Lemma~\ref{lem-phi-Z}, together with
$\abs h\le\frac12\abs{s-\eta}$, now gives
\begin{equation}
\begin{aligned}
    \abs{\Delta_{s,h}(\p_\eta \ell(s,\eta;\bm Z))} &\le C \abs{\bm Z}_*^{-1}[\bm Z']_{\ci{0,\alpha}} \abs{h}^\alpha\abs{s-\eta}^{-1} + C \abs{\bm Z}_*^{-2}\norm{\bm Z'}_{\ci{0}}[\bm Z']_{\ci{0,\alpha}} \abs{h} \abs{s-\eta}^{\alpha-2} \\
    &\le C \abs{\bm Z}_*^{-2}\norm{\bm Z'}_{\ci{0}}[\bm Z']_{\ci{0,\alpha}}\abs{h}^\alpha\abs{s-\eta}^{-1}.
\end{aligned}
\end{equation}
and
\begin{equation}
\begin{aligned}
    \abs{\Delta_{s,h}(\p_\eta m_{ij}(s,\eta;\bm Z))} &\le C \abs{\bm Z}_*^{-2}[\bm Z']_{\ci{0,\alpha}}^2 \abs{h}\abs{s-\eta}^{2\alpha-2} \\
    &\quad + C\abs{\bm Z}_*^{-2}\norm{\bm Z'}_{\ci{0}}[\bm Z']_{\ci{0,\alpha}} \abs{h}^\alpha\abs{s-\eta}^{-1} \\
    &\le C \abs{\bm Z}_*^{-2}\norm{\bm Z'}_{\ci{0,\alpha}}[\bm Z']_{\ci{0,\alpha}}  \abs{h}^\alpha\abs{s-\eta}^{-1}.
\end{aligned}
\end{equation}
The derivatives with respect to $s$ satisfy
\begin{equation}\label{eqn:ell-s-derivative}
\begin{aligned}
    \p_s \ell(s,\eta;\bm Z) = \frac{\bm Z'(s)\cdot \Delta \bm Z}{|\Delta \bm Z|^2} - \frac{1}{s-\eta} = \frac{\Delta \bm Z}{\abs{\Delta\bm Z}^2}\cdot\paren{ \bm Z'(s) -\frac{\Delta \bm Z}{s-\eta}  }.
\end{aligned}
\end{equation}
\begin{equation}\label{eqn:m-s-derivative}
\begin{aligned}
    \p_s m_{ij}(s,\eta;\bm Z) &= \p_s\paren{\frac{\Delta Z_i}{|\Delta\bm Z|}} \frac{\Delta Z_j}{|\Delta\bm Z|} + \frac{\Delta Z_i}{|\Delta\bm Z|}\p_s\paren{\frac{\Delta Z_j}{|\Delta\bm Z|}} .
\end{aligned}
\end{equation}
The estimates for $\p_s\ell$ and $\p_s m_{ij}$ follow from \eqref{eqn:ell-s-derivative}--\eqref{eqn:m-s-derivative} and the same bounds \eqref{eqn:diff-est1} and \eqref{eqn:phi-est-4}.
\end{proof}

\begin{proof}[Proof of Lemma~\ref{lem:LM-Lip}]
We first estimate the directional derivatives of the kernels and then
integrate these estimates along the segment joining $\bm Y$ to $\bm X$.
Let $\bm Z\in C^{1,\alpha}(I)$ be a direction. Differentiation with respect
to the curve commutes with $\p_\eta$:
\begin{equation}
    \p_{\bm X}\paren{\p_\eta \ell(s,\eta;\bm X)}\bm Z
    =\p_\eta\paren{\p_{\bm X}\ell(s,\eta;\bm X)\bm Z}.
\end{equation}
The corresponding derivatives of $\ell$ and $m_{ij}$ are
\begin{align}
    \p_{\bm X}\ell(s,\eta;\bm X)\bm Z
    &=\frac{\Delta\bm X\cdot\Delta\bm Z}
    {\abs{\Delta\bm X}^2},\label{eqn:ell-curve-derivative}\\
    \p_{\bm X}m_{ij}(s,\eta;\bm X)\bm Z
    &=\frac{\Delta Z_i\Delta X_j+\Delta X_i\Delta Z_j}
    {\abs{\Delta\bm X}^2}
    -2\frac{\Delta X_i\Delta X_j
    \paren{\Delta\bm X\cdot\Delta\bm Z}}
    {\abs{\Delta\bm X}^4}.\label{eqn:m-curve-derivative}
\end{align}
Each term has the same quotient structure. For example,
\begin{equation}
    \p_\eta\paren{\frac{\Delta X_i\Delta Z_j}
    {\abs{\Delta\bm X}^2}}
    =\p_\eta\paren{\frac{\Delta X_i}{\abs{\Delta\bm X}}}
    \frac{\Delta Z_j}{\abs{\Delta\bm X}}
    +\p_\eta\paren{\frac{\Delta Z_j}{\abs{\Delta\bm X}}}
    \frac{\Delta X_i}{\abs{\Delta\bm X}}.
\end{equation}
Taking an $s$-increment and using the discrete multiplication rule gives
\begin{equation}
\begin{aligned}
    \Delta_{s,h}\p_\eta\paren{
    \frac{\Delta X_i\Delta Z_j}{\abs{\Delta\bm X}^2}}
    &=T_{s,h}\p_\eta\paren{\frac{\Delta X_i}{\abs{\Delta\bm X}}}
    \Delta_{s,h}\frac{\Delta Z_j}{\abs{\Delta\bm X}}
    +\Delta_{s,h}\p_\eta\paren{\frac{\Delta X_i}{\abs{\Delta\bm X}}}
    \frac{\Delta Z_j}{\abs{\Delta\bm X}}\\
    &\quad+T_{s,h}\p_\eta\paren{\frac{\Delta Z_j}{\abs{\Delta\bm X}}}
    \Delta_{s,h}\frac{\Delta X_i}{\abs{\Delta\bm X}}
    +\Delta_{s,h}\p_\eta\paren{\frac{\Delta Z_j}{\abs{\Delta\bm X}}}
    \frac{\Delta X_i}{\abs{\Delta\bm X}}.
\end{aligned}
\end{equation}
Lemma~\ref{lem-phi-ZV} gives
\begin{equation}
    \abs{\p_\eta\paren{\frac{\Delta X_i\Delta Z_j}
    {\abs{\Delta\bm X}^2}}}
    \le C\abs{\bm X}_*^{-2}\norm{\bm Z'}_{\ci{0,\alpha}}
    \norm{\bm X'}_{\ci{0,\alpha}}
    \abs{s-\eta}^{\alpha-1}.
\end{equation}
If $\abs h\le\frac12\abs{s-\eta}$, the same lemma and the preceding
decomposition give
\begin{equation}
    \abs{\Delta_{s,h}\p_\eta\paren{
    \frac{\Delta X_i\Delta Z_j}{\abs{\Delta\bm X}^2}}}
    \le C\abs{\bm X}_*^{-3}\norm{\bm Z'}_{\ci{0,\alpha}}
    \norm{\bm X'}_{\ci{0,\alpha}}^2
    \abs h^\alpha\abs{s-\eta}^{-1}.
\end{equation}
The same calculation applies to $\p_\eta\ell$ and
$\p_\eta m_{ij}$, using \eqref{eqn:ell-curve-derivative}--\eqref{eqn:m-curve-derivative} and
Lemma~\ref{lem-phi-ZV}. Thus, for $q=\ell,m_{ij}$,
\begin{align}
    \abs{\p_{\bm X}\paren{\p_\eta q(s,\eta;\bm X)}\bm Z}
    &\le C\abs{\bm X}_*^{-2}\norm{\bm Z'}_{\ci{0,\alpha}}
    \norm{\bm X'}_{\ci{0,\alpha}}
    \abs{s-\eta}^{\alpha-1},\label{eqn:kernel-directional-bound}\\
    \abs{\Delta_{s,h}\p_{\bm X}
    \paren{\p_\eta q(s,\eta;\bm X)}\bm Z}
    &\le C\abs{\bm X}_*^{-3}\norm{\bm Z'}_{\ci{0,\alpha}}
    \norm{\bm X'}_{\ci{0,\alpha}}^2
    \abs h^\alpha\abs{s-\eta}^{-1}.\label{eqn:kernel-directional-increment}
\end{align}
Integrating~\eqref{eqn:kernel-directional-bound} along the segment from $\bm Y$ to $\bm X$ gives:
\begin{equation}
\begin{aligned}
    \abs{\p_\eta\paren{\ell(s,\eta;\bm X)-\ell(s,\eta;\bm Y)}} &\le\int_0^1 \abs{\p_{\bm X}\p_\eta\ell
    \paren{s,\eta;\rho\bm X+(1-\rho)\bm Y} \paren{\bm X-\bm Y}}\,d\rho \\
    &\le C_{M,m} \norm{\bm X'-\bm Y'}_{\ci{0,\alpha}} \abs{s-\eta}^{\alpha-1}.
\end{aligned}
\end{equation}
The same estimate holds for $m_{ij}$:
\begin{equation}
    \abs{\p_\eta\paren{m_{ij}(s,\eta;\bm X)-m_{ij}(s,\eta;\bm Y)}}
    \le C_{M,m}
    \norm{\bm X'-\bm Y'}_{\ci{0,\alpha}}
    \abs{s-\eta}^{\alpha-1}.
\end{equation}
Integrating~\eqref{eqn:kernel-directional-increment} along the same segment shows that,
for $q=\ell,m_{ij}$ and $\abs h\le\frac12\abs{s-\eta}$,
\begin{equation}
    \abs{\Delta_{s,h}\p_\eta
    \paren{q(s,\eta;\bm X)-q(s,\eta;\bm Y)}}
    \le C_{M,m}
    \norm{\bm X'-\bm Y'}_{\ci{0,\alpha}}
    \abs h^\alpha\abs{s-\eta}^{-1}.
\end{equation}
The corresponding estimates for $\p_s\ell$ and $\p_sm_{ij}$ follow in the same
way. Equation~\eqref{eqn:kernel-mapping-bound} gives the endpoint operator bounds, and the
interpolation argument from Lemma~\ref{LM-op-est} proves
\eqref{eqn:Lop-Lip} and~\eqref{eqn:Mop-Lip}.
\end{proof}

\begin{proof}[Proof of Lemma~\ref{K-op}]
The bound~\eqref{eqn:kernel-pointwise-hypothesis} is integrable because $\alpha>0$. More precisely,
\begin{equation}
    \sup_{s\in I}\int_I\abs{k(s,\eta)}\,d\eta
    +\sup_{\eta\in I}\int_I\abs{k(s,\eta)}\,ds
    \le C a_1.
\end{equation}
The Schur test gives
\begin{equation}
    \norm{K f}_{L^2(I)}
    \le C a_1 \norm{f}_{L^2(\R)}.
\end{equation}
This proves the case $\varepsilon=0$. Suppose now that
$0<\varepsilon<\alpha$. Writing the fractional seminorm in terms of
increments gives
\begin{equation}
\begin{aligned}
    [Kf]_{\hi{\varepsilon}}^2
    &=\int_{-2}^2\frac{1}{\abs h^{1+2\varepsilon}}
    \int_{I\cap(I-h)}
    \abs{Kf(s+h)-Kf(s)}^2\,ds\,dh.
\end{aligned}
\end{equation}
For fixed $h$, another application of the Schur test gives
\begin{equation}
    \int_{I\cap(I-h)}
    \abs{Kf(s+h)-Kf(s)}^2\,ds
    \le M_1(h)M_2(h)\norm{f}_{L^2(\R)}^2,
\end{equation}
where
\begin{align}
    M_1(h)&\coloneqq
    \sup_{s\in I\cap(I-h)}\int_I
    \abs{k(s+h,\eta)-k(s,\eta)}\,d\eta,\\
    M_2(h)&\coloneqq
    \sup_{\eta\in I}\int_{I\cap(I-h)}
    \abs{k(s+h,\eta)-k(s,\eta)}\,ds.
\end{align}

For $\abs h\ge1$, \eqref{eqn:kernel-pointwise-hypothesis} directly gives the estimate below.
It therefore remains to consider $0<\abs h<1$. Split each integral into
$\abs{s-\eta}\le2\abs h$ and $\abs{s-\eta}>2\abs h$. On the near region,
\eqref{eqn:kernel-pointwise-hypothesis} gives $Ca_1\abs h^\alpha$. On the far region,
\eqref{eqn:kernel-increment-hypothesis} gives
\begin{equation}
    Ca_2\abs h^\alpha
    \int_{2\abs h}^{2}\frac{dr}{r}
    \le Ca_2\abs h^\alpha\paren{1+\abs{\log\abs h}}.
\end{equation}
The same estimates apply to both integration variables. Hence
\begin{equation}
    M_1(h)M_2(h)
    \le C\paren{a_1+a_2}^2\abs h^{2\alpha}
    \paren{1+\abs{\log\abs h}^2}.
\end{equation}
Substituting into the seminorm formula gives
\begin{equation}
\begin{aligned}
    [Kf]_{\hi{\varepsilon}}
    &\le C\paren{a_1+a_2}
    \left(\int_{-2}^2
    \abs h^{-1+2(\alpha-\varepsilon)}
    \paren{1+\abs{\log\abs h}^2}\,dh\right)^{\frac12}
    \norm{f}_{L^2(\R)} \\
    &\le C\paren{a_1+a_2}\norm{f}_{L^2(\R)},
\end{aligned}
\end{equation}
because $\varepsilon<\alpha$.
\end{proof}

\begin{proof}[Uniform trace bound in Lemma~\ref{lem:stokes-layer-open-arc}]
Fix $0<\alpha<\beta<\delta$ as in the main text.
To establish uniformity of the constant in~\eqref{eqn:stokes-layer-bound}, observe that the family of normalized arcs satisfying $\norm{\kappa}_{\hr{-\frac12+\delta}}\le M$ and $\abs{\kappa}_*\ge m$ forms a precompact subset in $C^{1,\alpha}(\overline I;\R^2)$ whose limiting curves remain non-self-intersecting regular embedded arcs. Near each limit $\bm X_0$, a fixed closed extension $\wt{\bm X}_0$ produces nearby extensions $\wt{\bm X}_0+E(\bm X-\bm X_0)$, where $E$ is a bounded $C^{1,\alpha}$ extension operator. These extended curves remain embedded and enjoy uniform chart, pullback, and trace bounds in a sufficiently small neighborhood.
In the whole-space variational formulation~\cite[Proposition~7.1]{SayasSelgas2014}, the geometry enters solely through the trace operator and the normalization functionals fixing the rigid velocity modes. The trace operator is governed by the uniform $C^{1,\alpha}$ bounds, while the normalization functionals are bounded uniformly by fixing a single enclosing circle in~\cite[(7.2)--(7.4)]{SayasSelgas2014}. Furthermore, the whole-space Korn and divergence constants are completely independent of the curve geometry. A finite subcover of the compact closure yields~\eqref{eqn:stokes-layer-bound}.
\end{proof}

\subsection{Higher-order boundary integral estimates}
\label{app:higher-order-estimates}

Higher-order Sobolev estimates for the non-local nonlinearity $N(\kappa)$ rely on divided difference bounds for kernel compositions.

\begin{lemma}\label{lem:K-higher-order}
Let $d\ge1$ and $k\ge2$ be integers, let
$\alpha\in(0,1)$, and let $0\le\varepsilon<\alpha$. Let
$f\in C^{k,\alpha}(I;\R^d)$, let
$U\subset\R^d$ be open, and let $\phi\in C^\infty(U)$. Define the divided
differences
\begin{equation}\label{eq:Vj-def}
    V(s,\eta) \coloneqq \frac{f(s)-f(\eta)}{s-\eta},
    \quad
    V_j(s,\eta) \coloneqq \frac{f^{(j)}(s)-f^{(j)}(\eta)}{s-\eta},
    \quad 0\le j\le k,
\end{equation}
so that $V_0 = V$. Extend $V$ continuously to the diagonal by
$V(s,s)=f'(s)$, and assume that
$\overline{V(I\times I)}\Subset U$. The integral operator $K$ is defined by
\begin{equation}\label{eq:K-def}
    (K g)(s) \coloneqq \int_I \p_\eta\phi\bigl(V(s,\eta)\bigr)\,g(\eta)\,d\eta,
    \quad s\in I.
\end{equation}
Then $K:\wt L^2(I)\to\hi{k-1+\varepsilon}$ is bounded, with
\begin{equation}\label{eq:K-bound}
    \norm{K g}_{\hi{k-1+\varepsilon}}
    \le C\,\norm{g}_{L^2(\R)},
\end{equation}
where $C$ depends on $d,k,\varepsilon,\alpha$,
$\norm{f}_{C^{k,\alpha}(I)}$, and the derivatives of $\phi$
through order $k+1$ on a fixed neighborhood of
$\overline{V(I\times I)}$.
\end{lemma}

\begin{proof}
We first derive pointwise and H\"older bounds for $V_j$.
For $0\le \ell\le k-1$ and $n=k-\ell\ge 1$, the Leibniz rule applied to
$V_\ell(s,\eta)=[f^{(\ell)}(s)-f^{(\ell)}(\eta)]/(s-\eta)$ gives
\begin{equation}\label{eq:Vl-leibniz}
    \p_s^n V_\ell(s,\eta)
    = \bigl(f^{(\ell)}(s)-f^{(\ell)}(\eta)\bigr)
      \frac{(-1)^n n!}{(s-\eta)^{n+1}}
      +\sum_{j=1}^{n}\binom{n}{j}f^{(\ell+j)}(s)
      \frac{(-1)^{n-j}(n-j)!}{(s-\eta)^{n-j+1}}.
\end{equation}
The integral Taylor remainder for $f^{(\ell)}\in\ci{n,\alpha}$ at $s$ is
\begin{equation}\label{eq:remainder}
    f^{(\ell)}(\eta)-\sum_{j=0}^{n}\frac{f^{(\ell+j)}(s)}{j!}(\eta-s)^j
    = \frac{(\eta-s)^n}{(n-1)!}
      \int_0^1(1-r)^{n-1}\bigl(f^{(k)}(s+r(\eta-s))-f^{(k)}(s)\bigr)\,dr,
\end{equation}
Combining \eqref{eq:remainder} and \eqref{eq:Vl-leibniz} gives
\begin{equation}\label{eq:Vl-formula}
    \p_s^n V_\ell(s,\eta)
    = \frac{(-1)^{n+1}\,n!}{(s-\eta)^{n+1}}
      \Bigl[\,f^{(\ell)}(\eta)
      -\sum_{j=0}^{n}\frac{f^{(\ell+j)}(s)}{j!}(\eta-s)^j\Bigr].
\end{equation}
Since $f^{(k)}\in\ci{0,\alpha}$, the integrand in \eqref{eq:remainder} is bounded by
$C\abs{r(\eta-s)}^\alpha\le C\abs{s-\eta}^\alpha$, giving
\begin{equation}\label{eq:remainder-bound}
    \Bigl\lvert f^{(\ell)}(\eta)-\sum_{j=0}^{n}\frac{f^{(\ell+j)}(s)}{j!}(\eta-s)^j
    \Bigr\rvert \le C\,\abs{s-\eta}^{n+\alpha}.
\end{equation}
Substitution into \eqref{eq:Vl-formula} gives
\begin{equation}\label{eq:Vl-ptwise}
    \abs{\p_s^{k-\ell}V_\ell(s,\eta)} \le C\,\abs{s-\eta}^{\alpha-1},
    \quad 0\le \ell\le k.
\end{equation}
For $\ell=k$, this follows directly from the
$C^{0,\alpha}$-seminorm of $f^{(k)}$.
For the H\"older bound, \eqref{eq:Vl-formula} gives
\begin{equation}\label{eq:Vl-intrep}
    \p_s^n V_\ell(s,\eta)
    = -\frac{k-\ell}{s-\eta}
      \int_0^1(1-r)^{n-1}
      \bigl(f^{(k)}(s+r(\eta-s))-f^{(k)}(s)\bigr)\,dr.
\end{equation}
For $s,\eta,s+h\in I$ with
$\abs{h}\le\tfrac{1}{2}\abs{s-\eta}$, split
$\p_s^n V_\ell(s+h,\eta)-\p_s^n V_\ell(s,\eta)$ as
\begin{align}
    &\p_s^n V_\ell(s+h,\eta)-\p_s^n V_\ell(s,\eta)\notag\\
    &= \underbrace{\Bigl(\frac{1}{s+h-\eta}-\frac{1}{s-\eta}\Bigr)
      c_n\int_0^1(1-r)^{n-1}
      \bigl(f^{(k)}(s+r(\eta-s))-f^{(k)}(s)\bigr)\,dr}_{=: I_1}
    \notag\\
    &\quad
    +\frac{c_n}{s+h-\eta}\int_0^1(1-r)^{n-1}
    \Bigl[\bigl(f^{(k)}(s+h+r(\eta-s-h))-f^{(k)}(s+h)\bigr)\notag\\
    &\hspace{10em}
    -\bigl(f^{(k)}(s+r(\eta-s))-f^{(k)}(s)\bigr)\Bigr]\,dr,
    \label{eq:Vl-split}
\end{align}
where $c_n=-(k-\ell)$. Since $\abs{s+h-\eta}\ge\tfrac{1}{2}\abs{s-\eta}$,
\begin{equation}
    \Bigl\lvert\frac{1}{s+h-\eta}-\frac{1}{s-\eta}\Bigr\rvert
    \le \frac{C\abs{h}}{\abs{s-\eta}^2},
\end{equation}
and the integral is $O(\abs{s-\eta}^\alpha)$ by $f^{(k)}\in\ci{0,\alpha}$, so
$\abs{I_1}\le C\abs{h}\abs{s-\eta}^{\alpha-2}
\le C\abs{h}^\alpha\abs{s-\eta}^{-1}$
(using $\abs{h}\le\tfrac{1}{2}\abs{s-\eta}$).
For the second integral, $\abs{s+h-\eta}^{-1}\le C\abs{s-\eta}^{-1}$.
For each $r\in[0,1]$, pairing the evaluations at $s+h+r(\eta-s-h)$ and
$s+r(\eta-s)$ gives an argument difference $h(1-r)$; the evaluations at $s+h$ and $s$
differ by $h$. H\"older continuity bounds both increments by
$C\abs{h}^\alpha$, giving a contribution
$C\abs{h}^\alpha\abs{s-\eta}^{-1}$. Thus
\begin{equation}\label{eq:Vl-holder}
    \abs{\p_s^{k-\ell}V_\ell(s+h,\eta)-\p_s^{k-\ell}V_\ell(s,\eta)}
    \le C\,\abs{h}^\alpha\,\abs{s-\eta}^{-1},
    \quad \abs{h}\le\tfrac{1}{2}\abs{s-\eta}.
\end{equation}
For $\ell=k$, writing the difference of the two quotients and using
$[f^{(k)}]_{C^{0,\alpha}}$ gives
\begin{equation}
    \abs{V_k(s+h,\eta)-V_k(s,\eta)}
    \le C\left(
        \frac{\abs h^\alpha}{\abs{s-\eta}}
        +\abs h\,\abs{s-\eta}^{\alpha-2}
    \right)
    \le C\abs h^\alpha\abs{s-\eta}^{-1},
\end{equation}
so \eqref{eq:Vl-holder} also holds in this case.

It remains to prove $K:\wt L^2(I)\to\hi{k-1+\varepsilon}$. It suffices
to show $\p_s^{k-1}(K g)\in\hi{\varepsilon}$.
Taking $k-1$ derivatives in $s$ in the sense of distributions, which is
justified by the integrable kernel bounds below, gives
\begin{equation}\label{eq:diff-under}
    \p_s^{k-1}(K g)(s)
    = \int_I k_\phi(s,\eta)\,g(\eta)\,d\eta,
    \quad
    k_\phi(s,\eta)\coloneqq \p_s^{k-1}\p_\eta\phi\bigl(V(s,\eta)\bigr).
\end{equation}
We apply Lemma~\ref{K-op} to $k_\phi$.

By the multivariable Faà di Bruno formula,
$k_\phi = \p_s^{k-1}\p_\eta\phi(V)$ is a finite sum of
terms of the form
\begin{equation}\label{eq:fdb}
    D^m\phi(V)\bigl[Q_1V,\ldots,Q_mV\bigr],
\end{equation}
where $Q_i=\p_s^{a_i}\p_\eta^{b_i}$ has positive order and
$\sum_{i=1}^m\operatorname{ord}(Q_i)=k$.
Set $D_+=\p_s+\p_\eta$. Since
$\p_\eta=D_+-\p_s$ and $D_+V_j=V_{j+1}$, every factor
$Q_iV$ reduces to
a linear combination of $\p_s^a V_b$ with $a+b=\operatorname{ord}(Q_i)$.
If $\operatorname{ord}(Q_i)<k$, then $\abs{Q_iV}$ is
bounded because the diagonal extension of $V$ belongs to
$C^{k-1}(I\times I)$.
Since the orders sum to $k$, at most one factor has order $k$; that
factor satisfies $\abs{Q_iV}\le C\abs{s-\eta}^{\alpha-1}$ by
\eqref{eq:Vl-ptwise}. All remaining factors and $D^m\phi(V)$ are bounded,
so
\begin{equation}\label{eq:L-ptwise}
    \abs{k_\phi(s,\eta)} \le C\,\abs{s-\eta}^{\alpha-1}.
\end{equation}

For the H\"older bound, let $s,\eta,s+h\in I$ satisfy
$\abs{h}\le\tfrac{1}{2}\abs{s-\eta}$. We claim that
\begin{equation}\label{eq:L-holder}
    \abs{k_\phi(s+h,\eta)-k_\phi(s,\eta)} \le C\,\abs{h}^\alpha\,\abs{s-\eta}^{-1}.
\end{equation}
Let $Q$ be a mixed derivative of total order $r<k$. The
fundamental theorem of calculus and \eqref{eq:Vl-ptwise} give
\begin{equation}\label{eq:lower-factor-holder}
    \abs{QV(s+h,\eta)-QV(s,\eta)}
    \le
    \begin{cases}
        C\abs h, & r\le k-2,\\
        C\abs h\,\abs{s-\eta}^{\alpha-1}, & r=k-1,
    \end{cases}
    \le C\abs h^\alpha\abs{s-\eta}^{-1}.
\end{equation}
Here we used
$\abs{s+\rho h-\eta}\ge\frac12\abs{s-\eta}$ for
$0\le\rho\le1$ and the boundedness of $I$.
A telescoping expansion of \eqref{eq:fdb} and
\eqref{eq:lower-factor-holder} gives \eqref{eq:L-holder} when all factors
have order less than $k$. If one factor has order $k$, it is the only
factor. Since $k\ge2$, $V$ is uniformly Lipschitz in $s$, and
\begin{equation}
    \abs{\Delta_{s,h}\!\left(D\phi(V)[QV]\right)}
    \le C\abs h\,\abs{s-\eta}^{\alpha-1}
    +C\abs h^\alpha\abs{s-\eta}^{-1}
    \le C\abs h^\alpha\abs{s-\eta}^{-1}.
\end{equation}
Hence \eqref{eq:L-holder} also holds for the top-order term.

\noindent
Lemma~\ref{K-op} and \eqref{eq:L-ptwise}--\eqref{eq:L-holder} give
\begin{equation}\label{eq:step1-bound}
    \norm{\p_s^{k-1}(K g)}_{\hi{\varepsilon}} \le C\,\norm{g}_{L^2(\R)}.
\end{equation}
Since $k\ge2$, the kernel $\p_\eta\phi(V)$ is bounded, and the
Schur test gives
$\norm{K g}_{L^2(I)}\le C\norm{g}_{L^2(\R)}$.
The one-dimensional norm equivalence
\begin{equation}
    \norm{u}_{H^{k-1+\varepsilon}(I)}
    \simeq \norm{u}_{L^2(I)}
    +\norm{\p_s^{k-1}u}_{H^\varepsilon(I)}
\end{equation}
and \eqref{eq:step1-bound} give \eqref{eq:K-bound}.
\end{proof}

\subsection{Difference estimates for the nonlinear terms}
\label{app:nonlinear-estimates}

\begin{proof}[Proof of Lemma~\ref{R-Lip}]
Set $\alpha=\frac{1}{2}\paren{\varepsilon+\delta}$, so that
$\varepsilon<\alpha<\delta$. Every interpolated curve
$\rho\xka{1}+(1-\rho)\xka{2}$ has speed at most one, so its arc-chord
constant is at most one. The assumed lower bound gives
$0<m\le 1$.
First estimate the commutator part of $R_\kappa$ in~\eqref{eqn:Rop-def}. Lemmas~\ref{lem:Xk-holder},
\ref{tau-n-est}, \ref{dsH_f_comm}, and~\ref{point_mult_sobo} give
\begin{equation}
\begin{aligned}
    &\norm{ \tka{1}\cdot[\p_s\mc H,\tka{1}]\sigma - \tka{2}\cdot[\p_s\mc H,\tka{2}]\sigma }_{\hi{-\frac{1}{2}+\varepsilon}} \\
    & = \norm{ \tka{1}\cdot[\p_s\mc H,\paren{\tka{1}-\tka{2}}]\sigma + \paren{\tka{1}-\tka{2}}\cdot[\p_s\mc H,\tka{2}]\sigma }_{\hi{-\frac{1}{2}+\varepsilon}}\\
    &\le C \norm{\tka{1}}_{\hi{\frac{1}{2}+\delta}}\paren{\norm{\p_s\tka{1}-\p_s\tka{2}}_{\hi{-\frac{1}{2}+\delta}}+ [\tka{1}-\tka{2}]_{\ci{0,\alpha}}}\norm{\sigma}_{\hr{\frac{1}{2}}} \\
    & + C\norm{\tka{1}-\tka{2}}_{\hi{\frac{1}{2}+\delta}}\paren{\norm{\p_s\tka{2}}_{\hi{-\frac{1}{2}+\delta}}+ [\tka{2}]_{\ci{0,\alpha}}} \norm{\sigma}_{\hr{\frac{1}{2}}} \\
    &\le C_{M,m} \norm{\kappa_1-\kappa_2}_{\hr{-\frac{1}{2}+\delta}}\norm{\sigma}_{\hr{\frac{1}{2}}}.
\end{aligned}
\end{equation}
For the Stokes-remainder part, write
\begin{equation}
\begin{aligned}
    &\tka{1}\cdot\p_s\int_I \p_\eta r(s,\eta;\kappa_1)\sigma(\eta)\tka{1}(\eta)\,d\eta  \\
    &\quad - \tka{2}\cdot\p_s\int_I \p_\eta r(s,\eta;\kappa_2)\sigma(\eta)\tka{2}(\eta)\,d\eta = I_1 + I_2 +
    I_3,
\end{aligned}
\end{equation}
where
\begin{align}
I_1 & = \paren{\tka{1}-\tka{2}}\cdot\p_s\int_I \p_\eta r(s,\eta;\kappa_1)\sigma(\eta)\tka{1}(\eta)\,d\eta, \\
I_2 &= \tka{2}\cdot\p_s\int_I \p_\eta \paren{r(s,\eta;\kappa_1)-r(s,\eta;\kappa_2)}\sigma(\eta)\tka{1}(\eta)\,d\eta, \\
I_3 &= \tka{2}\cdot\p_s\int_I \p_\eta r(s,\eta;\kappa_2)\sigma(\eta)\paren{\tka{1}(\eta)-\tka{2}(\eta)} \,d\eta.
\end{align}
Equations~\eqref{eqn:Sop-bound}, \eqref{eqn:product-supported}, and~\eqref{eqn:product-restriction} give
\begin{equation}
\begin{aligned}
    \norm{I_1}_{\hi{-\frac{1}{2}+\varepsilon}} &\le  \abs{\xka{1}}_*^{-2}\norm{\tka{1}}_{\ci{0,\alpha}}[\tka{1}]_{\ci{0,\alpha}} \norm{\tka{1}-\tka{2}}_{\hi{\frac{1}{2}+\delta}}\norm{\tka{1}}_{\hi{\frac{1}{2}+\delta}} \norm{\sigma}_{\hr{\frac{1}{2}}}\\
    &\le C_{M,m}M \norm{\kappa_1-\kappa_2}_{\hr{-\frac{1}{2}+\delta}}\norm{\sigma}_{\hr{\frac{1}{2}}}.
\end{aligned}
\end{equation}
Similarly,
\begin{equation}
\begin{aligned}
    \norm{I_3}_{\hi{-\frac{1}{2}+\varepsilon}} & \le  \abs{\xka{2}}_*^{-2}\norm{\tka{2}}_{\ci{0,\alpha}}[\tka{2}]_{\ci{0,\alpha}} \norm{\tka{1}-\tka{2}}_{\hi{\frac{1}{2}+\delta}}\norm{\tka{2}}_{\hi{\frac{1}{2}+\delta}} \norm{\sigma}_{\hr{\frac{1}{2}}}\\
    & \le C_{M,m}M \norm{\kappa_1-\kappa_2}_{\hr{-\frac{1}{2}+\delta}}\norm{\sigma}_{\hr{\frac{1}{2}}}.
\end{aligned}
\end{equation}
Lemma~\ref{lem:Sop-Lip} gives
\begin{equation}
\begin{aligned}
    \norm{I_2}_{\hi{-\frac{1}{2}+\varepsilon}} & \le C_{M,m} \norm{\tka{2}}_{\hi{\frac{1}{2}+\delta}} \norm{\kappa_1-\kappa_2}_{\hr{-\frac{1}{2}+\delta}} \norm{\tka{1}}_{\hi{\frac{1}{2}+\delta}} \norm{\sigma}_{\hr{\frac{1}{2}}} \\
    & \le C_{M,m} \norm{\kappa_1-\kappa_2}_{\hr{-\frac{1}{2}+\delta}} \norm{\sigma}_{\hr{\frac{1}{2}}}.
\end{aligned}
\end{equation}
Combining the commutator estimate with the bounds for $I_1,I_2,I_3$
proves the result.
\end{proof}

\begin{proof}[Proof of Lemma~\ref{F-Lip-kappa}]
Set $\alpha=\frac{1}{2}\paren{\varepsilon+\delta}$, so that
$\varepsilon<\alpha<\delta$. Each interpolated curve has speed at most one.
Thus the assumed arc-chord lower bound gives $0<m\le 1$.
Split the commutator difference as
\begin{equation}
\begin{aligned}
    [\p_s\mc H,  \tka{1}]\cdot (\p_s \kappa_1 \nka{1}) - [\p_s\mc H, \tka{2}]\cdot (\p_s \kappa_2 \nka{2}) = I_1 + I_2 + I_3,
\end{aligned}
\end{equation}
where
\begin{align}
    I_1 &= [\p_s\mc H,  \tka{1}-\tka{2}]\cdot (\p_s \kappa_1 \nka{1}) , \\
    I_2 &= [\p_s\mc H,  \tka{2}]\cdot ( (\p_s \kappa_1 -\p_s\kappa_2)\nka{1}) , \\
    I_3 &= [\p_s\mc H,  \tka{2}]\cdot (\p_s \kappa_2( \nka{1} - \nka{2})) .
\end{align}
Since $\nka{i}=Q_{\frac{\pi}{2}}^{-1}\xka{i}'$ for $i=1,2$, we have
\begin{equation}
    \norm{\nka{i}}_{\hi{\frac{1}{2}+\delta}} = \norm{\xka{i}'}_{\hi{\frac{1}{2}+\delta}},
    \quad
    \norm{\nka{1}-\nka{2}}_{\hi{\frac{1}{2}+\delta}}
    = \norm{\xka{1}'-\xka{2}'}_{\hi{\frac{1}{2}+\delta}}.
\end{equation}
Lemmas~\ref{lem:Xk-holder}, \ref{tau-n-est}, \ref{dsH_f_comm},
and~\ref{point_mult_sobo} give
\begin{equation}
\begin{aligned}
    \norm{I_1}_{\hi{-\frac{1}{2}+\varepsilon}} &\le C \paren{ [\tka{1}-\tka{2}]_{\ci{0,\alpha}} + \norm{\p_s(\tka{1}-\tka{2})}_{\hi{-\frac{1}{2}+\delta}}  } \norm{\p_s \kappa_1 \nka{1}}_{\hr{\frac{1}{2}}}\\
    &\le C_{M,m} \norm{\kappa_1-\kappa_2}_{\hr{-\frac{1}{2}+\delta}}\norm{\kappa_1}_{\hr{\frac{3}{2}}}.
\end{aligned}
\end{equation}
\begin{equation}
\begin{aligned}
    \norm{I_2}_{\hi{-\frac{1}{2}+\varepsilon}} &\le C \paren{ [\tka{2}]_{\ci{0,\alpha}} + \norm{\p_s\tka{2}}_{\hi{-\frac{1}{2}+\delta}}  } \norm{\p_s (\kappa_1-\kappa_2) \nka{1}}_{\hr{\frac{1}{2}}}\\
    &\le C_{M,m}M \norm{\kappa_1-\kappa_2}_{\hr{\frac{3}{2}}}.
\end{aligned}
\end{equation}
\begin{equation}
\begin{aligned}
    \norm{I_3}_{\hi{-\frac{1}{2}+\varepsilon}} &\le C \paren{ [\tka{2}]_{\ci{0,\alpha}} + \norm{\p_s\tka{2}}_{\hi{-\frac{1}{2}+\delta}}  } \norm{\p_s \kappa_2 (\nka{1}-\nka{2}) }_{\hr{\frac{1}{2}}}\\
    &\le C_{M,m}M \norm{\kappa_1-\kappa_2}_{\hr{-\frac{1}{2}+\delta}}\norm{\kappa_2}_{\hr{\frac{3}{2}}}.
\end{aligned}
\end{equation}
\begin{equation}
\begin{aligned}
    \norm{\sum_{i=1}^3 I_i}_{\hi{-\frac{1}{2}+\varepsilon}}
    &\le C_{M,m}
    \norm{\kappa_1-\kappa_2}_{\hr{-\frac{1}{2}+\delta}}
    \paren{
        \norm{\kappa_1}_{\hr{\frac{3}{2}}}
        +\norm{\kappa_2}_{\hr{\frac{3}{2}}}
    } \\
    &\quad+ C_{M,m}M
    \norm{\kappa_1-\kappa_2}_{\hr{\frac{3}{2}}}.
\end{aligned}
\end{equation}
For the Stokes-remainder difference, write
\begin{equation}
\begin{aligned}
    &\tka{1}\cdot\p_s \int_I\p_\eta r(s,\eta;\kappa_1)\p_\eta\kappa_1\nka{1}(\eta) \,d\eta  \\
    &\quad - \tka{2}\cdot\p_s \int_I\p_\eta r(s,\eta;\kappa_2)\p_\eta\kappa_2\nka{2}(\eta) \,d\eta = I_1' + I_2'
    + I_3' + I_4'
\end{aligned}
\end{equation}
where
\begin{align}
    I_1' &= (\tka{1}-\tka{2})\cdot\p_s \int_I\p_\eta r(s,\eta;\kappa_1)\p_\eta\kappa_1\nka{1}(\eta) \,d\eta ,\\
    I_2' &= \tka{2}\cdot\p_s \int_I(\p_\eta r(s,\eta;\kappa_1)-\p_\eta r(s,\eta;\kappa_2))\p_\eta\kappa_1\nka{1}(\eta) \,d\eta ,\\
    I_3' &= \tka{2}\cdot\p_s \int_I\p_\eta r(s,\eta;\kappa_2)(\p_\eta\kappa_1-\p_\eta\kappa
    _2)\nka{1}(\eta) \,d\eta,\\
    I_4' &= \tka{2}\cdot\p_s \int_I\p_\eta r(s,\eta;\kappa_2)\p_\eta\kappa_2(\nka{1}-\nka{2})(\eta) \,d\eta.
\end{align}
Lemmas~\ref{lem:Sop-est} and~\ref{lem:Sop-Lip} give
\begin{equation}
\begin{aligned}
    \norm{I_1'}_{\hi{-\frac{1}{2}+\varepsilon}} &\le \norm{\tka{1}-\tka{2}}_{\hi{\frac{1}{2}+\delta}} \norm{\int_I\p_\eta r(s,\eta;\kappa_1)\p_\eta\kappa_1\nka{1}(\eta) \,d\eta}_{\hi{\frac{1}{2}+\varepsilon}} \\
    & \le C_{M,m}M \norm{\kappa_1-\kappa_2}_{\hr{-\frac{1}{2}+\delta}} \norm{\p_s \kappa_1}_{\hr{\frac{1}{2}}} \norm{\nka{1}}_{\hi{\frac{1}{2}+\delta}}\\
    & \le C_{M,m}M \norm{\kappa_1-\kappa_2}_{\hr{-\frac{1}{2}+\delta}}\norm{\kappa_1}_{\hr{\frac{3}{2}}},
\end{aligned}
\end{equation}
\begin{equation}
\begin{aligned}
    \norm{I_2'}_{\hi{-\frac{1}{2}+\varepsilon}} &\le C_{M,m}\norm{\tka{2}}_{\hi{\frac{1}{2}+\delta}} \norm{\kappa_1-\kappa_2}_{\hr{-\frac{1}{2}+\delta}}\norm{\p_s\kappa_1}_{\hr{\frac{1}{2}}} \norm{\nka{1}}_{\hi{\frac{1}{2}+\delta}} \\
    &\le C_{M,m} \norm{\kappa_1-\kappa_2}_{\hr{-\frac{1}{2}+\delta}}\norm{\kappa_1}_{\hr{\frac{3}{2}}},
\end{aligned}
\end{equation}
\begin{equation}
\begin{aligned}
    \norm{I_3'}_{\hi{-\frac{1}{2}+\varepsilon}} & \le C_{M,m}M \norm{\tka{2}}_{\hi{\frac{1}{2}+\delta}}  \norm{\p_s(\kappa_1-\kappa_2)}_{\hr{\frac{1}{2}}} \norm{\nka{1}}_{\hi{\frac{1}{2}+\delta}}\\
    &\le C_{M,m}M \norm{\kappa_1-\kappa_2}_{\hr{\frac{3}{2}}},
\end{aligned}
\end{equation}
\begin{equation}
\begin{aligned}
    \norm{I_4'}_{\hi{-\frac{1}{2}+\varepsilon}}
    &\le C_{M,m}M \norm{\tka{2}}_{\hi{\frac{1}{2}+\delta}}  \norm{\p_s\kappa_2 }_{\hr{\frac{1}{2}}} \norm{\nka{1}-\nka{2}}_{\hi{\frac{1}{2}+\delta}}\\
    &\le C_{M,m}M \norm{\kappa_1-\kappa_2}_{\hr{-\frac{1}{2}+\delta}}\norm{\kappa_2}_{\hr{\frac{3}{2}}}.
\end{aligned}
\end{equation}
\begin{equation}
\begin{aligned}
    \norm{\sum_{i=1}^4 I_i'}_{\hi{-\frac{1}{2}+\varepsilon}}
    &\le C_{M,m}
    \norm{\kappa_1-\kappa_2}_{\hr{-\frac{1}{2}+\delta}}
    \paren{
        \norm{\kappa_1}_{\hr{\frac{3}{2}}}
        +\norm{\kappa_2}_{\hr{\frac{3}{2}}}
    } \\
    &\quad+ C_{M,m}M
    \norm{\kappa_1-\kappa_2}_{\hr{\frac{3}{2}}}.
\end{aligned}
\end{equation}
The two decompositions give the estimate.
\end{proof}

\begin{proof}[Proof of~\eqref{eqn:N-pointwise-difference}]
Set $\alpha=\frac{1}{2}\paren{\varepsilon+\delta}$, so that
$\varepsilon<\alpha<\delta$. Each interpolated curve has speed at most one,
so the assumed arc-chord lower bound gives $0<m\le1$.
By~\eqref{eqn:inverse-uniform-bound}, the inverse constants $C_{\kappa_i}$
can be absorbed into $C_{M,m}$.
Whenever a frame field is multiplied by a function supported in
$\overline I$, we use a fixed bounded Sobolev extension of the frame field
to $\mathbb R$. The resulting function remains supported in $\overline I$.
Using~\eqref{eqn:N-def}, split
\begin{equation}
    N(\kappa_2) - N(\kappa_1) = \frac{1}{4} \sum_{i=1}^3 I_i +\frac{1}{4\pi} \sum_{i=1}^4 I_i'.
\end{equation}
The commutator terms are
\begin{align}
    I_1 &=[\p_s\mc H, \nka{1}-\nka{2}]\cdot\paren{\p_s\kappa_1\nka{1} +\sigma_1\tka{1}},\\
    I_2 &= [\p_s\mc H, \nka{2}]\cdot\paren{(\p_s\kappa_1-\p_s\kappa_2)\nka{1} + (\sigma_1-\sigma_2)\tka{1}},\\
    I_3 &= [\p_s\mc H, \nka{2}]\cdot\paren{\p_s\kappa_2(\nka{1}-\nka{2}) +\sigma_2(\tka{1}-\tka{2})}.
\end{align}
The Stokes-remainder terms are
\begin{align}
    I_1' &= (\nka{1}-\nka{2})\cdot\p_s\int_I \p_\eta r(s,\eta;\kappa_1) \paren{ \p_\eta\kappa_1\nka{1}+\sigma_1\tka{1} } \,d\eta, \\
    I_2' &= \nka{2}\cdot\p_s\int_I \p_\eta (r(s,\eta;\kappa_1)-r(s,\eta;\kappa_2)) \paren{ \p_\eta\kappa_1\nka{1}+\sigma_1\tka{1} } \,d\eta,\\
    I_3' &= \nka{2}\cdot\p_s\int_I \p_\eta r(s,\eta;\kappa_2) \paren{ \p_\eta(\kappa_1-\kappa_2)\nka{1}+(\sigma_1-\sigma_2)\tka{1} } \,d\eta,\\
    I_4' &= \nka{2}\cdot\p_s\int_I \p_\eta r(s,\eta;\kappa_2) \paren{ \p_\eta\kappa_2(\nka{1}-\nka{2})+\sigma_2(\tka{1}-\tka{2}) } \,d\eta.
\end{align}
Lemmas~\ref{tau-n-est} and~\ref{lem:Xk-holder} also give, for $i=1,2$,
\begin{equation}\label{eqn:N-frame-bounds}
    \norm{\nka{i}}_{\hi{\frac{1}{2}+\delta}}
    +\norm{\tka{i}}_{\hi{\frac{1}{2}+\delta}}
    \le C_{M,m},\quad
    [\nka{i}]_{\ci{0,\alpha}}
    +\norm{\p_s\nka{i}}_{\hi{-\frac{1}{2}+\delta}}
    \le C_{M,m}M,
\end{equation}
and
\begin{equation}\label{eqn:N-frame-difference}
\begin{aligned}
    &[\nka{1}-\nka{2}]_{\ci{0,\alpha}} +\norm{\p_s(\nka{1}-\nka{2})}_{\hi{-\frac{1}{2}+\delta}} \\
    &\quad +\norm{\nka{1}-\nka{2}}_{\hi{\frac{1}{2}+\delta}} +\norm{\tka{1}-\tka{2}}_{\hi{\frac{1}{2}+\delta}}
    \le C_{M,m} \norm{\kappa_1-\kappa_2}_{\hr{-\frac{1}{2}+\delta}}.
\end{aligned}
\end{equation}
Using the bounded frame extensions described above,
Lemma~\ref{point_mult_sobo} gives
\begin{equation}\label{eqn:N-density-bounds}
\begin{aligned}
    \norm{\p_s\kappa_1\nka{1}+\sigma_1\tka{1}}_{\hr{\frac{1}{2}}}
    &\le C_{M,m}
    \paren{
        \norm{\kappa_1}_{\hr{\frac{3}{2}}}
        +\norm{\sigma_1}_{\hr{\frac{1}{2}}}
    }, \\
    \norm{
        \p_s(\kappa_1-\kappa_2)\nka{1}
        +(\sigma_1-\sigma_2)\tka{1}
    }_{\hr{\frac{1}{2}}}
    &\le C_{M,m}
    \paren{
        \norm{\kappa_1-\kappa_2}_{\hr{\frac{3}{2}}}
        +\norm{\sigma_1-\sigma_2}_{\hr{\frac{1}{2}}}
    }, \\
    \norm{
        \p_s\kappa_2(\nka{1}-\nka{2})
        +\sigma_2(\tka{1}-\tka{2})
    }_{\hr{\frac{1}{2}}}
    &\le C_{M,m}
    \norm{\kappa_1-\kappa_2}_{\hr{-\frac{1}{2}+\delta}} \\
    &\quad\times
    \paren{
        \norm{\kappa_2}_{\hr{\frac{3}{2}}}
        +\norm{\sigma_2}_{\hr{\frac{1}{2}}}
    }.
\end{aligned}
\end{equation}
Applying Lemma~\ref{dsH_f_comm} to $I_1$, $I_2$, and $I_3$, and then
using~\eqref{eqn:tension-difference} and~\eqref{eqn:tension-bound}, gives
\begin{equation}
\begin{aligned}
    \norm{I_1}_{\hi{-\frac{1}{2}+\varepsilon}} &\le C_{M,m} \norm{\kappa_1-\kappa_2}_{\hr{-\frac{1}{2}+\delta}} \paren{ \norm{\kappa_1}_{\hr{\frac{3}{2}}} + \norm{\sigma_1}_{\hr{\frac{1}{2}}} }  \\
    & \le C_{M,m} \norm{\kappa_1-\kappa_2}_{\hr{-\frac{1}{2}+\delta}} \norm{\kappa_1}_{\hr{\frac{3}{2}}}
\end{aligned}
\end{equation}

\begin{equation}
\begin{aligned}
    \norm{I_2}_{\hi{-\frac{1}{2}+\varepsilon}} &\le C_{M,m}M \paren{ \norm{\kappa_1-\kappa_2}_{\hr{\frac{3}{2}}}
    +\norm{\sigma_1-\sigma_2}_{\hr{\frac{1}{2}}} } \le C_{M,m}M \norm{\kappa_1-\kappa_2}_{\hr{\frac{3}{2}}} \\
    &\quad + C_{M,m}M \norm{\kappa_1-\kappa_2}_{\hr{-\frac{1}{2}+\delta}} \paren{
    \norm{\kappa_1}_{\hr{\frac{3}{2}}} +\norm{\kappa_2}_{\hr{\frac{3}{2}}} }
\end{aligned}
\end{equation}

\begin{equation}
\begin{aligned}
    \norm{I_3}_{\hi{-\frac{1}{2}+\varepsilon}} &\le C_{M,m}M \norm{\kappa_1-\kappa_2}_{\hr{-\frac{1}{2}+\delta}} \paren{ \norm{\kappa_2}_{\hr{\frac{3}{2}}} +  \norm{\sigma_2}_{\hr{\frac{1}{2}}}  }  \\
    & \le C_{M,m}M \norm{\kappa_1-\kappa_2}_{\hr{-\frac{1}{2}+\delta}} \norm{\kappa_2}_{\hr{\frac{3}{2}}}
\end{aligned}
\end{equation}

\begin{equation}
\begin{aligned}
    \norm{\sum_{i=1}^3 I_i}_{\hi{-\frac{1}{2}+\varepsilon}} &\le C_{M,m} \norm{\kappa_1-\kappa_2}_{\hr{-\frac{1}{2}+\delta}} \paren{\norm{\kappa_1}_{\hr{\frac{3}{2}}}+\norm{\kappa_2}_{\hr{\frac{3}{2}}}} \\
    &+ C_{M,m}M \norm{\kappa_1-\kappa_2}_{\hr{\frac{3}{2}}}.
\end{aligned}
\end{equation}
For the remainder terms, we use~\eqref{eqn:N-frame-bounds}--\eqref{eqn:N-density-bounds}, the
continuity of
$\p_s:H^{\frac{1}{2}+\varepsilon}(I)\to
H^{-\frac{1}{2}+\varepsilon}(I)$, and~\eqref{eqn:Sop-bound}, \eqref{eqn:Sop-difference}, \eqref{eqn:tension-bound}, and~\eqref{eqn:tension-difference}. These bounds give
\begin{equation}
\begin{aligned}
    \norm{I_1'}_{\hi{-\frac{1}{2}+\varepsilon}} &\le C_{M,m}M \norm{\kappa_1-\kappa_2}_{\hr{-\frac{1}{2}+\delta}} \paren{ \norm{\kappa_1}_{\hr{\frac{3}{2}}} + \norm{\sigma_1}_{\hr{\frac{1}{2}}} } \\
    &\le C_{M,m}M \norm{\kappa_1-\kappa_2}_{\hr{-\frac{1}{2}+\delta}}  \norm{\kappa_1}_{\hr{\frac{3}{2}}}
\end{aligned}
\end{equation}
\begin{equation}
\begin{aligned}
    \norm{I_2'}_{\hi{-\frac{1}{2}+\varepsilon}} &\le C_{M,m} \norm{\kappa_1-\kappa_2}_{\hr{-\frac{1}{2}+\delta}}  \paren{ \norm{\kappa_1}_{\hr{\frac{3}{2}}} + \norm{\sigma_1}_{\hr{\frac{1}{2}}} } \\
    &\le C_{M,m} \norm{\kappa_1-\kappa_2}_{\hr{-\frac{1}{2}+\delta}} \norm{\kappa_1}_{\hr{\frac{3}{2}}},
\end{aligned}
\end{equation}
\begin{equation}
\begin{aligned}
    \norm{I_3'}_{\hi{-\frac{1}{2}+\varepsilon}} &\le C_{M,m}M \paren{\norm{\kappa_1-\kappa_2}_{\hr{\frac{3}{2}}}
    + \norm{\sigma_1-\sigma_2}_{\hr{\frac{1}{2}}}} \le C_{M,m}M \norm{\kappa_1-\kappa_2}_{\hr{\frac{3}{2}}} \\
    &\quad + C_{M,m}M \norm{\kappa_1-\kappa_2}_{\hr{-\frac{1}{2}+\delta}} \paren{
    \norm{\kappa_1}_{\hr{\frac{3}{2}}} +\norm{\kappa_2}_{\hr{\frac{3}{2}}} }.
\end{aligned}
\end{equation}
\begin{equation}
\begin{aligned}
    \norm{I_4'}_{\hi{-\frac{1}{2}+\varepsilon}} &\le C_{M,m}M \norm{\kappa_1-\kappa_2}_{\hr{-\frac{1}{2}+\delta}} \paren{\norm{\kappa_2}_{\hr{\frac{3}{2}}} + \norm{\sigma_2}_{\hr{\frac{1}{2}}}} \\
    &\le C_{M,m}M \norm{\kappa_1-\kappa_2}_{\hr{-\frac{1}{2}+\delta}} \norm{\kappa_2}_{\hr{\frac{3}{2}}}
\end{aligned}
\end{equation}
Hence
\begin{equation}
\begin{aligned}
    \norm{\sum_{i=1}^4 I_i'}_{\hi{-\frac{1}{2}+\varepsilon}}
    &\le C_{M,m}M
    \norm{\kappa_1-\kappa_2}_{\hr{\frac{3}{2}}} \\
    &\quad
    + C_{M,m}
    \norm{\kappa_1-\kappa_2}_{\hr{-\frac{1}{2}+\delta}}
    \paren{
        \norm{\kappa_1}_{\hr{\frac{3}{2}}}
        +\norm{\kappa_2}_{\hr{\frac{3}{2}}}
    }.
\end{aligned}
\end{equation}
Combining the commutator and remainder estimates gives
\begin{equation}
\begin{aligned}
    \norm{N(\kappa_1) - N(\kappa_2)}_{\hi{-\frac{1}{2}+\varepsilon}} &\le C_{M,m} \norm{\kappa_1-\kappa_2}_{\hr{-\frac{1}{2}+\delta}} \paren{\norm{\kappa_1}_{\hr{\frac{3}{2}}}+\norm{\kappa_2}_{\hr{\frac{3}{2}}}} \\
    & + C_{M,m}M \norm{\kappa_1-\kappa_2}_{\hr{\frac{3}{2}}}.
\end{aligned}
\end{equation}
\end{proof}

\section*{Acknowledgments}
The author thanks Yoichiro Mori for suggesting the problem and for many helpful discussions.

\section*{Conflict of interest}
The author declares no conflict of interest.

\section*{Data availability}
Data sharing is not applicable to this article as no datasets were generated or analyzed during this study.

\bibliographystyle{plain}
\bibliography{references}
\end{document}